%% file: affine_Schur.tex
\documentclass[11pt,letterpaper]{amsart}
\usepackage{amssymb,amstext,amscd,amsthm,amsfonts,mathdots,mathrsfs}
\usepackage{pdflscape}
\usepackage{hyperref}
\usepackage{pdflscape}
\usepackage{setspace}% for spacing
\usepackage[all]{xy}% for quiver
\usepackage{enumerate}% for [(1)] or [(a)]
\usepackage{indentfirst}%for indent
\usepackage{ytableau}
\usepackage{color}% for color
\usepackage{colortbl}
\usepackage{amsmath}% for the equation label
\usepackage{tikz}% for tikz quiver
\usetikzlibrary{shapes.geometric, arrows}
\usetikzlibrary{calc}
\numberwithin{equation}{section}

\usepackage[thicklines]{cancel}

\newtheorem{theorem}{Theorem}[section]
\newtheorem{lemma}[theorem]{Lemma}
\newtheorem{proposition}[theorem]{Proposition}
\newtheorem{definition}[theorem]{Definition}
\newtheorem{conjecture}[theorem]{Conjecture}
\newtheorem{corollary}[theorem]{Corollary}
\newtheorem{example}[theorem]{Example}

\theoremstyle{remark}
\newtheorem{rem}[theorem]{Remark}
\newtheorem{innercustomthm}{{\bf Theorem}}
\newenvironment{customthm}[1]
  {\renewcommand\theinnercustomthm{#1}\innercustomthm}
  {\endinnercustomthm}

\newcommand{\clr}{rgb:black,2;blue,2;red,0}

\tikzset{anchorbase/.style={baseline={([yshift=-0.5ex]current bounding box.center)}}}
\tikzset{wipe/.style={white,line width=4pt}}
\newcommand{\cred}{rgb:black,.5;blue,0;red,1.5}

\newcommand{\cgreen}{rgb:black,0;blue,2;red,2}

\DeclareFontFamily{OT1}{pzc}{}
\DeclareFontShape{OT1}{pzc}{m}{it}{ <-> s*[1.2] pzcmi7t }{}
\DeclareMathAlphabet{\mathpzc}{OT1}{pzc}{m}{it}
\newcommand{\AH}{\mathpzc{AH}}
\newcommand{\HH}{\mathpzc{H}}
\newcommand{\W}{\mathpzc{Web}}
\newcommand{\AW}{\mathpzc{Web}^\bullet}  
\newcommand{\AWC}{\W^{\bullet\prime}}
\newcommand{\ScatDJM}{\mathpzc{S}^{\text{DJM}}}
\newcommand{\Sch}{\mathpzc{Schur}}
\newcommand{\ASch}{\mathpzc{Schur}^{{\hspace{-.03in}}\bullet}} 
\newcommand{\ASchC}{\mathpzc{Schur}^{{\hspace{-.03in}}\bullet\prime}}
\newcommand{\WSch}{\mathpzc{wSchur}}
\newcommand{\Mat}{\text{Mat}}
\newcommand{\PMat}{\text{ParMat}}
\newcommand{\Par}{\text{Par}}
\newcommand{\SST}{\text{SST}}
\newcommand{\D}{\text{D}}
\newcommand{\x}{\textsc{x}}

\newcommand{\bfu}{\mathbf{u}}
\newcommand\C{\mathbb{C}}
\newcommand\Z{\mathbb{Z}}

\newcommand\N{\mathbb{N}}
\newcommand\kk{\Bbbk}
\newcommand\la{\lambda}
\newcommand{\Hom}{{\rm Hom}}
\newcommand{\End}{{\rm End}}
\newcommand{\rot}{\rotatebox[origin=c]{180}}
\newcommand{\arxiv}[1]{\href{http://arxiv.org/abs/#1}{\tt arXiv:\nolinkurl{#1}}}

\def\std{\text{ST}}
\def\s{\mathfrak s}

\def\t{\mathfrak t}
\def\Sc{{\mathcal S}^{\text{DJM}}}

\newcommand{\str}{\begin{tikzpicture}[baseline = 10pt, scale=0.4, color=\clr]
            \draw[-,thick] (0,0.5)to[out=up,in=down](0,1.7);
            \draw (0,0.2) node{$\scriptstyle 1$};
\end{tikzpicture} 
}

\newcommand{\stra}{\begin{tikzpicture}[baseline = 10pt, scale=0.4, color=\clr]
            \draw[-,line width=1.6pt] (0,0.5)to[out=up,in=down](0,1.7);
            \draw (0,0.2) node{$\scriptstyle a$};
\end{tikzpicture} 
}

\newcommand{\stru}{\begin{tikzpicture}[baseline = 10pt, scale=0.4, color=\cred]
\draw[-,line width=1pt] (0,0.4) to (0,1.7);
\draw(0,0.1) node {$\scriptstyle u$};
\end{tikzpicture}
}

\newcommand{\strui}{\begin{tikzpicture}[baseline = 10pt, scale=0.4, color=\cred]
\draw[-,line width=1pt] (0,0.5) to (0,1.7);
\draw(0,0.1) node {$\scriptstyle u_i$};
\end{tikzpicture}
}

\newcommand{\bdot}{ node[circle, draw, fill=\clr, thick, inner sep=0pt, minimum width=3pt]{}}

\newcommand{\merge}
{\begin{tikzpicture}[baseline = -.5mm,scale=.8,color=\clr]
	\draw[-,line width=1pt] (0.28,-.3) to (0.08,0.04);
	\draw[-,line width=1pt] (-0.12,-.3) to (0.08,0.04);
	\draw[-,line width=1.5pt] (0.08,.4) to (0.08,0);
        \node at (-0.22,-.4) {$\scriptstyle a$};
        \node at (0.35,-.4) {$\scriptstyle b$};\node at (0,.55){$\scriptstyle a+b$};\end{tikzpicture} }
\newcommand{\splits}
{\begin{tikzpicture}[baseline = -.5mm, scale=.8,color=\clr]
	\draw[-,line width=1.5pt] (0.08,-.3) to (0.08,0.04);
	\draw[-,line width=1pt] (0.28,.4) to (0.08,0);
	\draw[-,line width=1pt] (-0.12,.4) to (0.08,0);
        \node at (-0.22,.5) {$\scriptstyle a$};
        \node at (0.36,.5) {$\scriptstyle b$};
        \node at (0.1,-.45){$\scriptstyle a+b$};
\end{tikzpicture}}

\newcommand{\cross}{\begin{tikzpicture}[baseline=-.5mm,scale=.8,color=\clr]
	\draw[-,thick] (-0.3,-.3) to (.3,.4);
	\draw[-,thick] (0.3,-.3) to (-.3,.4);
\end{tikzpicture}}

\newcommand{\crossone}{\begin{tikzpicture}[baseline=-.5mm,scale=.8,color=\clr]
	\draw[-,thick] (-0.3,-.3) to (.3,.4);
	\draw[-,thick] (0.3,-.3) to (-.3,.4);
        \node at (0.3,-.45) {$\scriptstyle 1$};
        \node at (-0.3,-.45) {$\scriptstyle 1$};
\end{tikzpicture}}

\newcommand{\crossing}{\begin{tikzpicture}[baseline=-.5mm,scale=.8, color=\clr]
	\draw[-,thick] (-0.3,-.3) to (.3,.4);
	\draw[-,thick] (0.3,-.3) to (-.3,.4);
        \node at (0.3,-.4) {$\scriptstyle b$};
        \node at (-0.3,-.4) {$\scriptstyle a$};
         \node at (0.3,.55) {$\scriptstyle a$};
        \node at (-0.3,.55) {$\scriptstyle b$};
\end{tikzpicture}}
\newcommand{\dotgen}
{\begin{tikzpicture}[baseline = 3pt, scale=0.4, color=\clr]
\draw[-,thick] (0,0) to[out=up, in=down] (0,1.4);
\draw(0,0.6) \bdot;
\node at (0,-.3) {$\scriptstyle 1$};
\end{tikzpicture}}

\newcommand{\rightcrossing}{\begin{tikzpicture}[baseline = 10pt, scale=.8, color=\clr]
 \draw[-,line width=1.2pt] (-0.3,0) to (.3,.7);
\draw[-,line width=1pt,color=\cred] (0.3,0) to (-.3,.7);
\draw(-.3,-0.1) node{$\scriptstyle a$};
\draw (.3, -0.1) node{$\scriptstyle \red{u}$};
\end{tikzpicture}}

\newcommand{\leftcrossing}{\begin{tikzpicture}[baseline = 10pt, scale=.8, color=\clr]
 \draw[-,line width=1pt,color=\cred] (-0.3,0) to (.3,.7);
\draw[-,line width=1.2pt] (0.3,0) to (-.3,.7);
\draw(-.3,-.1) node{$\scriptstyle \red{u}$};
\draw (.3, -.1) node{$\scriptstyle a$};
\end{tikzpicture}}

\newcommand{\upliftui}{\begin{tikzpicture}[baseline = 10pt, scale=.8, color=\clr]
 \draw[-,line width=1.2pt] (-0.3,.3) to (.3,1);
\draw[-,line width=1pt,color=\cred] (0.3,.3) to (-.3,1);
\draw(0.3,.15) node {$\scriptstyle \red{u_i}$};
\draw (-.3,0.15) node{$\scriptstyle a$};
\draw (.3,1.15) node{$\scriptstyle {a}$};
\end{tikzpicture}}

\newcommand{\upliftuip}{\begin{tikzpicture}[baseline = 10pt, scale=.8, color=\clr]
 \draw[-,line width=1.2pt] (-0.3,.3) to (.3,1);
\draw[-,line width=1pt,color=\cred] (0.3,.3) to (-.3,1);
\draw(0.35,.15) node {$\scriptstyle \red{u_{i+1}}$};
\draw (-.3,0.15) node{$\scriptstyle a$};
\draw (.3,1.15) node{$\scriptstyle {a}$};
\end{tikzpicture}}

\newcommand{\downliftui}{\begin{tikzpicture}[baseline = 10pt, scale=.8, color=\clr]
 \draw[-,line width=1pt,color=\cred] (-0.3,.3) to (.3,1);
 \draw(-.3,.15) node {$\scriptstyle \red{u_i}$};
\draw[-,line width=1.2pt] (0.3,.3) to (-.3,1);
\draw (-.3,1.15) node{$\scriptstyle a$};
\draw (.3,0.15) node{$\scriptstyle {a}$};
\end{tikzpicture}}

\newcommand{\downliftuip}{\begin{tikzpicture}[baseline = 10pt, scale=.8, color=\clr]
 \draw[-,line width=1pt,color=\cred] (-0.3,.3) to (.3,1);
 \draw(-.35,.15) node {$\scriptstyle \red{u_{i+1}}$};
\draw[-,line width=1.2pt] (0.3,.3) to (-.3,1);
\draw (-.3,1.15) node{$\scriptstyle a$};
\draw (.3,0.15) node{$\scriptstyle {a}$};
\end{tikzpicture}}

\newcommand{\wkdota}{\begin{tikzpicture}[baseline = 3pt, scale=0.4, color=\clr]
\draw[-,line width=1.2pt] (0,0) to[out=up, in=down] (0,1.4);
\draw(0,0.6) \bdot; 
\draw (0.7,0.6) node {$\scriptstyle \omega_r$};
\node at (0,-.3) {$\scriptstyle a$};
\end{tikzpicture} }

\newcommand{\wkdotaa}{\begin{tikzpicture}[baseline = 3pt, scale=0.4, color=\clr]
\draw[-,line width=1.2pt] (0,0) to[out=up, in=down] (0,1.4);
\draw(0,0.6) \bdot; 
\draw (0.7,0.6) node {$\scriptstyle \omega_a$};
\node at (0,-.3) {$\scriptstyle a$};
\end{tikzpicture} }

\newcommand{\wkdotr}{
\begin{tikzpicture}[baseline = 3pt, scale=0.4, color=\clr]
\draw[-,line width=1.2pt] (0,0) to[out=up, in=down] (0,1.4);
\draw(0,0.6) \bdot; 
\draw (0.7,0.6) node {$\scriptstyle \omega_r$};
\node at (0,-.3) {$\scriptstyle r$};
\end{tikzpicture}
}

\newcommand{\wdotgen}{
\begin{tikzpicture}[baseline = 3pt, scale=0.4, color=\clr]
\draw[-,line width=1.2pt] (0,0) to[out=up, in=down] (0,1.4);
\draw(0,0.6) \bdot; 
\draw (0.7,0.6) node {$\scriptstyle \omega_1$};
\node at (0,-.3) {$\scriptstyle 1$};
\end{tikzpicture}}

\newcommand{\blue}[1]{{\color{blue}#1}}
\newcommand{\red}[1]{{\color{red}#1}}

\begin{document}
\setlength{\baselineskip}{17pt}
\title{Affine and cyclotomic Schur categories}
%\date{\today}

\author{Linliang Song}
\address{School of Mathematical Science, Key Laboratory of Intelligent Computing and Applications(Ministry of Education), Tongji University, Shanghai, 200092, China}\email{llsong@tongji.edu.cn}

 \author{Weiqiang Wang}
 \address{Department of Mathematics, University of Virginia,
Charlottesville, VA 22903, USA}\email{ww9c@virginia.edu}



\begin{abstract}
Using the affine web category introduced in a prequel as a building block, we formulate a diagrammatic $\Bbbk$-linear monoidal category, {\em the affine Schur category}, for any commutative ring $\Bbbk$. We then formulate the {\em cyclotomic Schur categories}. Integral bases consisting of elementary diagrams are obtained for affine and cyclotomic Schur categories. A second diagrammatic basis, called a double SST basis, for any such cyclotomic Schur category is also established, leading to a higher level RSK correspondence. We show that the path algebras with the double SST bases are isomorphic to degenerate cyclotomic Schur algebras of Dipper-James-Mathas with their cellular bases, providing a first diagrammatic presentation of the latter. The presentations for the affine and cyclotomic Schur categories are simplified for $\Bbbk=\mathbb C$.  
\end{abstract}

\maketitle

\setcounter{tocdepth}{1}
\tableofcontents

%=====================
% Main body
%=====================
%
\input{section1_introduction}

%\part{Affine and cyclotomic web categories}
%
\input{Section2_affineweb}

\input{Section3_affineschur}

\input{Section4_cycSchurCat}

\input{Section5_basescycSchur}
\input{Section6_zerochar}

%===========
% References
%===========

\bibliographystyle{alpha}
\bibliography{affineSchur}

\end{document}

%% file: Section1_Introduction.tex
\section{Introduction}

\subsection{Schur algebras}

Schur algebras associated to Hecke algebras of type $A$ are quasi-hereditary algebras in the sense of \cite{CPS} which enjoy favorable homological properties. Dipper-James-Mathas \cite{DJM98B} and Du-Scott \cite{DS00} independently generalized the construction associated to Hecke algebras of type $B$. These constructions were then generalized in \cite{DJM98} to cyclotomic $q$-Schur algebras, which are by definition the endomorphism algebras of a direct sum of permutation modules over cyclotomic Hecke algebras. Cyclotomic $q$-Schur algebras are quasi-hereditary and admit cellular bases. The discussions above apply to both the degenerate and quantum cases. 

Representations of cyclotomic $q$-Schur algebras have been well studied and they categorify the higher level Fock spaces (see \cite{Y06, R08} and \cite{RSVV, Los16, Web17}); many of these works are in a more general setup of rational Cherednik algebras.

\subsection{Affine and cyclotomic webs}

Throughout this paper, all categories and functors will be $\kk$-linear for an arbitrary commutative ring $\kk$ with $1$. Brundan, Entova-Aizenbud, Etingof and Ostrik \cite{BEEO} formulated a strict monoidal category $\W$,  called a (polynomial) web category $\W$. Th category $\W$ is a simpler $\mathfrak{gl}_\infty$-version of Cautis-Kamnitzer-Morrison $\mathfrak{sl}_N$-web category \cite{CKM}. They further showed that $\W$ is isomorphic to a Schur category $\Sch$ whose path algebras are Schur algebras over $\kk$. This in particular provides a diagrammatic presentation for the Schur algebras. An integral (double ladder) basis for the $\mathfrak{sl}_N$-web category was constructed by Elias \cite{El15} and two integral bases for $\W$ were given in \cite{BEEO}. 

A monoidal category, called the affine web category and denoted $\AW$, was introduced in our prequel \cite{SWweb} and in \cite{DKM25}. The category  $\AW$ admits certain quotient categories called cyclotomic web categories and denoted $\W_\bfu$, for $\bfu \in \kk^\ell$ and $\ell \ge 1$. It was shown \cite{SWweb} that $\W_\bfu$ is isomorphic to the $W$-Schur category $\WSch_{-\bfu}$ whose path algebras are idempotent subalgebras of $W$-Schur algebras which arise from Brundan-Kleshchev's higher level Schur duality for finite W-algebras of type $A$ \cite{BK08}. Thanks to such an isomorphism, $\W_\bfu$ has a well-developed representation theory over $\kk=\C$. 

\subsection{Goal}

In this paper we introduce a new $\kk$-linear strict diagrammatic monoidal category, {\em the affine Schur category} which will be denoted $\ASch$, building on the groundwork laid by the prequel \cite{SWweb}. We further formulate cyclotomic Schur categories $\Sch_\bfu$ as distinguished quotient categories of $\ASch$, for $\bfu \in \kk^\ell$. They fit together with the affine/cyclotomic web categories into the following commutative diagram of categories (with all vertical arrows as inclusions of full subcategories):  
\begin{align}  \label{diag:HWS}
\xymatrix{
\HH \ar[r]   \ar[d] 
&\AH \ar[r]  
 \ar[d]  & \HH_{\bfu}
 \ar[d]  
  \\
\W \ar[r]  \ar[d]^{\cong}
& \AW \ar[r]  \ar@[red][d]  
& \W_\bfu   \ar@[red][d] 
\\
\Sch \ar[r]
 & \red{\ASch} \ar@[red][r] & \red{\Sch_\bfu}
}
\end{align}
The top row of the diagram \eqref{diag:HWS} consists of the well-known (degenerate) Hecke category  $\HH$  (with $\cross$ as a generating morphism), the affine Hecke category $\AH$ (with $\cross, \dotgen$ as generating morphisms) and the cyclotomic Hecke categories $\HH_\bfu$. As explained in \cite{SWweb}, the vertical arrows from the top row to the second row are fully faithful functors. Building on the intermediate constructions $\AW \rightarrow \W_\bfu$ in the prequel, we construct in this paper the categories $\ASch$ and $\Sch_\bfu$ and hence complete the right bottom square of the diagram \eqref{diag:HWS}, where the vertical arrows are shown to be fully faithful functors as well. 

We show that the path algebras of $\ASch$ (respectively, $\Sch_\bfu$) are isomorphic to degenerate versions of higher level affine $q$-Schur algebras in \cite{MS21} (and respectively, cyclotomic $q$-Schur algebras in \cite{DJM98}), providing diagrammatic presentations for these algebras. Alternatively, one defines a category $\ScatDJM_\bfu$ out of the cyclotomic Schur algebras, and we establish an isomorphism of categories $\Sch_\bfu \cong \ScatDJM_\bfu$. 
%This fits well with the isomorphism $\W_\bfu \cong \WSch_{-\bfu}$ (cf. \cite{SWweb}) into the following commutative diagram (with vertical arrows being fully faithful):
%\begin{align}  \label{diag:WScyc}
%\xymatrix{
%\W_\bfu \ar[r]^{\cong}  \ar[d]  & \WSch_{-\bfu}   \ar[d] 
%\\
%\Sch_\bfu \ar[r]^{\cong} & \ScatDJM_\bfu
%}
%\end{align}

This paper will serve as a foundation on a research program aiming at formulating new monoidal categories of various types and their cyclotomic quotients, which admit rich representation theories and natural connections to categorification. See \S\ref{subsec:future} for more details. 

\subsection{A basis theorem for $\ASch$}

By Definition~\ref{def-affine-Schur}, the category $\ASch$ extends the generating objects $\stra$ ($a\in \Z_{\ge 1}$) for $\AW$ to include new generating objects $\stru$ ($u\in \kk$); it extends the generating morphisms $\merge, \splits, \crossing$ and $\wkdotaa$,
for $a,b \in \Z_{\ge 1}$, for $\AW$ to include two new types of generating morphisms
\begin{equation}
\label{rightleftcrossinggen1}
\begin{tikzpicture}[baseline = 10pt, scale=.8, color=\clr]
 \draw[-,line width=1.5pt] (-0.3,.3) to (.3,1);
\draw[-,line width=1pt,color=\cred] (0.3,.3) to (-.3,1);
\draw(-.3,0.15) node{$\scriptstyle a$};
\draw (.3, 0.15) node{$\scriptstyle \red{u}$};
\end{tikzpicture} \;\; (\text{traverse-up}), 
                \qquad
\begin{tikzpicture}[baseline = 10pt, scale=.8, color=\clr]
 \draw[-,line width=1pt,color=\cred] (-0.3,.3) to (.3,1);
\draw[-,line width=1.5pt] (0.3,.3) to (-.3,1);
\draw(-.3,0.15) node{$\scriptstyle \red{u}$};
\draw (.3, 0.15) node{$\scriptstyle a$};
\end{tikzpicture}  \;\; (\text{traverse-down}), 
\end{equation}
subject to relations in \eqref{webassoc}--\eqref{intergralballon} and \eqref{adaptorR}--\eqref{adaptermovemerge}. Diagrammatrics similar to \eqref{rightleftcrossinggen1} appeared in Webster's work \cite{WebMemoirs}. 

The monoidal category $\ASch$ over $\kk =\C$ admits a much simpler  presentation given in Section~ \ref{sec:C} with rather simple relations, that is, the more involved relations in Definition~\ref{def-affine-Schur} can be derived from these simple relations over $\C$. 

Our first main result is the construction of a faithful polynomial representation of $\ASch$. 

 \begin{customthm} {\bf A}   [Theorem \ref{thm:polyRep}, Remark \ref{rem:faithfulofF}]
 \label{th:A}
There is a $\kk$-linear monoidal 
faithful functor $\mathcal F$ from $\ASch$ to the category  $\mathpzc{Vec}_\kk$ of free $\kk$-modules, which sends objects $a\in \N$ to $ \text{Sym}_a$, $u\in \kk$ to $ \kk$, and sends the generating morphisms to the linear maps in the same symbols defined in \S\ref{polynomilrep}.
 \end{customthm} 

Recall \cite{SWweb} that the Hom-space $\Hom_{\AW}(\overline{\mu},\overline{\la})$,  for strict compositions $\overline{\mu}, \overline{\la}$, admits an elementary diagram basis $\PMat_{\overline{\la}, \overline{\mu}}$, which can also be identified with partition-enhanced $\N$-matrices with $\overline{\la}, \overline{\mu}$ as row/column sum vectors. An elementary diagram is a reduced chicken foot diagram connecting $\overline{\mu},\overline{\la}$ (cf. \cite{BEEO}) with decoration by ``elementary dot packets" parametrized by partitions. 

An arbitrary object in $\ASch$ is of the form $\lambda^{(0)} \red{u_1} \lambda^{(1)} \red{u_2}  \ldots \red{u_{\ell}} \lambda^{(\ell)}$. When fixing $\bfu =(u_1, \ldots, u_\ell)$ in the background, such an object is simply identified with a strict multicomposition $\la=(\lambda^{(0)} \red{,} \lambda^{(1)} \red{,}  \ldots  \red{,} \lambda^{(\ell)})$. There is a natural forgetful map sending $\la$ to a strict composition $\overline{\la}$ (by ignoring the color of the $\ell$ commas and empty components). 
By Lemma \ref{lem:Homzero}, the Hom-space $\Hom_{\ASch}(\mu,\lambda)$ is zero unless $\lambda,\mu \in \Lambda_{\text{st}}^{1+\ell}(m)$, for some $m$ and $\ell$. In this case, we define $\PMat_{\lambda,\mu}$ as in \eqref{PMatblock}, which can be viewed as a $(1+\ell) \times (1+\ell)$-block matrix generalization (or a multicomposition generalization) of $\PMat_{\overline{\la}, \overline{\mu}}$ above. Diagrammatically, $\PMat_{\lambda,\mu}$ can be represented by ornamented elementary chicken foot diagrams by making (and fixing) a choice of adding $\ell$ red strands connecting the corresponding  $\strui$ (i.e., the red commas above) in $\la$ and $\mu$ in a diagram of $\PMat_{\overline{\la}, \overline{\mu}}$. 

 \begin{customthm} {\bf B} [Theorem~ \ref{thm:basisASchur}]
  \label{th:B}
$\PMat_{\lambda,\mu}$ forms a basis for the Hom-space $\Hom_{\ASch}(\mu,\lambda)$, for any $\lambda,\mu \in \Lambda_{\text{st}}^{1+\ell}(m)$. 
 \end{customthm} 
 
To show that $\PMat_{\lambda,\mu}$ is a spanning set for $\Hom_{\ASch}(\mu,\lambda)$, it suffices to show that multiplying an arbitrary generating morphism with any diagram in $\PMat_{\lambda,\mu}$ can be rewritten (by using the defining relations for $\ASch$) as a linear combination of diagrams in $\PMat_{\lambda,\mu}$. The proof of linear independence of $\PMat_{\lambda,\mu}$ follows by showing that they act as linearly independent operators on the representation of $\ASch$ given in Theorem~\ref{th:A}.

\subsection{Bases for cyclotomic Schur categories $\Sch_\bfu$}

Denote by $\SST(\lambda,\mu)$ the set of semistandard $\la$-tableaux of type $\mu$, for $\la \in \Par^\ell(m)$ and $\mu\in \Lambda_{\text{st}}^\ell(m)$. They can be represented by reduced chicken foot diagrams from $\la$ to $\mu$; see \S\ref{subsec:SST}. Recall that $\div$ denotes the symmetry by  reflection across the  horizontal axis \eqref{anti-autocyc}. 

 \begin{customthm} {\bf C}   [Theorem~\ref{thm:SSTbasis}(1)]
 \label{th:C}
For any $\mu,\nu\in \Lambda_{\text{st}}^\ell(m)$, the Hom-space $\Hom_{\Sch_{\mathbf u}}(\mu,\nu)$ has a double SST basis 
\begin{align}  \label{eq:SST2}
\bigcup_{\lambda\in \Par^\ell(m)}
\big\{ [\mathbf T]\circ [\mathbf S]^\div \mid \mathbf T\in \SST(\lambda,\nu), \mathbf S\in \SST(\lambda,\mu) \big\}. 
\end{align}
 \end{customthm} 
In case when $\ell=1$, the basis in Theorem~\ref{th:C} appeared in \cite{BEEO} as a reformulation of Green's codeterminant basis \cite{Gre93} for Schur algebras. A basis over a commutative ring $\kk$ for the $\mathfrak{sl}_n$-web category of \cite{CKM} was constructed earlier by Elias \cite{El15}. 

One can form a category $\ScatDJM_\bfu$ with the degenerate cyclotomic Schur algebras as path algebras. We further specify certain distinguished morphisms in  $\ScatDJM_\bfu$ and represent them in the same symbols as for the generating morphisms in $\Sch_\bfu$; see \S\ref{subsec:functorG}. 

 \begin{customthm} {\bf D} [Theorem~\ref{thm:G}, Proposition~\ref{pro:mapsdjm}, Theorem~\ref{thm:SSTbasis}(3)]
   \label{th:D}
 There is an isomorphism of categories $\mathcal G: \Sch_{\mathbf u}\rightarrow \ScatDJM_{\mathbf u}$, which sends an object $\mu$ to $\mu$ and morphisms to morphisms in the same symbols. 
The induced algebra isomorphism $\Sch_\mathbf u(m)\cong \Sc_{m,\mathbf u}$  matches the double SST basis with the cellular basis \eqref{basisofdjm} for $\Sc_{m,\mathbf u}$.  
 \end{customthm} 
In the theorem above and also in \eqref{Schurm}, we have denoted the path algebras of $\Sch_\bfu$, $\Sch_\bfu(m)= \oplus_{\mu,\nu\in \Lambda_{\text{st}}^\ell(m)}\Hom_{\Sch_{\bfu}}(\mu,\nu)$, for $m\ge 1$. Hence, we have obtained a diagrammatic presentation of the cyclotomic Schur algebra $\Sc_{m,\mathbf u}$. In the case when $\ell=1$, our isomorphism in Theorem~\ref{th:D}  (with a new proof over any field $\kk$) specializes to the isomorphism $\W \cong \Sch$ due to \cite{BEEO}. 
 
 The proofs of Theorems \ref{th:C} and \ref{th:D} go hand in hand. One first proves that \eqref{eq:SST2} forms a spanning set for $\Hom_{\Sch_{\mathbf u}}(\mu,\nu)$. Then one constructs a natural functor $\mathcal G: \Sch_{\mathbf u}\rightarrow \Sc_{\mathbf u}$ and shows that the spanning set \eqref{eq:SST2} is mapping by $\mathcal G$ to the cellular basis for $\Hom_{\Sc_{\mathbf u}}(\mu,\nu)$ in \cite{DJM98}. The theorems follow from these. 

 Recall from \eqref{def:barlambdaum}
the subset $\Lambda_{\text{st}}^{\emptyset,\ell}(m)$ of $\Lambda_{\text{st}}^{1+\ell}(m)$ consisting of strict  multicompositions of the form $(\emptyset, \la^{(1)}, \ldots, \la^{(\ell)})$. This can be clearly identified with $\Lambda_{\text{st}}^{\ell}(m)$, except that the diagrams for morphisms between objects in $\Lambda_{\text{st}}^{\emptyset,\ell}(m)$ start with a red strand on the left. 
Denote by $\PMat_{\lambda,\mu}^\flat$ from \eqref{PM0} the set of bounded-partition-enhanced $\ell\times \ell$-block $\N$-matrices; diagrammatically these block matrices represent elementary diagrams with bounded elementary dot packets.

 \begin{customthm} {\bf E} [Theorem~\ref{thm:basis2CycSchur}]
   \label{th:E}
    Suppose $\mu,\nu\in \Lambda_{\text{st}}^\ell(m)$. Then $\PMat_{\nu,\mu}^\flat$ forms a basis for $\Hom_{\Sch_{\mathbf u}}(\mu,\nu)$.
 \end{customthm} 
In case when $\ell=1$, the basis in Theorem~\ref{th:E} is known as the reduced chicken foot diagram basis for $\Sch$ \cite{BEEO}.

As a consequence of Theorem~\ref{th:E}, one obtains (see \S\ref{subsec:proofbasis}) a natural functor $\W_\bfu \rightarrow \Sch_\bfu$, which is clearly fully faithful by comparing a spanning set $\PMat_{\la,\mu}^\ell$ of the Hom-space $\Hom_{\W_\bfu} (\mu,\la)$ in $\W_\bfu$ (see \cite[Lemma~4.2]{SWweb}) with the basis of the corresponding Hom-space in $\Sch_\bfu$ in Theorem~\ref{th:E}. This completes the proof of the basis theorem for $\W_\bfu$ in \cite[Theorem~C or 4.3]{SWweb}. 

Theorems \ref{th:C} and \ref{th:E} naturally lead to a conjecture on higher level RSK correspondence between the sets $\PMat^\flat_{\nu,\mu}$ and $\SST^2_{\nu,\mu}$; see Conjecture~\ref{conjRSK}.

%\subsection{Higher level RSK correspondence}

%It follows by Theorems \ref{th:C} and \ref{th:E} that the cardinalities of $\PMat^\flat_{\nu,\mu}$ and $\SST^2_{\nu,\mu}$ are equal. In case when $\ell=1$, $\PMat_{\nu,\mu}^\flat$ can be identified with non-negative integer matrices with row/column sum vectors $\nu, \mu$, and the celebrated RSK correspondence provides a bijection between $\PMat_{\nu,\mu}^\flat$ and pairs of semi-standard tableaux of same shape with contents $\nu,\mu$. 

% \begin{customconj} {\bf F} [Higher level RSK correspondence, Conjecture~\ref{conjRSK}]
%   \label{conj:F}
%    There exists a combinatorial level $\ell$ one-to-one correspondence $\psi_\ell:\PMat^\flat_{\nu,\mu} \longrightarrow \SST^2_{\nu,\mu}$, for all $\mu,\nu\in \Lambda_{\text{st}}^\ell(m)$, so that the level $1$ map $\psi_1$ is the classic RSK correspondence and the level $\ell$ map $\psi_\ell$ is compatible with $\psi_{\ell-1}$. 
% \end{customconj} 

\subsection{Past and future works}
 \label{subsec:future}

A quantum version of the main constructions in \cite{SWweb} and this paper will appear in a sequel \cite{SSW25}. In particular, we shall show that cyclotomic $q$-Schur categories provide diagrammatic presentations for cyclotomic $q$-Schur algebras in \cite{DJM98}.  

Wada \cite{Wad11} gave a presentation for cyclotomic $q$-Schur algebras based on some non-standard quantum group like algebras, which looks quite involved. We are not aware of any connection with our simple diagrammatic construction.
 
Maksimau and Stroppel \cite{MS21} constructed a monoidal category called higher level affine Hecke category with path algebras called higher level affine Hecke algebras; this category appears as a monoidal subcategory of a $q$-version of $\ASch$ (see \cite{SSW25}) with generating objects being a thin strand $\str$ and $\stru$ ($u\in \kk$) 
and generating morphisms $\crossone$, $\dotgen$, and thin mixed color crossings (i.e. \eqref{rightleftcrossinggen1} with $a=1$). Then they followed \cite{DJM98} to formulate higher level affine $q$-Schur algebras as endomorphism algebras of a sum of permutation modules over higher level affine Hecke algebras. The path algebra of $\ASch$ can be viewed as a degenerate version of higher level $q$-Schur algebras. 
A representation and a basis of a higher level affine $q$-Schur algebra were constructed in \cite{MS21}. Theorems {\bf A} and {\bf B} can be viewed as a monoidal categorical enhancement of (degenerate versions of) these constructions. Without connecting to webs, the approach {\em loc. cit.} does not lead to a presentation for higher level affine Schur algebras or cyclotomic $q$-Schur algebras even though the abstract therein seems to give a reader an opposite impression. 

The constructions in \cite{SWweb} and this paper of (degenerate) affine/cyclotomic web and Schur categories admit several (quantum, graded, beyond type A) ramifications and generalizations; see \cite[\S1.7]{SWweb} for more discussions and references at the level of web categories. 

To our best knowledge, there was little hint of interest in the literature in developing the representation theory of web categories. As a point of departure, a main goal for our constructions of the cyclotomic web/Schur categories is to produce diagrammatically presented algebras with rich representation theories and natural categorical actions. As shown in \cite{SWweb} and this paper, for type $A$, the resulting algebras were well-known important algebras with well-developed representation theory and their diagrammatic presentations are new here; for quantum type $A$, the resulting algebras from cyclotomic web category will be new. 

In the sequels we shall generalize the main constructions in \cite{SWweb} and this paper in several directions, and the strategies developed in these two papers will continue to be applicable. 
Within type $A$, we can extend to affine rational (i.e., oriented) Schur categories and their cyclotomic quotients. As further generalizations, we shall construct degenerate/quantum affine Schur categories of types $Sp/O, Q,$ and $P$. The corresponding cyclotomic web/Schur categories will afford rich representation theories which are intimately related to the representations of cyclotomic web/Schur categories of type $A$. The representations of cyclotomic Schur categories of types $Sp/O$ shall further be connected to $\imath$canonical bases via $\imath$categorification.
%(generalizing \cite{R08, RSVV, Los16, Web17}). 
In yet another direction, we shall develop a graded version, leading to a diagrammatic presentation of quiver Schur algebras and a thick calculus for 2-Kac-Moody categories. % \cite{KL3, Rou}. 

\subsection{Organization}
The paper is organized as follows. 

In Section~\ref{sec:AWeb}, the affine web category $\AW$ and its elementary diagram basis from \cite{SWweb} are reviewed. In Section~\ref{sec:ASchur}, we define the affine Schur category $\ASch$, construct its faithful polynomial representation (Theorem~\ref{th:A}), and use it to establish the elementary diagram basis for $\ASch$ (Theorem~\ref{th:B}). It follows that $\AW$ is a full subcategory of $\ASch$ with compatible elementary diagram bases. 
In Section~\ref{sec:cycSchur}, we formulate the cyclotomic Schur categories $\Sch_\bfu$. We establish a functor $\mathcal G: \Sch_\bfu \rightarrow \ScatDJM_\bfu$, which maps a double SST set to a cellular basis for $\Sc_\bfu$. 

In Section~\ref{sec:basis-cycschur}, we show the double SST set is a spanning set for the Hom-spaces in $\Sch_\bfu$ and hence a basis (Theorem~\ref{th:C}) when combined with results on the functor $\mathcal G$ established from Section~\ref{sec:cycSchur}. This also allows us to conclude that $\mathcal G$ is an isomorphism (Theorem~\ref{th:D}), leading to a diagrammatic presentation of the cyclotomic Schur algebras. 

In addition, we establish in  Section~\ref{sec:basis-cycschur} an elementary diagram basis for $\Sch_\bfu$ (Theorem~\ref{th:E}) by proving that it is a spanning set with the same cardinality as the double SST basis. It follows by the compatibility of the elementary diagrams that $\AW$ is a full subcategory of $\ASch$. A quick proof of the elementary diagram basis result for the cyclotomic web categories $\W_\bfu$ is also given (as promised in \cite{SWweb}). 
In Section~\ref{sec:C}, we give a simplified presentation of $\ASch$ over $\kk=\C$.

%% file: Section2_affineweb.tex
\section{The affine web category}
 \label{sec:AWeb}

In this section we recall some basic results for the affine web category from \cite{SWweb} (and also \cite{CKM, BEEO}).

\subsection{Definition of the affine web category} 

Throughout this paper, let $\kk$ be a commutative ring with $1$. All categories and functors will be $\kk$-linear without further mention.

The affine web category $\AW$ is a strict monoidal category with generating  objects $a\in \mathbb Z_{\ge1}$ drawn as a vertical strand labeled by $a$
\[
\begin{tikzpicture}[baseline = 10pt, scale=0.4, color=\clr]
            \draw[-,line width=1.5pt] (0,0.5)to[out=up,in=down](0,2.2);
            \draw (0,0.1) node{$ a$};
\end{tikzpicture} \; .
\]
We will use the usual string calculus for a strict monoidal category.

\begin{definition}
\label{def-affine-web}
\cite{SWweb}
The affine web category $\AW$ is the strict monoidal category with generating  objects $a\in \mathbb Z_{\ge1}$ and generating  morphisms  
\begin{align}
\label{merge+split+crossing}
\begin{tikzpicture}[baseline = -.5mm,color=\clr]
	\draw[-,line width=1pt] (0.28,-.3) to (0.08,0.04);
	\draw[-,line width=1pt] (-0.12,-.3) to (0.08,0.04);
	\draw[-,line width=1.5pt] (0.08,.4) to (0.08,0);
        \node at (-0.22,-.4) {$\scriptstyle a$};
        \node at (0.35,-.4) {$\scriptstyle b$};\node at (0,.55){$\scriptstyle a+b$};\end{tikzpicture} 
&:(a,b) \rightarrow (a+b),&
\begin{tikzpicture}[baseline = -.5mm,color=\clr]
	\draw[-,line width=1.5pt] (0.08,-.3) to (0.08,0.04);
	\draw[-,line width=1pt] (0.28,.4) to (0.08,0);
	\draw[-,line width=1pt] (-0.12,.4) to (0.08,0);
        \node at (-0.22,.5) {$\scriptstyle a$};
        \node at (0.36,.5) {$\scriptstyle b$};
        \node at (0.1,-.45){$\scriptstyle a+b$};
\end{tikzpicture}
&:(a+b)\rightarrow (a,b),&
\begin{tikzpicture}[baseline=-.5mm,color=\clr]
	\draw[-,line width=1pt] (-0.3,-.3) to (.3,.4);
	\draw[-,line width=1pt] (0.3,-.3) to (-.3,.4);
        \node at (0.3,-.45) {$\scriptstyle b$};
        \node at (-0.3,-.45) {$\scriptstyle a$};
         \node at (0.3,.55) {$\scriptstyle a$};
        \node at (-0.3,.55) {$\scriptstyle b$};
\end{tikzpicture}
&:(a,b) \rightarrow (b,a),
\end{align}
(called the merges, splits and crossings, respectively), and 
\begin{equation}
\label{dotgenerator}
\begin{tikzpicture}[baseline = 3pt, scale=0.5, color=\clr]
\draw[-,line width=1.5pt] (0,0) to[out=up, in=down] (0,1.4);
\draw(0,0.6) \bdot; 
\draw (0.7,0.6) node {$\scriptstyle \omega_a$};
\node at (0,-.3) {$\scriptstyle a$};
\end{tikzpicture}  
\;, 
\end{equation}
for $a,b \in \Z_{\ge 0}$, 
subject to the following relations \eqref{webassoc}--\eqref{dotmovesplitss}, for $a,b,c,d \in \Z_{\ge 1}$ with $d-a=c-b$:
\begin{align}
\label{webassoc}
\begin{tikzpicture}[baseline = 0,color=\clr]
	\draw[-,thick] (0.35,-.3) to (0.08,0.14);
	\draw[-,thick] (0.1,-.3) to (-0.04,-0.06);
	\draw[-,line width=1pt] (0.085,.14) to (-0.035,-0.06);
	\draw[-,thick] (-0.2,-.3) to (0.07,0.14);
	\draw[-,line width=1.5pt] (0.08,.45) to (0.08,.1);
        \node at (0.45,-.41) {$\scriptstyle c$};
        \node at (0.07,-.4) {$\scriptstyle b$};
        \node at (-0.28,-.41) {$\scriptstyle a$};
\end{tikzpicture}
&=
\begin{tikzpicture}[baseline = 0, color=\clr]
	\draw[-,thick] (0.36,-.3) to (0.09,0.14);
	\draw[-,thick] (0.06,-.3) to (0.2,-.05);
	\draw[-,line width=1pt] (0.07,.14) to (0.19,-.06);
	\draw[-,thick] (-0.19,-.3) to (0.08,0.14);
	\draw[-,line width=1.5pt] (0.08,.45) to (0.08,.1);
        \node at (0.45,-.41) {$\scriptstyle c$};
        \node at (0.07,-.4) {$\scriptstyle b$};
        \node at (-0.28,-.41) {$\scriptstyle a$};
\end{tikzpicture}\:,
\qquad
\begin{tikzpicture}[baseline = -1mm, color=\clr]
	\draw[-,thick] (0.35,.3) to (0.08,-0.14);
	\draw[-,thick] (0.1,.3) to (-0.04,0.06);
	\draw[-,line width=1pt] (0.085,-.14) to (-0.035,0.06);
	\draw[-,thick] (-0.2,.3) to (0.07,-0.14);
	\draw[-,line width=1.5pt] (0.08,-.45) to (0.08,-.1);
        \node at (0.45,.4) {$\scriptstyle c$};
        \node at (0.07,.42) {$\scriptstyle b$};
        \node at (-0.28,.4) {$\scriptstyle a$};
\end{tikzpicture}
=\begin{tikzpicture}[baseline = -1mm, color=\clr]
	\draw[-,thick] (0.36,.3) to (0.09,-0.14);
	\draw[-,thick] (0.06,.3) to (0.2,.05);
	\draw[-,line width=1pt] (0.07,-.14) to (0.19,.06);
	\draw[-,thick] (-0.19,.3) to (0.08,-0.14);
	\draw[-,line width=1.5pt] (0.08,-.45) to (0.08,-.1);
        \node at (0.45,.4) {$\scriptstyle c$};
        \node at (0.07,.42) {$\scriptstyle b$};
        \node at (-0.28,.4) {$\scriptstyle a$};
\end{tikzpicture}\:,
\\
\label{mergesplit}
\begin{tikzpicture}[baseline = 7.5pt,scale=.8, color=\clr]
	\draw[-,line width=1.2pt] (0,0) to (.275,.3) to (.275,.7) to (0,1);
	\draw[-,line width=1.2pt] (.6,0) to (.315,.3) to (.315,.7) to (.6,1);
        \node at (0,1.13) {$\scriptstyle b$};
        \node at (0.63,1.13) {$\scriptstyle d$};
        \node at (0,-.1) {$\scriptstyle a$};
        \node at (0.63,-.1) {$\scriptstyle c$};
\end{tikzpicture}
&=
\sum_{\substack{0 \leq s \leq \min(a,b)\\0 \leq t \leq \min(c,d)\\t-s=d-a}}
\begin{tikzpicture}[baseline = 7.5pt,scale=.8, color=\clr]
	\draw[-,thick] (0.58,0) to (0.58,.2) to (.02,.8) to (.02,1);
	\draw[-,thick] (0.02,0) to (0.02,.2) to (.58,.8) to (.58,1);
	\draw[-,thick] (0,0) to (0,1);
	\draw[-,line width=1pt] (0.61,0) to (0.61,1);
        \node at (0,1.13) {$\scriptstyle b$};
        \node at (0.6,1.13) {$\scriptstyle d$};
        \node at (0,-.1) {$\scriptstyle a$};
        \node at (0.6,-.1) {$\scriptstyle c$};
        \node at (-0.1,.5) {$\scriptstyle s$};
        \node at (0.77,.5) {$\scriptstyle t$};
\end{tikzpicture},
\\
\label{splitmerge}
\begin{tikzpicture}[baseline = -1mm,scale=.8,color=\clr]
\draw[-,line width=1.5pt] (0.08,-.8) to (0.08,-.5);
\draw[-,line width=1.5pt] (0.08,.3) to (0.08,.6);
\draw[-,thick] (0.1,-.51) to [out=45,in=-45] (0.1,.31);
\draw[-,thick] (0.06,-.51) to [out=135,in=-135] (0.06,.31);
\node at (-.33,-.05) {$\scriptstyle a$};
\node at (.45,-.05) {$\scriptstyle b$};
\end{tikzpicture}
&= 
\binom{a+b}{a}\:
\begin{tikzpicture}[baseline = -1mm,scale=.7,color=\clr]
	\draw[-,line width=1.5pt] (0.08,-.8) to (0.08,.6);
        \node at (.08,-1) {$\scriptstyle a+b$};
\end{tikzpicture} ,
\\
\label{dotmovecrossing}
\begin{tikzpicture}[baseline = 7.5pt, scale=0.35, color=\clr]
\draw[-,line width=1.2pt] (0,-.2) to  (1,2.2);
\draw[-,line width=1.2pt] (1,-.2) to  (0,2.2);
\draw(0.2,1.6)\bdot;
\draw(-.4,1.6) node {$\scriptstyle \omega_b$};
\node at (0, -.5) {$\scriptstyle a$};
\node at (1, -.5) {$\scriptstyle b$};
\end{tikzpicture}
&= 
 \sum_{0 \leq t \leq \min(a,b)}t!~
\begin{tikzpicture}[baseline = 7.5pt,scale=.7, color=\clr]
\draw[-,thick] (0.58,-.2) to (0.58,0) to (-.48,.8) to (-.48,1);
\draw[-,thick] (-.48,-.2) to (-.48,0) to (.58,.8) to (.58,1);
\draw[-,thick] (-.5,-.2) to (-.5,1);
\draw[-,thick] (0.605,-.2) to (0.605,1);
\draw(.35,0.2)\bdot;
\draw(.1,-0.1) node {$\scriptstyle {\omega_{b-t}}$};
%\node at (0,1.13) {$\scriptstyle b$};
%\node at (0.6,1.13) {$\scriptstyle d$};
\node at (-.5,-.3) {$\scriptstyle a$};
\node at (0.6,-.3) {$\scriptstyle b$};
%\node at (-0.1,.5) {$\scriptstyle s$};
\node at (0.77,.5) {$\scriptstyle t$};
\end{tikzpicture},
\qquad
\begin{tikzpicture}[baseline = 7.5pt, scale=0.35, color=\clr]
\draw[-, line width=1.2pt] (0,-.2) to (1,2.2);
\draw[-,line width=1.2pt] (1,-.2) to(0,2.2);
\draw(.2,0.2)\bdot;
\draw(-.4,.2)node {$\scriptstyle \omega_b$};
\node at (0, -.5) {$\scriptstyle b$};
\node at (1, -.5) {$\scriptstyle a$};
\end{tikzpicture}
=  
 \sum_{0 \leq t \leq \min(a,b)}t!~
\begin{tikzpicture}[baseline = 7.5pt,scale=.7, color=\clr]
\draw[-,thick] (0.58,-.2) to (0.58,0) to (-.48,.8) to (-.48,1);
\draw[-,thick] (-.48,-.2) to (-.48,0) to (.58,.8) to (.58,1);
\draw[-,thick] (-.5,-.2) to (-.5,1);
\draw[-,thick] (0.605,-.2) to (0.605,1);
\draw(.35,0.6)\bdot;
\draw(.01,0.85) node {$\scriptstyle {\omega_{b-t}}$};
%\node at (0,1.13) {$\scriptstyle b$};
%\node at (0.6,1.13) {$\scriptstyle d$};
\node at (-.5,-.3) {$\scriptstyle b$};
\node at (0.6,-.3) {$\scriptstyle a$};
%\node at (-0.1,.5) {$\scriptstyle s$};
\node at (0.77,.5) {$\scriptstyle t$};
\end{tikzpicture},
\\
\label{dotmovesplitss}
\begin{tikzpicture}[baseline = -.5mm,scale=.8,color=\clr]
\draw[-,line width=1.5pt] (0.08,-.5) to (0.08,0.04);
\draw[-,line width=1pt] (0.34,.5) to (0.08,0);
\draw[-,line width=1pt] (-0.2,.5) to (0.08,0);
\node at (-0.22,.6) {$\scriptstyle a$};
\node at (0.36,.65) {$\scriptstyle b$};
\draw (0.08,-.2) \bdot;
\draw (0.65,-.2) node{$\scriptstyle \omega_{a+b}$};
\end{tikzpicture} 
&=
\begin{tikzpicture}[baseline = -.5mm,scale=.8,color=\clr]
\draw[-,line width=1.5pt] (0.08,-.5) to (0.08,0.04);
\draw[-,line width=1pt] (0.34,.5) to (0.08,0);
\draw[-,line width=1pt] (-0.2,.5) to (0.08,0);
\node at (-0.22,.6) {$\scriptstyle a$};
\node at (0.36,.65) {$\scriptstyle b$};
\draw (-.05,.24) \bdot;
\draw (-0.4,.2) node{$\scriptstyle \omega_{a}$};
\draw (0.6,.2) node{$\scriptstyle \omega_{b}$};
\draw (.22,.24) \bdot;
\end{tikzpicture}, 
\qquad\qquad 
%\label{dotmovemergess}
\begin{tikzpicture}[baseline = -.5mm, scale=.8, color=\clr]
\draw[-,line width=1pt] (0.3,-.5) to (0.08,0.04);
\draw[-,line width=1pt] (-0.2,-.5) to (0.08,0.04);
\draw[-,line width=1.5pt] (0.08,.6) to (0.08,0);
\node at (-0.22,-.6) {$\scriptstyle a$};
\node at (0.35,-.6) {$\scriptstyle b$};
\draw (0.08,.2) \bdot;
\draw (0.65,.2) node{$\scriptstyle \omega_{a+b}$};  
\end{tikzpicture}
 = ~
\begin{tikzpicture}[baseline = -.5mm,scale=.8, color=\clr]
\draw[-,line width=1pt] (0.3,-.5) to (0.08,0.04);
\draw[-,line width=1pt] (-0.2,-.5) to (0.08,0.04);
\draw[-,line width=1.5pt] (0.08,.6) to (0.08,0);
\node at (-0.22,-.6) {$\scriptstyle a$};
\node at (0.35,-.6) {$\scriptstyle b$};
\draw (-.08,-.3) \bdot; \draw (.22,-.3) \bdot;
\draw (-0.4,-.3) node{$\scriptstyle \omega_{a}$};
\draw (0.6,-.3) node{$\scriptstyle \omega_{b}$};
\end{tikzpicture}.
\end{align}
\end{definition}

\begin{rem}
  The original definition of $\AW$ in \cite{SWweb} includes the following relation 
 \begin{equation}
     \label{intergralballon}
\begin{tikzpicture}[baseline = 1.5mm, scale=.4, color=\clr]
\draw[-, line width=1.5pt] (0.5,2) to (0.5,2.5);
\draw[-, line width=1.5pt] (0.5,0) to (0.5,-.4);
\draw[-,thin]  (0.5,2) to[out=left,in=up] (-.5,1)
to[out=down,in=left] (0.5,0);
\draw[-,thin]  (0.5,2) to[out=left,in=up] (0,1)
 to[out=down,in=left] (0.5,0);      
\draw[-,thin] (0.5,0)to[out=right,in=down] (1.5,1)
to[out=up,in=right] (0.5,2);
\draw[-,thin] (0.5,0)to[out=right,in=down] (1,1) to[out=up,in=right] (0.5,2);
\node at (0.5,.7){$\scriptstyle \cdots$};
\draw (-0.5,1) \bdot; 
%\node at (-.8,1) {$\scriptstyle \omega_1$}; 
\draw (0,1) \bdot; 
%\node at (0.3,1) {$\scriptstyle \omega_1$};
\node at (0.5,-.6) {$\scriptstyle a$};
\draw (1,1) \bdot;
\draw (1.5,1) \bdot; 
%\node at (1.8,1) {$\scriptstyle \omega_1$};
\node at (-.22,0) {$\scriptstyle 1$};
\node at (1.2,0) {$\scriptstyle 1$};
\node at (.3,0.3) {$\scriptstyle 1$};
\node at (.7,0.3) {$\scriptstyle 1$};
\end{tikzpicture}
=~a!~
\begin{tikzpicture}[baseline = 1.5mm, scale=.6, color=\clr]
\draw[-,line width=1.5pt] (0,-0.1) to[out=up, in=down] (0,1.4);
\draw(0,0.6) \bdot; 
\draw (0.4,0.6) node {$\scriptstyle \omega_a$};
\node at (0,-.3) {$\scriptstyle a$} ;
\end{tikzpicture} 
 \end{equation} 
 as a defining relation, 
where $\dotgen :=\wdotgen$. The $a$-fold merges (and splits) in \eqref{intergralballon} are defined by using \eqref{webassoc} via iterated merges (and splits), and then \eqref{intergralballon} follows from \eqref{webassoc}, \eqref{splitmerge} and \eqref{dotmovesplitss}. 
\end{rem}

Define the following elements 
\begin{equation}
  \label{equ:r=0andr=a}
\begin{tikzpicture}[baseline = -1mm,scale=.6,color=\clr]
\draw[-,line width=1.2pt] (0.08,-.7) to (0.08,.5);
\node at (.08,-.9) {$\scriptstyle a$};
\draw(0.08,0) \bdot;
\draw(.55,0)node {$\scriptstyle \omega_r$};
\end{tikzpicture}
:=
\begin{tikzpicture}[baseline = -1mm,scale=.6,color=\clr]
\draw[-,line width=1.5pt] (0.08,-.8) to (0.08,-.5);
\draw[-,line width=1.5pt] (0.08,.3) to (0.08,.6);
\draw[-,line width=1pt] (0.1,-.51) to [out=45,in=-45] (0.1,.31);
\draw[-,line width=1pt] (0.06,-.51) to [out=135,in=-135] (0.06,.31);
\draw(-.1,0) \bdot;
\draw(-.5,0)node {$\scriptstyle \omega_r$};
\node at (-.25,-.3) {$\scriptstyle r$};
\node at (.08,-1) {$\scriptstyle a$};
%\node at (.45,-.35) {$\scriptstyle b$};
\end{tikzpicture},
\end{equation}
for $1\le r\le a$.
 We have  adopted the following convention which will be used  throughout the paper that
\[ \begin{tikzpicture}[baseline = 10pt, scale=0.4, color=\clr]
            \draw[-,line width=1.2pt] (0,0.5)to[out=up,in=down](0,1.9);
            %\draw (0.3,1.1) node{$\scriptstyle t$};
\node at (0,.2) {$\scriptstyle a$};
\end{tikzpicture}=0 
\text{ unless  } a \ge 0, 
\quad \text{ }
\begin{tikzpicture}[baseline = 3pt, scale=0.4, color=\clr]
\draw[-,line width=1.2pt] (0,0) to[out=up, in=down] (0,1.4);
\draw(0,0.6) \bdot; 
\draw (0.65,0.6) node {$\scriptstyle \omega_r$};
\node at (0,-.3) {$\scriptstyle a$};
\end{tikzpicture}=
0  \text{ unless } 0\le r\le a, \quad \text{ and } 
\begin{tikzpicture}[baseline = 3pt, scale=0.4, color=\clr]
\draw[-,line width=1.2pt] (0,0) to[out=up, in=down] (0,1.4);
\draw(0,0.6) \bdot; 
\draw (0.65,0.6) node {$\scriptstyle \omega_0$};
\node at (0,-.3) {$\scriptstyle a$};
\end{tikzpicture} 
=
\begin{tikzpicture}[baseline = 10pt, scale=0.4, color=\clr]
            \draw[-,line width=1.2pt] (0,0.5)to[out=up,in=down](0,1.9);
            %\draw (0.3,1.1) node{$\scriptstyle t$};
\node at (0,.2) {$\scriptstyle a$};
\end{tikzpicture}.
 \]

\subsection{Basis for the affine web category}

In this subsection, we recall the basis theorem for the category $\AW$ in \cite{SWweb}.

For any $ m\in \mathbb N$, a composition (respectively, partition) $\mu=(\mu_1,\mu_2,\ldots,\mu_n)$ of $m$ is a sequence of (respectively, weakly decreasing) non-negative (respectively, positive) integers such that $\sum_{i\ge 1} \mu_i=m$. The length of $\mu$ is denoted by $l(\mu)$. A composition $\mu$ is called strict if  $\mu_i>0$ for any $i$. Let $\Lambda_{\text{st}}(m)$ (respectively, $\Par(m)$) be the set of all strict compositions (respectively, partitions) of $m$ and $\Lambda_{\text{st}}$ (respectively, $\Par$) be the set of all strict compositions (respectively, partitions).

The set of objects in $\AW$ is $\Lambda_{\text{st}}$ with the empty sequence being the unit object. 
It follows from the generating morphisms  preserving $m$ that  $\Hom_{\AW}(\lambda,\mu)=0$ if $\lambda\in \Lambda_{\text{st}}(m)$ and $\mu \in\Lambda_{\text{st}}(m')$ with $m\ne m'$.

Let $\lambda,\mu\in \Lambda_{\text{st}}(m)$. A $\lambda\times \mu$ \emph{chicken foot diagram} (following \cite{BEEO}) is a  web diagram in $\Hom_{\W}(\mu,\lambda)$ consisting of three  horizontal  parts such that
\begin{itemize}
    \item the bottom part consists of only splits such that the thick strings at the bottom are determined by $\mu$,
    \item the top part consists of only merges such that the thick strings at the top are determined by $\lambda$,
    \item the middle part consists of only crossings of the thinner strands split from the thick strings of $\mu$ at the bottom, and merge back to the thick strings of $\lambda$ at the top.
\end{itemize}
We will always refer to the thinner strings in the middle as the \emph{toes} and the thick strings at the top and bottom as the \emph{legs}.
A chicken foot diagram (CFD) is called \emph{reduced} if there is at most one
crossing or one merge/split between every pair of the toes (i.e., the thinner strings). 

The following is a reduced $\lambda\times \mu$ CFD  for $\lambda=(6,5)$ and $\mu=(4,3,4)$:
\begin{align}
\label{exam-CFD}
\begin{tikzpicture}[anchorbase,scale=1.8,color=\clr]
\draw[-,line width=.6mm] (.212,.5) to (.212,.39);
\draw[-,line width=.75mm] (.595,.5) to (.595,.39);
\draw[-,line width=.15mm] (0.0005,-.396) to (.2,.4);
\draw[-,line width=.15mm]  (.2,.4)to (.4,-.4);
\draw[-,line width=.3mm] (0.01,-.4) to (.59,.4);
\draw[-,line width=.3mm] (.4,-.4) to (.607,.4);
\draw[-,line width=.45mm] (.79,-.4) to (.214,.4);
\draw[-,line width=.15mm] (.8035,-.398) to (.614,.4);
\draw[-,line width=.3mm] (.4006,-.5) to (.4006,-.395);
\draw[-,line width=.6mm] (.788,-.5) to (.788,-.395);
\draw[-,line width=.45mm] (0.011,-.5) to (0.011,-.395);
\node at (0.05,0.15) {$\scriptstyle 3$};
\node at (0.79,0.05) {$\scriptstyle 2$};
\node at (0.35,0.37) {$\scriptstyle 2$};
\node at (0.17,-0.3) {$\scriptstyle 1$};
\node at (0.3,-0.3) {$\scriptstyle 1$};
\node at (0.5,-0.3) {$\scriptstyle 2$};
\node at (0.22,0.58) {$\scriptstyle 6$};
\node at (0.59,0.58) {$\scriptstyle 5$};
\node at (0.0,-0.57) {$\scriptstyle 4$};
\node at (0.4,-0.57) {$\scriptstyle 3$};
\node at (0.78,-0.57) {$\scriptstyle 4$};
\end{tikzpicture}    
\end{align}
  
For $\lambda,\mu\in \Lambda_{\text{st}}(m)$, let 
$\Mat_{\lambda,\mu}$ be the set of all $l(\lambda)\times l(\mu)$
non-negative integer matrices $A=(a_{ij})$ with row sum vector $\la$ and column sum vector $\mu$, i.e.,  $\sum_{1\le h\le l(\mu)}a_{ih}=\lambda_i$ and $\sum_{1\le h\le l(\lambda)}a_{hj}=\mu_j$ for all $1\le i\le l(\lambda) $ and $1\le j\le l(\mu)$. 
The information of each reduced chicken foot diagram is encoded in a matrix $A=(a_{ij})\in \Mat_{\lambda,\mu}$ such that 
$a_{ij}$ is the thickness of the unique strand connecting $\lambda_i$ and $\mu_j$. In this case, we say the reduced chicken foot diagram is of shape $A$.
For example, the CFD in \eqref{exam-CFD} is of shape 
$\begin{pmatrix} 3&1&2\\1&2&2\end{pmatrix}$.

The {\em $r$th elementary dot packet (of thickness $a$)} is defined as 
\[
\omega_{a,r}:= \wkdota.   
\]
We may also write $\omega_r=\omega_{a,r}$ if $a$ is clear in the context.
Let $\Par_a$ be the set of  partitions $\nu=(\nu_1,\ldots,\nu_k)$ such that each part  $\nu_i\le a$.
For any partition $\nu=(\nu_1,\nu_2,\ldots, \nu_k)\in \Par_a$, the {\em elementary dot packet (of thickness $a$)} is defined to be 
\[
\omega_{a, \nu}:= \omega_{a,\nu_1}\omega_{a,\nu_2}\cdots \omega_{a,\nu_k}=
\begin{tikzpicture}[baseline = 3pt, scale=0.5, color=\clr]
\draw[-,line width=1.5pt] (0,-.2) to[out=up, in=down] (0,2.2);
\draw(0,1.7) \bdot; 
\node at (0.6,1.7) {$ \scriptstyle \nu_1$};
\node at (0.3,1.2) {$\vdots$};
\draw(0,0.3) \bdot; 
\node at (0.6,0.2) {$ \scriptstyle \nu_k$};
\node at (0,-.4) {$\scriptstyle a$};
\end{tikzpicture}
\in \End_{\AW}(a).
\]
Recall we only define dot packet for partitions instead of arbitrary compositions since the dots commute (e.g., \cite{SWweb}).
We write $\omega_\nu=\omega_{a,\nu}$ if $a$ is clear from the context and draw it as 
$
\begin{tikzpicture}[baseline = 3pt, scale=0.5, color=\clr]
\draw[-,line width=1.5pt] (0,-.2) to[out=up, in=down] (0,1.2);
\draw(0,0.5) \bdot; \node at (0.6,0.5) {$ \scriptstyle \nu$};
\node at (0,-.4) {$\scriptstyle a$};
\end{tikzpicture}$.

\begin{definition} \label{def:diagram}
    For $\lambda,\mu\in \Lambda_{\text{st}}(m)$, a $\lambda\times \mu$ elementary chicken foot diagram (or simply elementary diagram) is a $\lambda\times \mu$ reduced chicken foot diagram with an elementary dot packet $\omega_{\nu}$, for some partition $\nu\in \Par_a$, attached at the bottom of each toe with thickness $a$. 
\end{definition}

Just as a reduced chicken foot diagram of shape $A$ is encoded by $A\in \Mat_{\lambda,\mu}$, the elementary chicken foot diagrams of shape $A$ are encoded in the matrix $A$ enriched by certain partitions. Denote  
\begin{equation}\label{dottedreduced}
 \PMat_{\lambda,\mu}:=\{ (A, P)\mid  A=(a_{ij})\in \Mat_{\lambda,\mu}, P=(\nu_{ij}), \nu_{ij}\in \Par_{a_{ij}} \}.   
\end{equation}
We will identify the set of all elementary chicken foot diagrams from $\mu$ to $\lambda$ with $\PMat_{\lambda,\mu}$.

\begin{theorem} \cite{SWweb}
 \label{basisAW}
 %{basisofaffine}
For any $\mu,\lambda\in \Lambda_{\text{st}}(m)$, $\Hom_{\AW}(\mu,\lambda)$ has a basis $\PMat_{\lambda,\mu}$ which consists of all elementary chicken foot diagrams from $\mu$ to $\la$. In particular, for each $a\in \N$, the $\kk$-algebra $\End_{\AW}(a)$ is a polynomial algebra in generators $\wkdota$, for $1 \le r \le a$.  
\end{theorem}

\subsection{The isomorphism $\W \cong \Sch$}

For any $m\in \mathbb N$, let $\mathcal H_m:= \kk \mathfrak S_m$ be the group algebra of the symmetric group $\mathfrak S_m$.
For any $\lambda=(\lambda_1,\ldots, \lambda_k)\in \Lambda_{\text{st}}(m)$, let $\mathfrak S_\lambda$ be the Young subgroup $\mathfrak S_{\lambda_1}\times \ldots \times \mathfrak S_{\lambda_k}$ and define 
\begin{equation}
    \textsc{x}_{\lambda}=\sum_{w\in \mathfrak S_\lambda} w.
\end{equation}
Then the permutation module of $\mathcal H_m$ is defined to be
$$ M^\lambda: = \x_\lambda \mathcal H_m.$$
The Schur category $\Sch$ is a category with the object set $\Lambda_{\text{st}}$. For any $\lambda\in \Lambda_{\text{st}}(m)$ and $\mu\in\Lambda_{\text{st}}(m')$, the morphism space is given by
$\Hom_{\Sch}(\lambda,\mu)=\Hom_{\mathcal H_m}(M^\lambda,M^\mu)$ if $m=m'$ and $0$ otherwise.

For any $\lambda =(\lambda_1, \ldots,\lambda_{k})\in \Lambda_{\text{st}}$, the {\em $i$-th merge} of $\lambda$ is defined to be     $\la^{\vartriangle_i}$.
$$\la^{\vartriangle_i}:=(\lambda_1, \lambda_2,\ldots, \lambda_{i-1},\lambda_{i}+\lambda_{i+1},\lambda_{i+2},\ldots, \lambda_k),$$ 
for $1\le i\le k-1$. We also introduce the following sum of minimal length representatives: 
\begin{align}
    \label{eq:minl}
\sigma_{\lambda_i, \lambda_{i+1}} =\sum_{w\in (\mathfrak S_\la/\mathfrak S_{\la^{\vartriangle_i}})_{\text{min}} }w \in \mathcal H_m. 
\end{align}
Then we have 
\begin{equation}\label{glambdai}
    \x_{\la^{\vartriangle_i}}=\x_\lambda \sigma_{\lambda_i,\lambda_{i+1}} = \sigma_{\lambda_i,\lambda_{i+1}}^* \x_\lambda,
\end{equation}
where $*$ is the anti-involution of $\mathcal H_m$ fixing all generators $s_i$. 

\begin{proposition}
[{\cite[Theorem 4.10]{BEEO}}]
\label{cor-isom-schur}
 We have an algebra isomorphism  
\[
\phi: \bigoplus_{\lambda,\mu \in \Lambda_{\text{st}}(m)}\Hom_{\W}(\lambda,\mu) \stackrel{\cong}{\longrightarrow} \End_{\mathcal H_m}(\oplus _{\lambda\in \Lambda_{\text{st}}(m)} M^\lambda)
\]
such that the generators are sent by $\phi$ to 
$$\begin{aligned}\phi\Big(1_{*}\begin{tikzpicture}[baseline = -.5mm,color=\clr]
	\draw[-,line width=1pt] (0.28,-.3) to (0.08,0.04);
	\draw[-,line width=1pt] (-0.12,-.3) to (0.08,0.04);
	\draw[-,line width=1.5pt] (0.08,.4) to (0.08,0);
        \node at (-0.22,-.4) {$\scriptstyle \lambda_i$};
        \node at (0.4,-.45) {$\scriptstyle \lambda_{i+1}$};\end{tikzpicture}1_{*}\Big) &: M^{\lambda} \rightarrow M^{\la^{\vartriangle_i} }, \quad  \x_\lambda h\mapsto \x_{\la^{\vartriangle_i}}h, \\     
\phi\Big(1_{*}\begin{tikzpicture}[baseline = -.5mm,color=\clr]
	\draw[-,line width=1.5pt] (0.08,-.3) to (0.08,0.04);
	\draw[-,line width=1pt] (0.28,.4) to (0.08,0);
	\draw[-,line width=1pt] (-0.12,.4) to (0.08,0);
        \node at (-0.22,.6) {$\scriptstyle \lambda_i$};
        \node at (0.36,.6) {$\scriptstyle \lambda_{i+1}$};
\end{tikzpicture}1_{*} \Big)&: M^{\la^{\vartriangle_i}}\mapsto M^{\lambda},\quad  \x_{\la^{\vartriangle_i}}h\mapsto \x_{\la^{\vartriangle_i}} h,
\end{aligned}
$$
for any $h\in \mathcal H_m$, where $1_*$ stands for suitable identity morphisms.    
\end{proposition}
We include some well-known implied relations in $\W$ (see \cite{CKM} and \cite[(4.30)--(4.33)]{BEEO}), and hence also in $\AW$, which shall be used frequently:
\begin{align}
\begin{tikzpicture}[baseline = .3mm,scale=.7,color=\clr]
	\draw[-,line width=1.5pt] (0.08,.3) to (0.08,.5);
\draw[-,thick] (-.2,-.8) to [out=45,in=-45] (0.1,.31);
\draw[-,thick] (.36,-.8) to [out=135,in=-135] (0.06,.31);
        \node at (-.3,-.95) {$\scriptstyle a$};
        \node at (.45,-.95) {$\scriptstyle b$};
\end{tikzpicture}
&=
\begin{tikzpicture}[baseline = -.6mm,scale=.7,color=\clr]
	\draw[-,line width=1.5pt] (0.08,.1) to (0.08,.5);
\draw[-,thick] (.46,-.8) to [out=100,in=-45] (0.1,.11);
\draw[-,thick] (-.3,-.8) to [out=80,in=-135] (0.06,.11);
        \node at (-.3,-.95) {$\scriptstyle a$};
        \node at (.43,-.95) {$\scriptstyle b$};
\end{tikzpicture},
 \qquad
\begin{tikzpicture}[anchorbase,scale=.7,color=\clr]
	\draw[-,line width=1.5pt] (0.08,-.3) to (0.08,-.5);
\draw[-,thick] (-.2,.8) to [out=-45,in=45] (0.1,-.31);
\draw[-,thick] (.36,.8) to [out=-135,in=135] (0.06,-.31);
        \node at (-.3,.95) {$\scriptstyle a$};
        \node at (.45,.95) {$\scriptstyle b$};
\end{tikzpicture}
=
\begin{tikzpicture}[anchorbase,scale=.7,color=\clr]
	\draw[-,line width=1.5pt] (0.08,-.1) to (0.08,-.5);
\draw[-,thick] (.46,.8) to [out=-100,in=45] (0.1,-.11);
\draw[-,thick] (-.3,.8) to [out=-80,in=135] (0.06,-.11);
        \node at (-.3,.95) {$\scriptstyle a$};
        \node at (.43,.95) {$\scriptstyle b$};
\end{tikzpicture},
\label{swallows}
\\
\begin{tikzpicture}[anchorbase,scale=0.7,color=\clr]
	\draw[-,thick] (0.4,0) to (-0.6,1);
	\draw[-,thick] (0.08,0) to (0.08,1);
	\draw[-,thick] (0.1,0) to (0.1,.6) to (.5,1);
        \node at (0.6,1.13) {$\scriptstyle c$};
        \node at (0.1,1.16) {$\scriptstyle b$};
        \node at (-0.65,1.13) {$\scriptstyle a$};
\end{tikzpicture}
\!\!=\!\!
\begin{tikzpicture}[anchorbase,scale=0.7,color=\clr]
	\draw[-,thick] (0.7,0) to (-0.3,1);
	\draw[-,thick] (0.08,0) to (0.08,1);
	\draw[-,thick] (0.1,0) to (0.1,.2) to (.9,1);
        \node at (0.9,1.13) {$\scriptstyle c$};
        \node at (0.1,1.16) {$\scriptstyle b$};
        \node at (-0.4,1.13) {$\scriptstyle a$};
\end{tikzpicture},
\quad
\begin{tikzpicture}[anchorbase,scale=0.7,color=\clr]
	\draw[-,thick] (-0.4,0) to (0.6,1);
	\draw[-,thick] (-0.08,0) to (-0.08,1);
	\draw[-,thick] (-0.1,0) to (-0.1,.6) to (-.5,1);
        \node at (0.7,1.13) {$\scriptstyle c$};
        \node at (-0.1,1.16) {$\scriptstyle b$};
        \node at (-0.6,1.13) {$\scriptstyle a$};
\end{tikzpicture}
\!\!& =\!\!
\begin{tikzpicture}[anchorbase,scale=0.7,color=\clr]
	\draw[-,thick] (-0.7,0) to (0.3,1);
	\draw[-,thick] (-0.08,0) to (-0.08,1);
	\draw[-,thick] (-0.1,0) to (-0.1,.2) to (-.9,1);
        \node at (0.4,1.13) {$\scriptstyle c$};
        \node at (-0.1,1.16) {$\scriptstyle b$};
        \node at (-0.95,1.13) {$\scriptstyle a$};
\end{tikzpicture},
\quad
\:\begin{tikzpicture}[baseline=-3.3mm,scale=0.7,color=\clr]
	\draw[-,thick] (0.4,0) to (-0.6,-1);
	\draw[-,thick] (0.08,0) to (0.08,-1);
	\draw[-,thick] (0.1,0) to (0.1,-.6) to (.5,-1);
        \node at (0.6,-1.13) {$\scriptstyle c$};
        \node at (0.07,-1.13) {$\scriptstyle b$};
        \node at (-0.6,-1.13) {$\scriptstyle a$};
\end{tikzpicture}
\!\!=\!\!
\begin{tikzpicture}[baseline=-3.3mm,scale=0.7,color=\clr]
	\draw[-,thick] (0.7,0) to (-0.3,-1);
	\draw[-,thick] (0.08,0) to (0.08,-1);
	\draw[-,thick] (0.1,0) to (0.1,-.2) to (.9,-1);
        \node at (1,-1.13) {$\scriptstyle c$};
        \node at (0.1,-1.13) {$\scriptstyle b$};
        \node at (-0.4,-1.13) {$\scriptstyle a$};
\end{tikzpicture},
\quad
\begin{tikzpicture}[baseline=-3.3mm,scale=0.7,color=\clr]
	\draw[-,thick] (-0.4,0) to (0.6,-1);
	\draw[-,thick] (-0.08,0) to (-0.08,-1);
	\draw[-,thick] (-0.1,0) to (-0.1,-.6) to (-.5,-1);
        \node at (0.6,-1.13) {$\scriptstyle c$};
        \node at (-0.1,-1.13) {$\scriptstyle b$};
        \node at (-0.6,-1.13) {$\scriptstyle a$};
\end{tikzpicture}
\!\!=\!\!
\begin{tikzpicture}[baseline=-3.3mm,scale=0.7,color=\clr]
	\draw[-,thick] (-0.7,0) to (0.3,-1);
	\draw[-,thick] (-0.08,0) to (-0.08,-1);
	\draw[-,thick] (-0.1,0) to (-0.1,-.2) to (-.9,-1);
        \node at (0.34,-1.13) {$\scriptstyle c$};
        \node at (-0.1,-1.13) {$\scriptstyle b$};
        \node at (-0.95,-1.13) {$\scriptstyle a$};
\end{tikzpicture},
\label{sliders}
\\
\mathord{
\begin{tikzpicture}[baseline = -1mm,scale=0.8,color=\clr]
	\draw[-,thick] (0.28,0) to[out=90,in=-90] (-0.28,.6);
	\draw[-,thick] (-0.28,0) to[out=90,in=-90] (0.28,.6);
	\draw[-,thick] (0.28,-.6) to[out=90,in=-90] (-0.28,0);
	\draw[-,thick] (-0.28,-.6) to[out=90,in=-90] (0.28,0);
        \node at (0.3,-.75) {$\scriptstyle b$};
        \node at (-0.3,-.75) {$\scriptstyle a$};
\end{tikzpicture}
}
&=
\mathord{
\begin{tikzpicture}[baseline = -1mm,scale=0.8,color=\clr]
	\draw[-,thick] (0.2,-.6) to (0.2,.6);
	\draw[-,thick] (-0.2,-.6) to (-0.2,.6);
        \node at (0.2,-.75) {$\scriptstyle b$};
        \node at (-0.2,-.75) {$\scriptstyle a$};
\end{tikzpicture}
}\:,
\qquad
\mathord{
\begin{tikzpicture}[baseline = -1mm,scale=0.8,color=\clr]
	\draw[-,thick] (0.45,.6) to (-0.45,-.6);
	\draw[-,thick] (0.45,-.6) to (-0.45,.6);
        \draw[-,thick] (0,-.6) to[out=90,in=-90] (-.45,0);
        \draw[-,thick] (-0.45,0) to[out=90,in=-90] (0,0.6);
        \node at (0,-.77) {$\scriptstyle b$};
        \node at (0.5,-.77) {$\scriptstyle c$};
        \node at (-0.5,-.77) {$\scriptstyle a$};
\end{tikzpicture}
}
=
\mathord{
\begin{tikzpicture}[baseline = -1mm,scale=0.8,color=\clr]
	\draw[-,thick] (0.45,.6) to (-0.45,-.6);
	\draw[-,thick] (0.45,-.6) to (-0.45,.6);
        \draw[-,thick] (0,-.6) to[out=90,in=-90] (.45,0);
        \draw[-,thick] (0.45,0) to[out=90,in=-90] (0,0.6);
        \node at (0,-.77) {$\scriptstyle b$};
        \node at (0.5,-.77) {$\scriptstyle c$};
        \node at (-0.5,-.77) {$\scriptstyle a$};
\end{tikzpicture}
}\:.
\label{braid}
\end{align}

%% file: Section3_affineschur.tex
\section{The affine Schur category}
 \label{sec:ASchur}

In this section, we introduce the affine Schur category $\ASch$ over a commutative ring $\kk$ with 1. Then we construct a polynomial representation of $\ASch$ and use it to establish a basis theorem for $\ASch$.
%Later on, we will give an equivalent but much simpler definition   if $\kk$ is a field with characteristic zero. 

\subsection{Definition of the affine Schur category}
Recall the notation \eqref{equ:r=0andr=a}. For any $r\in \Z_{>0}$ and $u\in \kk$, let
\begin{equation}
    \label{def-gau}
 g_{r}(u):=\sum_{0\le i\le r}(-1)^i\left(\prod_{0\le j\le i-1}(u+j)\right)\:\:
 \begin{tikzpicture}[baseline = -1mm,scale=.7,color=\clr]
\draw[-,line width=1.2pt] (0.08,-.6) to (0.08,.5);
\node at (.08,-.8) {$\scriptstyle r$};
\draw(0.08,0) \bdot;
\draw(.7,0)node {$\scriptstyle \omega_{r-i}$};
\end{tikzpicture}.
\end{equation}
Recall in \cite[Lemma 2.11]{SWweb} that  $ g_r(u)$ is a balloon morphism (up to a factor of $r!$) and it is also used to define the cyclotomic web category  \cite[Definition 4.1]{SWweb}.
We  shall also draw $g_{r}(u)$ as 
\begin{tikzpicture}[baseline = 10pt, scale=0.4, color=\clr]
\draw[-,line width=1pt] (-1,2) to (-1,0.2);
\draw (-1,1) \bdot;
\draw (0.1,1) node{$\scriptstyle {g_{r}(u)}$};
\draw (-1,-.2) node{$\scriptstyle {r}$}; 
\end{tikzpicture}
and use them in the definition below. 

\begin{definition}  \label{def-affine-Schur}
The affine Schur category $\ASch$ is a strict $\kk$-linear monoidal category with generating objects $a\in \Z_{\ge 1}$ and $u\in \kk$. We denote the object $a \in \Z_{\ge1}$ by a thick strand labeled by $a $ and denote the object $u\in \kk$ by a red strand labeled by $u$: 
\[
\begin{tikzpicture}[baseline = 10pt, scale=0.5, color=\cred]
\draw[-,line width=1pt] (0,0.2) to (0,1.5);
\draw(0,-.2) node {$\scriptstyle u$};
\end{tikzpicture} \; .
\]
The morphisms are generated by 
\begin{equation} 
\label{generator-affschur}
  \begin{tikzpicture}[baseline = 10pt, scale=0.4, color=\clr] 
                \draw[-,line width=1pt] (-2,0.2) to (-1.5,1);
	\draw[-,line width=1pt ] (-1,0.2) to (-1.5,1);
	\draw[-,line width=1.5pt] (-1.5,1) to (-1.5,1.8);
                \draw (-1,0) node{$\scriptstyle {b}$};
                \draw (-2,0) node{$\scriptstyle {a}$};
                \draw (-1.5,2.2) node{$\scriptstyle {a+b}$};
                \end{tikzpicture}, 
                \quad 
\begin{tikzpicture}[baseline = 10pt, scale=0.4, color=\clr] 
                \draw[-,line width=1pt] (-2,0.2) to (-1,1.8);
	\draw[-,line width=1pt ] (-1,0.2) to (-2,1.8); 
                \draw (-1,0) node{$\scriptstyle {b}$};
                \draw (-2,0) node{$\scriptstyle {a}$};
                \draw (-1,2.2) node{$\scriptstyle {a}$};
                 \draw (-2,2.2) node{$\scriptstyle {b}$};
                \end{tikzpicture},  
                \quad 
  \begin{tikzpicture}[baseline = 10pt, scale=0.4, color=\clr]
                \draw[-,line width=1pt] (-2,1.8) to (-1.5,1);
	\draw[-,line width=1pt ] (-1,1.8) to (-1.5,1);
	\draw[-,line width=1.5pt] (-1.5,1) to (-1.5,0.2);
                \draw (-1,2.2) node{$\scriptstyle {b}$};
                \draw (-2,2.2) node{$\scriptstyle {a}$};
                \draw (-1.5,0) node{$\scriptstyle {a+b}$};
                \end{tikzpicture}, 
                \quad  \;
 ~
 \wkdotaa \quad   \text{ for } a,b \in \Z_{\ge1},              
 \end{equation}
and two mixed-color crossings
\begin{equation}
\label{rightleftcrossinggen}
\begin{tikzpicture}[baseline = 10pt, scale=.8, color=\clr]
 \draw[-,line width=1.2pt] (-0.3,.3) to (.3,1);
\draw[-,line width=1pt,color=\cred] (0.3,.3) to (-.3,1);
\draw(-.3,0.15) node{$\scriptstyle a$};
\draw (.3, 0.15) node{$\scriptstyle \red{u}$};
\end{tikzpicture} \;\; (\text{traverse-up}), 
                \qquad
\begin{tikzpicture}[baseline = 10pt, scale=.8, color=\clr]
 \draw[-,line width=1pt,color=\cred] (-0.3,.3) to (.3,1);
\draw[-,line width=1.2pt] (0.3,.3) to (-.3,1);
\draw(-.3,0.15) node{$\scriptstyle \red{u}$};
\draw (.3, 0.15) node{$\scriptstyle a$};
\end{tikzpicture}  \;\; (\text{traverse-down}),
\qquad \text{ for } a\ge 1, u \in \kk
\end{equation}
subject to relations in \eqref{webassoc}--\eqref{intergralballon} among generating morphisms in \eqref{generator-affschur} and additional relations \eqref{adaptorR}--\eqref{adaptermovemerge} involving also traverse-ups and traverse-downs:
\begin{align}
  \label{adaptorR}
\begin{tikzpicture}[baseline = 10pt, scale=0.35, color=\clr]
\draw[-,line width=1pt,color=\cred](0,-.2)to (0,2.6);
\draw (0,-.4) node{$\scriptstyle \red{u}$};
\draw[-,line width= 1.2pt](-.5,-.1) to[out=45, in=down] (.5,1.4);
\draw[-,line width= 1.2pt](.5,1.4) to[out=up, in=300] (-.5,2.4); 
\draw (-.5,-0.4) node{$\scriptstyle {r}$};
\end{tikzpicture}
&=~
 \begin{tikzpicture}[baseline = -1mm,scale=.7,color=\clr]
\draw[-,line width=1.2pt] (0.08,-.6) to (0.08,.5);
\node at (.08,-.8) {$\scriptstyle r$};
\draw(0.08,0) \bdot;
\draw(-.5,0)node {$\scriptstyle g_r(u)$};
\draw[-,line width=1pt,color=\cred](.5,-.6)to (.5,.5);
\draw (.5,-.8) node{$\scriptstyle \red{u}$};
\end{tikzpicture},  
\\
  \label{adaptorL}
\begin{tikzpicture}[baseline = 10pt, scale=0.35, color=\clr]
\draw[-,line width=1pt,color=\cred](0,-.6)to (0,2.4);
\draw (0,-.9) node{$\scriptstyle \red{u}$};
\draw[-,line width= 1.2pt](.5,-.5) to[out=135, in=down] (-.5,1);
\draw[-,line width= 1.2pt](-.5,1) to[out=up, in=270] (.5,2.3);
\draw (.5,-0.9) node{$\scriptstyle {r}$};
\end{tikzpicture}
&=~
 \begin{tikzpicture}[baseline = -1mm,scale=.7,color=\clr]
\draw[-,line width=1.2pt] (0.08,-.6) to (0.08,.5);
\node at (.08,-.8) {$\scriptstyle r$};
\draw(0.08,0) \bdot;
\draw(.7,0)node {$\scriptstyle g_{r}(u)$};
\draw[-,line width=1pt,color=\cred](-.4,-.6)to (-.4,.5);
\draw (-.4,-.8) node{$\scriptstyle \red{u}$};
\end{tikzpicture} , 
\\
\label{adaptrermovecross} 
\begin{tikzpicture}[baseline = 7.5pt, scale=0.35, color=\clr]
\draw[-,line width=1pt,color=\cred](.8,2) to[out=down,in=90] (0.1,1) to [out=down,in=up] (.8,-.4);
\draw (.8,-.65) node{$\scriptstyle \red{u}$};
\draw[-,line width=1.2pt] (0,-.35) to (1.5,2);
\draw[-,line width=1.2pt](0,2) to  (1.5,-.35);
\node at (0, -.6) {$\scriptstyle a$};
\node at (1.5, -.6) {$\scriptstyle b$};
\end{tikzpicture}
&=~
\begin{tikzpicture}[baseline = 7.5pt, scale=0.35,color=\clr]
\draw[-,line width=1pt,color=\cred](.8,2) to[out=down,in=90] (1.5,1) to [out=down,in=up] (.8,-.4);
\draw (.8,-.65) node{$\scriptstyle \red{u}$};
\draw[-,line width=1.2pt] (0,-.35) to (1.5,2);
\draw[-,line width=1.2pt](0,2) to  (1.5,-.35);
\node at (0, -.6) {$\scriptstyle a$};
\node at (1.5, -.6) {$\scriptstyle b$};
\end{tikzpicture}
~+~
\sum_{1\le t\le \min\{a,b\}} t! 
\begin{tikzpicture}[baseline = 7.5pt, scale=0.35, color=\clr]
\draw[-, line width=1pt] (0,-.3) to (0,2);
\draw[-, line width=1.2pt](0,-.3) to (0,.4);
\draw[-, line width=1.2pt] (1.8,-.2) to (1.8,.5);
\draw[-, line width=1.2pt] (1.8,1.6) to (1.8,2);
\draw[-, line width=1.2pt]   (0,.2) to  (1.8,1.6);
\draw[-, line width=1.2pt]   (0,1.6) to(1.8,0);
\draw[-,line width=1pt,color=\cred](.8,2) to[out=down,in=90] (1.5,1) to [out=down,in=up] (.8,-.4);
\draw (.8,-.6) node{$\scriptstyle \red{u}$};
\draw[-, line width=1.2pt]   (0,1.6) to (0,2); 
\draw[-, line width=1pt] (1.8,2) to (1.8,-.2); 
\node at (0, -.5) {$\scriptstyle a$};
\node at (2, -.5) {$\scriptstyle b$};
\node at (2, 2.2) {$\scriptstyle a$};
\node at (0, 2.2) {$\scriptstyle b$};
\node at (2,1.2) {$\scriptstyle t$};
\end{tikzpicture} ,
\\
 \label{adaptermovemerge}
\begin{tikzpicture}[anchorbase,scale=.7,color=\clr]
	\draw[-,line width=1pt,color=\cred] (0.4,0) to (-0.6,1);
  \draw (-.6,1.2) node{$\scriptstyle \red{u}$};
	\draw[-,thick] (0.08,0) to (0.08,1);
	\draw[-,thick] (0.1,0) to (0.1,.6) to (.5,1);
        \node at (0.6,1.13) {$\scriptstyle c$};
        \node at (0.1,1.16) {$\scriptstyle b$};
        %\node at (-0.65,1.13) {$\scriptstyle a$};
\end{tikzpicture}
& =\!\!
\begin{tikzpicture}[anchorbase,scale=.7,color=\clr]
	\draw[-,line width=1pt,color=\cred] (0.7,0) to (-0.3,1);
  \draw (-.3,1.16) node{$\scriptstyle \red{u}$};
	\draw[-,thick] (0.08,0) to (0.08,1);
	\draw[-,thick] (0.1,0) to (0.1,.2) to (.9,1);
        \node at (0.9,1.13) {$\scriptstyle c$};
        \node at (0.1,1.16) {$\scriptstyle b$};
        %\node at (-0.4,1.13) {$\scriptstyle a$};
\end{tikzpicture},
\:
\qquad
\begin{tikzpicture}[anchorbase,scale=.7,color=\clr]
	\draw[-,line width=1pt,color=\cred] (-0.4,0) to (0.6,1);
  \draw (.6,1.16) node{$\scriptstyle \red{u}$};
	\draw[-,thick] (-0.08,0) to (-0.08,1);
	\draw[-,thick] (-0.1,0) to (-0.1,.6) to (-.5,1);
        %\node at (0.7,1.13) {$\scriptstyle c$};
        \node at (-0.1,1.16) {$\scriptstyle b$};
        \node at (-0.6,1.13) {$\scriptstyle a$};
\end{tikzpicture}
\!\!=\!\!
\begin{tikzpicture}[anchorbase,scale=.7,color=\clr]
	\draw[-,line width=1pt,color=\cred] (-0.7,0) to (0.3,1);
 \draw (.3,1.16) node{$\scriptstyle \red{u}$};
	\draw[-,thick] (-0.08,0) to (-0.08,1);
	\draw[-,thick] (-0.1,0) to (-0.1,.2) to (-.9,1);
        %\node at (0.4,1.13) {$\scriptstyle c$};
        \node at (-0.1,1.16) {$\scriptstyle b$};
        \node at (-0.95,1.13) {$\scriptstyle a$};
\end{tikzpicture},
\\
\notag
\begin{tikzpicture}[baseline=-3.3mm,scale=.7,color=\clr]
\draw[-,line width=1pt,color=\cred] (0.4,.2) to (-0.6,-.8);
\draw (-.6,-.91) node{$\scriptstyle \red{u}$};
\draw[-,thick] (0.08,0.2) to (0.08,-.75);
\draw[-,thick] (0.1,0.2) to (0.1,-.4) to (.5,-.8);
\node at (0.6,-.91) {$\scriptstyle c$};
\node at (0.07,-.9) {$\scriptstyle b$};
\end{tikzpicture}
&=\!\!
\begin{tikzpicture}[baseline=-3.3mm,scale=.7,color=\clr]
\draw[-,line width=1pt,color=\cred] (0.7,0.2) to (-0.3,-.8);
\draw (-.3,-.91) node{$\scriptstyle \red{u}$};
\draw[-,thick] (0.08,0.2) to (0.08,-.75);
\draw[-,thick] (0.1,0.2) to (0.1,0) to (.9,-.8);
\node at (1,-.91) {$\scriptstyle c$};
\node at (0.1,-.9) {$\scriptstyle b$};
\end{tikzpicture}
,\:
\qquad
\begin{tikzpicture}[baseline=-3.3mm,scale=.7,color=\clr]
\draw[-,line width=1pt,color=\cred] (-0.4,0.2) to (0.6,-.8);
\draw (.6,-.91) node{$\scriptstyle \red{u}$};
\draw[-,thick] (-0.08,0.2) to (-0.08,-.75);
\draw[-,thick] (-0.1,0.2) to (-0.1,-.4) to (-.5,-.8);
\node at (-0.1,-.9) {$\scriptstyle b$};
\node at (-0.6,-.91) {$\scriptstyle a$};
\end{tikzpicture}
\!\!=\!\!
\begin{tikzpicture}[baseline=-3.3mm,scale=.7,color=\clr]
\draw[-,line width=1pt,color=\cred] (-0.7,0.2) to (0.3,-.8);
\draw (.3,-.91) node{$\scriptstyle \red{u}$};
\draw[-,thick] (-0.08,0.2) to (-0.08,-.75);
\draw[-,thick] (-0.1,0.2) to (-0.1,0) to (-.9,-.8);
\node at (-0.1,-.9) {$\scriptstyle b$};
\node at (-0.95,-.91) {$\scriptstyle a$};
\end{tikzpicture} .
\end{align}
\end{definition}

\begin{rem}
The relations \eqref{adaptorR}--\eqref{adaptorL} for a general $r\ge 1$ will be understood in terms of the corresponding relations for $r=1$ for $\ASch$ over $\C$ in Section~\ref{sec:C}. 
\end{rem}

\begin{rem} 
A ``free" monoidal category with essentially the same generating objects and many more generating morphisms as a $q$-version of $\ASch$ (but with no relations imposed) was also suggested in \cite{MS21}.
\end{rem}

By the symmetry of the defining relations in Definition \ref{def-affine-Schur}, there is an equivalence  of  categories:
\begin{equation} 
\label{anti-autocyc}
\div : \ASch\longrightarrow \mathpzc{Schur}^{{\hspace{-.03in}}\bullet\,\text{op}}%\ASch^{\text{op}} 
\end{equation}
defined by reflecting each diagram along a horizontal axis.

\subsection{More relations in $\ASch$}

\begin{lemma}
 \label{adamovecrossings}
The following local relations hold in $\ASch$:
\begin{align*}
\begin{tikzpicture}[baseline = 7.5pt, scale=0.4, color=\clr]
\draw[-,line width=1.2pt](-.5,-.5) to (.5,1.5); 
\draw[-,line width=1.2pt](-.5,1.5) to (.5,-.5);
\draw[-,line width=1pt,color=\cred] (.8,1.5) to (0.8,.8) to[out=down,in=up] (-.8,-.5);
\draw(-.85,-.7) node{$\scriptstyle \red u$};
\end{tikzpicture}
    &
    ~=~
\begin{tikzpicture}[baseline = 7.5pt, scale=0.4, color=\clr]
\draw[-,line width=1.2pt](-.5,-.5) to (.5,1.5); 
\draw[-,line width=1.2pt](-.5,1.5) to (.5,-.5);
\draw[-,line width=1pt,color=\cred] (.8,1.5) to [out=down,in=45](-.7,.5) to [out=-135,in=up] (-.8,-.4);
\draw(-.85,-.7) node{$\scriptstyle \red u$};
\end{tikzpicture},  
    \qquad\qquad
\begin{tikzpicture}[baseline = 7.5pt, scale=0.4, color=\clr]
\draw[-,line width=1.2pt](-.5,-.5) to (.5,1.5); 
\draw[-,line width=1.2pt](-.5,1.5) to (.5,-.5);
\draw[-,line width=1pt,color=\cred] (-.8,1.5) to (-.8,.9) to [out=down,in=up] (.8,-.5);
\draw(.8,-.7) node{$\scriptstyle \red u$};
\end{tikzpicture}
    ~=~
\begin{tikzpicture}[baseline = 7.5pt, scale=0.4, color=\clr]
\draw[-,line width=1.2pt](-.5,-.5) to (.5,1.5); 
\draw[-,line width=1.2pt](-.5,1.5) to (.5,-.5);
\draw[-,line width=1pt,color=\cred] (-.8,1.5) to [out=down,in=134] 
(.7,.5) to[out=-45,in=up](.8,-.4);
\draw(.8,-.7) node{$\scriptstyle \red{u}$};
\end{tikzpicture}.
\end{align*}
\end{lemma}

\begin{proof}
Recall that  the crossing can be defined via (e.g., \cite[(4.36)]{BEEO})  
\begin{equation}
\label{crossgen}
\begin{tikzpicture}[baseline = 2mm,scale=.8, color=\clr]
	\draw[-,line width=1.2pt] (0,0) to (.6,1);
	\draw[-,line width=1.2pt] (0,1) to (.6,0);
        \node at (0,-.1) {$\scriptstyle a$};
        \node at (0.6,-0.1) {$\scriptstyle b$};
\end{tikzpicture}
:=\sum_{t=0}^{\min(a,b)}
(-1)^t
\begin{tikzpicture}[baseline = 2mm,scale=.8, color=\clr]
\draw[-,thick] (0,0) to (0,1);
\draw[-,thick] (.015,0) to (0.015,.2) to (.57,.4) to (.57,.6)to (.015,.8) to (.015,1);
\draw[-,line width=1.2pt] (0.6,0) to (0.6,1);
\node at (0.6,-.1) {$\scriptstyle b$};
\node at (0,-.1) {$\scriptstyle a$};
\node at (0.6,1.1) {$\scriptstyle a$};
\node at (0,1.1) {$\scriptstyle b$};        \node at (-0.13,.5) {$\scriptstyle t$};
\end{tikzpicture}.
\end{equation}
Then the result follows since the red strand can move through merges and splits freely by \eqref{adaptermovemerge}. 
\end{proof}

\begin{lemma}
The following  relations hold in $\ASch$:
\begin{equation}
\label{dotmoveadaptor}
\begin{tikzpicture}[baseline = 10pt, scale=0.4, color=\clr]
\draw[-,line width=1.2pt] (-1.5,2.2) to (-1.5,1);
\draw[-,line width=1,color=\cred](-2,2.2) to [out=down,in=up] (-1,0);
\draw (-1,-.5) node{$\scriptstyle \red{u}$};
\draw[-,line width=1.2pt] (-1.5,1) to (-1.5,0); 
\draw(-1.5,.5)\bdot;
\draw (-2.1,.5) node{$\scriptstyle \omega_r$};
\draw (-1.5,-.5) node{$\scriptstyle {a}$};
\end{tikzpicture}
 ~=~
\begin{tikzpicture}[baseline = 10pt, scale=0.4, color=\clr]
\draw[-,line width=1.2pt] (-1.5,2.2) to (-1.5,1);
\draw[-,line width=1,color=\cred](-2.2,2) to [out=down,in=up] (-1,0);
\draw (-1,-.5) node{$\scriptstyle \red{u}$};
\draw[-,line width=1.2pt] (-1.5,1) to (-1.5,0); 
\draw(-1.5,1.5)\bdot;
\draw (-.9,1.5) node{$\scriptstyle \omega_r$};
\draw (-1.5,-.5) node{$\scriptstyle {a}$};
\end{tikzpicture}, 
\quad 
\begin{tikzpicture}[baseline = 10pt, scale=0.4, color=\clr]
\draw[-,line width=1.2pt] (-1.5,2.2) to (-1.5,1);
\draw[-,line width=1,color=\cred](-1,2) to [out=down,in=up] (-2.2,0);
\draw (-2.2,-.5) node{$\scriptstyle \red{u}$};
\draw[-,line width=1.2pt] (-1.5,1) to (-1.5,0);
\draw(-1.5,.5)\bdot;
\draw (-.9,.5) node{$\scriptstyle \omega_r$};
\draw (-1.5,-.5) node{$\scriptstyle {a}$};
\end{tikzpicture}
~=~
\begin{tikzpicture}[baseline = 10pt, scale=0.4, color=\clr]
\draw[-,line width=1.2pt] (-1.5,2.2) to (-1.5,1);
\draw[-,line width=1,color=\cred](-.8,2) to [out=down,in=up] (-2.2,0);
\draw (-2.2,-.5) node{$\scriptstyle \red{u}$};
\draw[-,line width=1.2pt] (-1.5,1) to (-1.5,0);
\draw(-1.5,1.5)\bdot;
\draw (-2.1,1.5) node{$\scriptstyle \omega_r$};
\draw (-1.5,-.5) node{$\scriptstyle {a}$};
\end{tikzpicture}
 \end{equation}
 for all admissible $a,r$ and any  $u\in \kk$.
 \end{lemma}
 
\begin{proof}
We only prove the first one as the second one is obtained by symmetry.

We prove by induction on $r$. 
Suppose $a=1=r$. Then  we have 
\begin{align}
\label{r=1anda=1}
\begin{tikzpicture}[baseline = 10pt, scale=0.35, color=\clr]
\draw[-,line width=1pt] (-1.5,2.2) to (-1.5,1);
\draw[-,line width=1,color=\cred](-2,2.2) to [out=down,in=up] (-1,0);
\draw (-1,-.35) node{$\scriptstyle \red{u}$};
\draw[-,line width=1pt] (-1.5,1) to (-1.5,0); 
\draw(-1.5,.5)\bdot;
\draw (-1.5,-.35) node{$\scriptstyle {1}$};
 \end{tikzpicture}
\overset{\eqref{adaptorR}}
=
\begin{tikzpicture}[baseline = 10pt, scale=0.35, color=\clr]
\draw[-,line width=1pt] (0,0) to (0,2);
\draw[-,line width=1,color=\cred](-.5,2) to [out=down,in=up] (.5,1.5) to[out=down,in=up] (-.5,0.5)to[out=down,in=up](.5,0);
\draw (.5,-.35) node{$\scriptstyle \red{u}$};
\draw (0,-.35) node{$\scriptstyle {1}$};
\end{tikzpicture} 
+~u~ 
\begin{tikzpicture}[baseline = 10pt, scale=0.35, color=\clr]
\draw[-,line width=1pt] (-1.5,2) to (-1.5,0);
\draw[-,line width=1,color=\cred](-2,2) to [out=down,in=up] (-1,0);
\draw (-1,-.35) node{$\scriptstyle \red{u}$};
\draw (-1.5,-.35) node{$\scriptstyle {1}$};
\end{tikzpicture}          
~\overset{\eqref{adaptorL}}
=~
\begin{tikzpicture}[baseline = 10pt, scale=0.35, color=\clr]
\draw[-,line width=1pt] (-1.5,2.2) to (-1.5,1);
\draw[-,line width=1,color=\cred](-2.2,2) to [out=down,in=up] (-1,0);
\draw (-1,-.35) node{$\scriptstyle \red{u}$};
\draw[-,line width=1pt] (-1.5,1) to (-1.5,0); 
\draw(-1.5,1.5)\bdot;
\draw (-1.5,-.35) node{$\scriptstyle {1}$};
\end{tikzpicture}. 
    \end{align}
Suppose $a\ge 2$ and $r=1$. Then the result follows from 
\eqref{r=1anda=1} and \eqref{adaptermovemerge} since 
$
\begin{tikzpicture}[baseline = 3pt, scale=0.4, color=\clr]
\draw[-,line width=1.2pt] (0,0) to[out=up, in=down] (0,1.4);
\draw(0,0.6) \bdot; 
\draw (0.7,0.6) node {$\scriptstyle \omega_1$};
\node at (0,-.3) {$\scriptstyle a$};
\end{tikzpicture}
=
\begin{tikzpicture}[baseline = 5pt, scale=0.4, color=\clr] 
 \draw[-, line width=1.2pt] (0.5,.9) to (0.5,1.5);
 \draw[-, line width=1.2pt] (0.5,0) to (0.5,-.4);
\draw[-,thick] (0.6,1) to (0.5,1) to[out=left,in=up] (0,0.5)to[out=down,in=left] (0.5,0);
\draw[-,line width=1.2pt] (0.5,0)to[out=right,in=down] (1,0.5)to[out=up,in=right] (0.5,1);
\draw (0,0.5) \bdot;
\node at (0.1,-.1) {$\scriptstyle 1$};
\node at (0.5,-0.6) {$\scriptstyle a$};
\end{tikzpicture}
$
by \eqref{equ:r=0andr=a}.
 Suppose $a=r\ge2$.  We have 
\begin{align*}
\begin{tikzpicture}[baseline = 10pt, scale=0.4, color=\clr]
\draw[-,line width=1.2pt] (-1.5,2.2) to (-1.5,1);
\draw[-,line width=1,color=\cred](-2,2.2) to [out=down,in=up] (-1,0);
\draw (-1,-.5) node{$\scriptstyle \red{u}$};
\draw[-,line width=1.2pt] (-1.5,1) to (-1.5,0); 
\draw(-1.5,.5)\bdot;
\draw (-2.1,.5) node{$\scriptstyle \omega_r$};
\draw (-1.5,-.5) node{$\scriptstyle {r}$};
\end{tikzpicture}
&
\overset{\eqref{adaptorR}}
=
\begin{tikzpicture}[baseline = 10pt, scale=0.4, color=\clr]
\draw[-,line width=1.2pt] (0,0) to (0,2);
\draw[-,line width=1,color=\cred](-.5,2) to [out=down,in=up] (.5,1.5) to[out=down,in=up] (-.5,0.5)to[out=down,in=up](.5,0);
\draw (.5,-.35) node{$\scriptstyle \red{u}$};
\draw (0,-.35) node{$\scriptstyle {r}$};
\end{tikzpicture} 
-
\sum_{1\le i\le r}(-1)^i\prod_{0\le j\le i-1}(u+j)\:\:
\begin{tikzpicture}[baseline = 10pt, scale=0.4, color=\clr]
\draw[-,line width=1.2pt] (-1.5,2) to (-1.5,0);
\draw[-,line width=1,color=\cred](-2,2) to [out=down,in=up] (-1,0);
\draw (-1,-.35) node{$\scriptstyle \red{u}$};
\draw(-1.5,.5)\bdot;
\draw (-2.5,.5) node{$\scriptstyle \omega_{r-i}$};
\draw (-1.5,-.35) node{$\scriptstyle {r}$};
\end{tikzpicture}\\
&
\stackrel{(*)}{=} \begin{tikzpicture}[baseline = 10pt, scale=0.4, color=\clr]
\draw[-,line width=1.2pt] (0,0) to (0,2);
\draw[-,line width=1,color=\cred](-.5,2) to [out=down,in=up] (.5,1.5) to[out=down,in=up] (-.5,0.5)to[out=down,in=up](.5,0);
\draw (.5,-.35) node{$\scriptstyle \red{u}$};
\draw (0,-.35) node{$\scriptstyle {r}$};
\end{tikzpicture} 
-
\sum_{1\le i\le r}(-1)^i\prod_{0\le j\le i-1}(u+j)\:\:
\begin{tikzpicture}[baseline = 10pt, scale=0.4, color=\clr]
\draw[-,line width=1.2pt] (-1.5,2) to (-1.5,0);
\draw[-,line width=1,color=\cred](-2,2) to [out=down,in=up] (-1,0);
\draw (-1,-.35) node{$\scriptstyle \red{u}$};
\draw(-1.5,1.5)\bdot;
\draw (-.4,1.5) node{$\scriptstyle \omega_{r-i}$};
\draw (-1.5,-.35) node{$\scriptstyle {r}$};
\end{tikzpicture}  \\
&
\overset{\eqref{adaptorL}}=
\begin{tikzpicture}[baseline = 10pt, scale=0.4, color=\clr]
\draw[-,line width=1.2pt] (-1.5,2.2) to (-1.5,1);
\draw[-,line width=1,color=\cred](-2.2,2) to [out=down,in=up] (-1,0);
\draw (-1,-.5) node{$\scriptstyle \red{u}$};
\draw[-,line width=1.2pt] (-1.5,1) to (-1.5,0); 
\draw(-1.5,1.5)\bdot;
\draw (-.8,1.5) node{$\scriptstyle \omega_r$};
\draw (-1.5,-.5) node{$\scriptstyle {r}$};
\end{tikzpicture},
\end{align*}
where we have applied the inductive assumption on $r-i$ for the equation (*) above.
The case $a>r$ follows from the case $a=r$,  \eqref{adaptermovemerge}  and \eqref{splitmerge}.
This completes the proof of the first relation.
\end{proof}

\subsection{ A polynomial representation of $\ASch$}
 \label{polynomilrep}

Recall the degenerate affine Hecke algebra $\widehat {\mathcal H}_d$ is the associative unital $\kk$-algebra generated by $s_i$ and $x_l$, for $1\le i\le d-1$, $1\le l \le d$, subject to the relations for all admissible $i,j,l$:
\begin{align}
        s_i^2=1, \quad s_is_{i+1}s_i &=s_{i+1}s_is_{i+1}, \quad s_js_i =s_is_j \; (|i-j|>1),
        \label{ss} \\
         x_ix_l& =x_lx_i,  \quad     s_jx_i=x_is_j\; (j\neq i,i+1),
        \label{xx} \\
        x_{i+1}s_i &=s_ix_i-1. 
        %\quad s_ix_{i+1}=x_is_i-1,
        \label{xs}
\end{align}  
(%Actually, the relation \eqref{xs} differs from that in \cite{Kle05} by a sign so that  
The isomorphism  $\End_{\AW}(1^m)\cong \hat {\mathcal H}_d$ \cite{SWweb} sends the diagram with a dot on $i$-th strand, reading from left to right, to $x_i$.)
 
For any $a\in \N$, let 
\[
\text{Sym}_a:=\kk[x_1,x_2,\ldots, x_a]^{\mathfrak S_a}
\]
be the ring of symmetric polynomials in $x_1,\ldots, x_a$.
More generally we consider the ring of partially symmetric functions
\[\text{Sym}_\lambda:=\kk[x_1,\ldots, x_a]^{\mathfrak S_{\lambda}}\cong \text{Sym}_{\lambda_1} \otimes \ldots\otimes \text{Sym}_{\lambda_s}
\]
for any $\lambda=(\lambda_1,\ldots, \lambda_s)\in \Lambda_{\text{st}}(a)$.

 Recall that the degenerate affine  Hecke algebra $\widehat H_d$ acts  on $\kk[x_1,\ldots, x_d]$ by letting $x_i$ act by multiplication by $x_i$ and $s_j$ act by
 \begin{equation}\label{actionofdeaffonpoly}
     s_j. f= f^{s_j}+\frac{ f-f^{s_j}}{x_j-x_{j+1}}
 \end{equation}
for all 
$f\in \kk[x_1,\ldots, x_d] $ and all admissible $i,j$,
where 
\[ f^{s_j}(x_1,\ldots, x_d)=f(x_1, \ldots, x_{j+1}, x_j,\ldots,x_d).\]
 
Associated with each generator of $\ASch$ in \eqref{merge i}--\eqref{chameleonsgen}, we introduce certain linear maps between polynomial rings denoted by the same symbols as follows. 
Associated with the split $\splits$, we define the natural embedding
\[
\splits: \text{Sym}_{a+b}\longrightarrow \text{Sym}_{(a,b)}, \quad f\mapsto f.
\]

Associated with the merge $\merge$, we define 
\[
\merge: \text{Sym}_{(a,b)}\longrightarrow \text{Sym}_{a+b}, \quad f\mapsto \sigma_{a,b}.f
\]
where the action of 
\begin{align}  \label{eq:sigmaab}
    \sigma_{a,b} := \sum_{w\in (\mathfrak S_{a+b} /\mathfrak S_a \times \mathfrak S_b)_{\min} }w
\end{align}
%(also see \eqref{eq:minl})
is given by the action of the degenerate affine Hecke algebra on the polynomial ring in \eqref{actionofdeaffonpoly}; one sees that the image of $\sigma_{a,b}$ is in $\text{Sym}_{a+b}$. 
In fact, $\sigma_{a,b}.f=\frac{1}{a!b!} \sum_{w\in\mathfrak S_{a+b}}w. f$
since $f\in \text{Sym}_{a,b}$. Note that here the coefficient $\frac{1}{a!b!}$ will disappear when written as linear combination of symmetric polynomials as $f\in \text{Sym}_{a,b}$.  
Then for any $1\le j\le a+b-1$, we have $s_j(\sigma_{a,b}.f)=\sigma_{a,b}.f$ and hence  $\sigma_{a,b}.f\in \text{Sym}_{a+b}$.

Associated with $\wkdotr$, we define the map
 \begin{equation} 
 \label{equ:defofdotaction}
 \wkdotr: \text{Sym}_r\longrightarrow \text{Sym}_r, \quad f\mapsto x_1x_2\cdots x_r f.
 \end{equation}
In general, the map associated to $\wkdota$ with $r<a$ is determined by \eqref{equ:r=0andr=a},  i.e.,  the compositions of  maps for $
\begin{tikzpicture}[baseline = -.5mm,scale=.8,color=\clr]
	\draw[-,line width=1.5pt] (0.08,-.3) to (0.08,0.04);
	\draw[-,line width=1pt] (0.28,.4) to (0.08,0);
	\draw[-,line width=1pt] (-0.12,.4) to (0.08,0);
        \node at (-0.22,.5) {$\scriptstyle r$};
        \node at (0.36,.55) {$\scriptstyle b$};
        \node at (0.1,-.45){$\scriptstyle a$};
\end{tikzpicture}$,
$
\begin{tikzpicture}[baseline = 3pt, scale=0.4, color=\clr]
\draw[-,line width=1.2pt] (0,0) to[out=up, in=down] (0,1.4);
\draw(0,0.6) \bdot; 
\draw (-.6,0.6) node {$\scriptstyle \omega_r$};
\node at (0,-.3) {$\scriptstyle r$};
\draw[-,line width=1.2pt](0.4,0) to (0.4,1.4);
  \draw(0.4,-.3)node {$\scriptstyle b$};
\end{tikzpicture}
$ 
and 
$\begin{tikzpicture}[baseline = -.5mm,scale=.8,color=\clr]
	\draw[-,line width=1pt] (0.28,-.3) to (0.08,0.04);
	\draw[-,line width=1pt] (-0.12,-.3) to (0.08,0.04);
	\draw[-,line width=1.5pt] (0.08,.4) to (0.08,0);
        \node at (-0.13,-.45) {$\scriptstyle r$};
        \node at (0.35,-.4) {$\scriptstyle b$};\node at (0.08,.55){$\scriptstyle a$};\end{tikzpicture} $ 
        with $b=a-r$. That is, $\wkdota$ sends $f\in \text{Sym}_a$ to $\sigma_{r,b}.(x_1\ldots x_rf) $.

For the traverse-up, we define 
\begin{equation}
\label{equ:actionofmixedcrossiup}
\rightcrossing: \text{Sym}_a \longrightarrow \text{Sym}_a, \quad  f \mapsto f.
\end{equation}
For the traverse-down, we define 
\begin{equation}
\label{equ:actionofmixedcrossingdown}
\leftcrossing: \text{Sym}_a\longrightarrow \text{Sym}_a, \quad f\mapsto \prod_{1\le j\le a}(x_j-u)f.
\end{equation}
Finally, the linear maps associated with the crossings $\crossing$ are uniquely determined by those of splits and merges by \eqref{crossgen}. 
For example, the linear map for 
$ \begin{tikzpicture}[baseline = 2mm,scale=.7, color=\clr]
	\draw[-,thick] (0,0) to (.6,1);
	\draw[-,thick] (0,1) to (.6,0);
        \node at (0,-.2) {$\scriptstyle 1$};
        \node at (0.6,-0.2) {$\scriptstyle 1$};
\end{tikzpicture} $
is given by the action of $s_1$ in \eqref{actionofdeaffonpoly}. In general,
the action of $\crossing$ on $\text{Sym}_{(a,b)}$ is given by the action of the longest element
$w_{a,b}\in (\mathfrak S_{a+b}/\mathfrak S_{a}\times \mathfrak S_b)_{\min}$; more explicitly the permutation $w_{a,b}$ is specified as  
\begin{equation}\label{actionofcrossing}
    w_{a,b}(i)=\begin{cases}
        i+b, & \text{ if } 1\le i\le a,\\
        i-a, & \text{ if } a+1\le i\le a+b. 
    \end{cases}
\end{equation}

Denote by $\mathpzc{Vec}_\kk$ the category of free $\kk$-modules. 

\begin{theorem}
\label{thm:polyRep} %{pro:polynomilareofaffschur}
There is a $\kk$-linear monoidal functor 
\[
\mathcal F: \ASch \longrightarrow \mathpzc{Vec}_\kk, 
\]
which sends objects $a\in \N$ to $ \text{Sym}_a$, and  $u\in \kk$ to $ \kk$, and sends the generating morphisms to the linear maps in the same symbols.
\end{theorem}

\begin{proof}
We shall check that all the relations \eqref{webassoc}--\eqref{intergralballon} and \eqref{adaptorR}--\eqref{adaptermovemerge} are preserved under the functor $\mathcal F$. 

Let $\x_{(a+b)} :=\sum_{w\in \mathfrak S_{a+b}}w$ and $\x_{(a,b)}:=\sum_{w\in \mathfrak S_{a}\times \mathfrak S_b}w$. 
Note first that by definition the split $\splits$ is determined by sending  $\x_{(a+b)}f$ to $\x_{(a+b)}f$ and the merge $\merge$ is determined by sending $\x_{(a,b)}f$ to $\x_{(a+b)}f$, for $f\in \kk[x_1,\ldots, x_d] $. This is in the same way as those in Proposition~\ref{cor-isom-schur}. Hence, the relations in \eqref{webassoc}--\eqref{splitmerge} are satisfied by noting that these are also the defining relations for  $\W\cong\Sch$.
Moreover, the relation \eqref{intergralballon} is satisfied by definition. The relation \eqref{adaptermovemerge}
can be checked by using the fact that
\[
\sigma_{a,b}\prod_{1\le j\le a+b}(x_j-u) = \prod _{1\le j\le a+b} (x_j-u) \sigma_{a,b}.
\]

It remains to check the relations \eqref{dotmovecrossing}--\eqref{dotmovesplitss}, \eqref{adaptorR}--\eqref{adaptorL} and \eqref{adaptrermovecross}. The strategy here is to verify these relations over $\C$ (and then over $\Z$ and then $\kk$ by base change). Below we will freely quote the results from the equivalent but simpler presentation of $\ASch$ over $\C$ in Section~\ref{sec:C} (which does not rely on any results before that).

For \eqref{dotmovecrossing}, the case for $a=1=b$ follows since the actions of the crossing and the dot here are the same as the action of the degenerate affine Hecke algebra on the polynomial ring.
 For general $a,b$, let $\phi_\kk $ and $\varphi_\kk$ be the linear maps in $\Hom_{\kk}(\text{Sym}_{(a,b)}, \text{Sym}_{(b,a)})$ given by the left and right side of any equation in \eqref{dotmovecrossing}. Note that  $ \text{Sym}_{(a,b)}$ are free modules defined over $\Z$, i.e., 
 $_\kk \text{Sym}_{(a,b)}= \:_\Z \text{Sym}_{(a,b)}\otimes_\Z \kk$, 
 and also
 \begin{equation}
\phi_\kk(f\otimes 1)=\phi_\Z(f) \otimes _\Z \kk, \quad \varphi_\kk(f\otimes 1)=\varphi_\Z(f) \otimes_\Z \kk     
 \end{equation}
for any $f\in\:  _\Z \text{Sym}_{(a,b)}$.
We remark that each generators preserve $_\Z \text{Sym}$ and hence $\phi$ and $\psi$ preserves $_\Z \text{Sym}$. 
Thanks to the proof of \cite[Lemma~5.3]{SWweb}, which establishes \eqref{dotmovecrossing} over $\mathbb C$ using the affine Hecke relations, 
%\cite[Lemma \ref{definadotred}]{SWweb}, 
we have 
$c \phi_\Z(f)= c\varphi_\Z(f)$, 
for some non-zero integer $c$. This implies that 
$\phi_\Z(f)=\varphi_\Z(f)$ and hence 
$\phi_\kk(f\otimes 1)=\varphi_\kk(f\otimes 1)$ for any $f$. Thus, $\phi_\kk=\varphi_\kk$ and \eqref{dotmovecrossing} is satisfied.

The relations \eqref{dotmovesplitss}   follow from the proof of \cite[Lemma 5.4]{SWweb} and the same argument about $\Z$-forms of the symmetric  polynomial rings as in the preceding paragraph. 

 For \eqref{adaptorR}--\eqref{adaptorL}, we may also assume that $r=1$ by Lemma \ref{adpterLRa=1} and considering the $\Z[\tilde u]$-form of the polynomial rings, where $\Z[\tilde u]$ is the symmetric polynomial ring in one variable and $\tilde u$ acts on $\kk$ as $u$ via base change. Then \eqref{adaptorR}--\eqref{adaptorL} are satisfied by the definitions \eqref{equ:defofdotaction}-\eqref{equ:actionofmixedcrossingdown} above.
 
Finally, for \eqref{adaptrermovecross}, we may assume $a=1=b$ by Lemma \ref{chraterzero} and the same reasoning as above. Then the map defined by the left (respectively, the first summand on the right) hand side
sends $f$ to $(x_1-u)(s_1f)$ (respectively, $s_1(x_2-u)f$).
We have 
\[ 
(x_1-u)(s_1f)=[(x_1-u)s_1](f)=[s_1(x_2-u)+1](f)=s_1(x_2-u)f+f,
\]
and hence \eqref{adaptrermovecross} is satisfied.

This finishes the proof of the theorem.
\end{proof}

\begin{rem}
    A representation similar to the one in Theorem \ref{thm:polyRep} was constructed for the higher level affine $q$-Schur algebra in \cite{MS21}.
 In the sequel \cite{SSW25} it will be shown that the path algebras of the affine $q$-Schur category therein are isomorphic to the higher level affine $q$-Schur algebras in \cite{MS21}.
\end{rem}

\subsection{ An elementary diagram basis of $\ASch$}
In this subsection, we shall construct a spanning set of $\ASch$ and then prove that it is linearly independent by using the polynomial representation in \S\ref{polynomilrep}.

For $\ell\in \N$, an $(1+\ell)$-multicomposition of $m$ is an ordered $(1+\ell)$-tuple of compositions 
$\mu=(\mu^{(0)}, \mu^{(1)}, \ldots, \mu^{(\ell)})$ 
such that $\sum_{i=0}^{\ell}|\mu^{(i)}|=m$. 
%It is helpful to view the separating commas in $\mu$ colored in red. 

We call an $(1+\ell)$-multicomposition $\mu$ strict if all its components are strict compositions (where the partition $\emptyset$ is understood to be strict). Let $\Lambda_{\text{st}}^{1+\ell}(m)$ be the set of all strict $(1+\ell)$-multicompositions of $m$. Write 
\[
\Lambda_{\text{st}}^{1+\ell}:=\bigcup_{m\in \N}\Lambda_{\text{st}}^{1+\ell}(m).
\]
(The reason we use the superscript $1+\ell$ instead of $\ell$ in this section will be clear in Section~\ref{sec:basis-cycschur} when we introduce a subset $\Lambda_{\text{st}}^{\emptyset,\ell}$ of $\Lambda_{\text{st}}^{1+\ell}$ in \eqref{def:barlambdau} which is in bijection with $\Lambda_{\text{st}}^{\ell}$.)

By definition, an arbitrary object of $\ASch$ is of the form 
\begin{equation}
\label{equ:object-of-affineschur}
 \lambda^{(0)} \red{u_1} \lambda^{(1)} \red{u_2}  \ldots \lambda^{(\ell-1)} \red{u_{\ell}} \lambda^{(\ell)}    
\end{equation}
where $\lambda^{(i)}$ is a strict composition and $\mathbf u=(u_1,u_2,\ldots,u_{\ell})$ is a vector in $ \kk^{\ell}$, for some $\ell\in \N$. When $\ell=0$ should, the convention is that there is no $u_i$.
Thus, the set of objects of $\ASch$
can be identified with $ \bigcup_{\ell\in \N} \big(\Lambda_{\text{st}}^{1+\ell}\times \kk^{\ell} \big)$.

\begin{lemma}
\label{lem:Homzero}
Let $(\lambda,\mathbf u )\in \Lambda_{\text{st}}^{1+\ell}(m)\times \kk^{\ell}$ and $ (\mu,\mathbf u')\in \Lambda_{\text{st}}^{1+\ell'}(m')\times \kk^{\ell'}$.
Then $\Hom_{\ASch}((\lambda,\mathbf u ),  (\mu,\mathbf u')) =0$ unless
$m=m',  \ell=\ell',$ and $\mathbf u=\mathbf u'.
$
\end{lemma}

\begin{proof}
Follows from the forms of the generating morphisms of $\ASch$ given in \eqref{generator-affschur}--\eqref{rightleftcrossinggen}.
\end{proof}

Thus, it suffices to consider the morphisms between the objects of the form \eqref{equ:object-of-affineschur}
for given $m, \ell$ and $\mathbf u\in \kk^{\ell}$.
Below we shall always fix some $\mathbf u$; in this way we identify an object in \eqref{equ:object-of-affineschur} without ambiguity with
\begin{equation}
\label{equ:object-comma}
 \lambda^{(0)} \red{,} \lambda^{(1)} \red{,}  \ldots \lambda^{(\ell-1)} \red{,} \lambda^{(\ell)}   
\end{equation}
which can be viewed as an $(1+\ell)$-multicomposition $\lambda \in \Lambda_{\text{st}}^{1+\ell}(m)$. 

We define a forgetful map
\begin{align*}
    \Lambda_{\text{st}}^{1+\ell}(m) \longrightarrow \Lambda_{\text{st}}(m),
    \qquad \mu \mapsto \overline{\mu},
\end{align*}
by viewing it as a strict $1$-multicomposition (by omitting the empty components, i.e., those $\mu^{(i)}=\emptyset$).

We shall refer to the diagrams obtained by compositions and tensor products of \eqref{generator-affschur}--\eqref{rightleftcrossinggen} (and identity morphisms) {\em red strand dotted web diagrams}. Recall that a composition of generating morphisms in the affine web category $\AW$ is represented by a dotted web diagram. Thus, any morphism of $\Hom_{\ASch}(\lambda,\mu)$ is a linear combination of  red strand dotted web diagrams,  which
is obtained from  a dotted web diagram from $\bar \lambda$ to $\bar \mu$ by adding $\ell$ red strands to connect $u_i$ at the bottom to $u_i$ at the top, for $1\le i\le \ell$, 
such that
\begin{enumerate}
    \item there is no crossings between red strands, and
    \item a red strand intersects with the dotted web diagram only at toe segments.
\end{enumerate} 
(Here a toe segment of a dotted web diagram means a connected component of the diagram after removing all the intersection and dots.) 

\begin{example}
Let $\lambda=((\emptyset),(9),(7))$ and $\mu=((5),(5),(6))$. The following are two  red strand dotted web diagrams from $\mu$ to $\lambda$: 
\begin{align}
    \label{ex-of-dotchickenfootd-red}
\begin{tikzpicture}[anchorbase,scale=1.6,color=\clr]
\draw[-,line width=.6mm] (.212,.9) to (.212,.8);
\draw(.212,1) node{$\scriptstyle 9$};
\draw(1,1) node{$\scriptstyle 7$};
\draw(-1.1,-.55) node{$\scriptstyle 5$};
\draw(.4,-.6) node{$\scriptstyle 5$};
\draw(1.4,-.6) node{$\scriptstyle 6$};
\draw[-,line width=.75mm] (1,.9) to (1,.8);
\draw[-,line width=.15mm] (-1,-.396) to (.2,.8);
\draw[-,line width=.75mm](-1,-.39) to (-1.1,-.45);
\draw[-,line width=.15mm]  (.2,.8)to (.4,-.4);
\draw[-,line width=.75mm](.4,-.4) to (.4,-.5);
\draw[-,line width=.15mm]  (.2,.8)to (1.4,-.4);
\draw[-,line width=.75mm](1.4,-.4) to (1.4,-.5);
\draw[-,line width=.15mm] (1,.8) to (-1,-.39);
\draw[-,line width=.15mm] (1,.8) to (.4,-.39);
\draw[-,line width=.15mm] (1,.8) to (1.4,-.39);
\node at (-0.3,0.5) {$\scriptstyle 2$}; 
\draw (-.5,0.13)\bdot;
\node at (-.7,0.13) {$\scriptstyle  \nu^1$};
\draw (-.3,0)\bdot; \node at (-.3,-0.13) {$\scriptstyle  \nu^2$};
\draw (.35,-0.1)\bdot;\node at (.2,-0.13) {$\scriptstyle  \nu^3$};
\draw (.55,-0.1)\bdot;\node at (.7,-0.2) {$\scriptstyle  \nu^4$};
\draw (1.1,-0.1)\bdot;\node at (1,-0.2) {$\scriptstyle  \nu^5$};
\draw (1.3,-0.1)\bdot;\node at (1.5,-0.13) {$\scriptstyle  \nu^6$};
\node at (1.2,0.5) {$\scriptstyle 2$};
\node at (0.15,0.5) {$\scriptstyle 3$};
\node at (0.1,0.13) {$\scriptstyle 3$};
\node at (0.4,0.7) {$\scriptstyle 4$};
\node at (.9,.4) {$\scriptstyle 2$};
\draw[-,line width=1pt, color=\cred](.6,.8) to (0.6,.4);
\draw[-,line width=1pt, color=\cred](.6,.4) to[out=down,in=135] (1,-.4);
\node at (0.6, .9){$\red{\scriptstyle u_2}$};
\node at (1,-.5){$\red{\scriptstyle u_2}$};
\draw[-,line width=1pt, color=\cred](0,.8) to (0,-.4);
\node at (0, .9){$\red{\scriptstyle u_1}$};
\node at (0,-.5){$\red{\scriptstyle u_1}$};
\end{tikzpicture}
, \quad 
\begin{tikzpicture}[anchorbase,scale=1.6,color=\clr]
\draw[-,line width=.6mm] (.212,.9) to (.212,.8);
\draw(.212,1) node{$\scriptstyle 9$};
\draw(1,1) node{$\scriptstyle 7$};
\draw(-1.1,-.55) node{$\scriptstyle 5$};
\draw(.4,-.6) node{$\scriptstyle 5$};
\draw(1.4,-.6) node{$\scriptstyle 6$};
\draw[-,line width=.75mm] (1,.9) to (1,.8);
\draw[-,line width=.15mm] (-1,-.396) to (.2,.8);
\draw[-,line width=.75mm](-1,-.39) to (-1.1,-.45);
\draw[-,line width=.15mm]  (.2,.8)to (.4,-.4);
\draw[-,line width=.75mm](.4,-.4) to (.4,-.5);
\draw[-,line width=.15mm]  (.2,.8)to (1.4,-.4);
\draw[-,line width=.75mm](1.4,-.4) to (1.4,-.5);
\draw[-,line width=.15mm] (1,.8) to (-1,-.39);
\draw[-,line width=.15mm] (1,.8) to (.4,-.39);
\draw[-,line width=.15mm] (1,.8) to (1.4,-.39);
\node at (-0.3,0.5) {$\scriptstyle 2$}; 
\draw (-.5,0.13)\bdot;
\node at (-.7,0.13) {$\scriptstyle  \nu^1$};
\draw (-.3,0)\bdot; \node at (-.3,-0.13) {$\scriptstyle  \nu^2$};
\draw (.35,-0.1)\bdot;\node at (.2,-0.13) {$\scriptstyle  \nu^3$};
\draw (.55,-0.1)\bdot;\node at (.7,-0.2) {$\scriptstyle  \nu^4$};
\draw (1.1,-0.1)\bdot;\node at (1,-0.2) {$\scriptstyle  \nu^5$};
\draw (1.3,-0.1)\bdot;\node at (1.5,-0.13) {$\scriptstyle  \nu^6$};
\node at (1.2,0.5) {$\scriptstyle 2$};
\node at (0.15,0.5) {$\scriptstyle 3$};
\node at (0.1,0.13) {$\scriptstyle 3$};
\node at (0.4,0.7) {$\scriptstyle 4$};
\node at (.9,.4) {$\scriptstyle 2$};
\draw[-,line width=1pt, color=\cred](.6,.8) to [out=down, in =135] (1,.4);
\draw[-,line width=1pt, color=\cred](1,.4) to[out=-45,in=90] (1.2,-.4);
\node at (0.6, .9){$\red{\scriptstyle u_2}$};
\node at (1.2,-.5){$\red{\scriptstyle u_2}$};
\draw[-,line width=1pt, color=\cred](0,.8) to (0,-.4);
\node at (0, .9){$\red{\scriptstyle u_1}$};
\node at (0,-.5){$\red{\scriptstyle u_1}$};
\end{tikzpicture}.
\end{align} 
\end{example}

\begin{definition}
  \label{def:degree}
Letting the degree of $\wkdota$ be $r$, the degree of  $\merge, \splits$, $\crossing$ and a traverse-up $\rightcrossing$ be $0$, and the degree of a traverse-down $\leftcrossing$ be $a$, we define the {\em degree} of a red strand dotted web diagram as the sum of the degrees of its local components.
\end{definition}
The degrees for the generating morphisms in Definition \ref{def:degree} are motivated  by  the degrees of the corresponding operators in the definition of the polynomial representation of $\ASch$ in \S\ref{polynomilrep} when regarding the polynomial ring as graded ring such that  each variable is of  degree 1. 

Let $\Hom_{\ASch}(\lambda,\mu)_{\le k}$ (respectively, $\Hom_{\ASch}(\lambda,\mu)_{< k}$)
be the $\kk$-span of all red strand dotted web diagram  with degree $\le k$ (respectively, $<k$).
For any  $D$ and $D'$  in  $\Hom_{\ASch}(\lambda,\mu)_{\le k}$, we write 
\begin{align}  \label{modLOT}
D\equiv D' \quad \text{ if } D=D'  \mod \Hom_{\ASch}(\lambda,\mu)_{< k}.
\end{align}

It follows from the defining relations \eqref{dotmovecrossing}, \eqref{dotmovesplitss}, 
and \eqref{adaptrermovecross} that we may move dots through crossings, splits and merges, and pull red lines through crossings freely up to lower degree terms, as formulated in the following lemma, where (1)--(3) are already appeared in \cite[Lemma 2.12]{SWweb}. This fact will be used frequently without any explicit reference.

\begin{lemma}
\label{dotmovefreely}
In $\ASch$, the following $\equiv$-relations hold:
 \begin{enumerate}
  \item 
    $\begin{tikzpicture}[baseline = 7.5pt, scale=0.35, color=\clr]
\draw[-,line width=1.2pt] (0,-.2) to  (1,2.2);
\draw[-,line width=1.2pt] (1,-.2) to  (0,2.2);
\draw(0.2,1.6)\bdot;
\draw(-.5,1.6) node {$\scriptstyle \omega_r$};
\node at (0, -.5) {$\scriptstyle a$};
\node at (1, -.5) {$\scriptstyle b$};
\end{tikzpicture}
~\equiv~ 
\begin{tikzpicture}[baseline = 7.5pt, scale=0.35, color=\clr]
\draw[-, line width=1.2pt] (0,-.2) to (1,2.2);
\draw[-,line width=1.2pt] (1,-.2) to (0,2.2);
\draw(.8,0.3)\bdot;
\draw(1.55,0.3) node {$\scriptstyle \omega_r$};
\node at (0, -.5) {$\scriptstyle a$};
\node at (1, -.5) {$\scriptstyle b$};
\end{tikzpicture}, 
\quad 
\begin{tikzpicture}[baseline = 7.5pt, scale=0.35, color=\clr]
\draw[-, line width=1.2pt] (0,-.2) to (1,2.2);
\draw[-,line width=1.2pt] (1,-.2) to(0,2.2);
\draw(.2,0.2)\bdot;
\draw(-.5,.2)node {$\scriptstyle \omega_r$};
\node at (0, -.5) {$\scriptstyle b$};
\node at (1, -.5) {$\scriptstyle a$};
\end{tikzpicture}
~\equiv~
\begin{tikzpicture}[baseline = 7.5pt, scale=0.35, color=\clr]
\draw[-,line width=1.2pt] (0,-.2) to  (1,2.2);
\draw[-,line width=1.2pt] (1,-.2) to  (0,2.2);
\draw(0.8,1.6)\bdot;
\draw(1.6,1.6)node {$\scriptstyle \omega_r$};
\node at (0, -.5) {$\scriptstyle b$};
\node at (1, -.5) {$\scriptstyle a$};
\end{tikzpicture}
$;
\item
$
\begin{tikzpicture}[baseline = -.5mm,scale=.7,color=\clr]
\draw[-,line width=1.5pt] (0.08,-.5) to (0.08,0.04);
\draw[-,line width=1pt] (0.34,.5) to (0.08,0);
\draw[-,line width=1pt] (-0.2,.5) to (0.08,0);
\node at (-0.22,.6) {$\scriptstyle a$};
\node at (0.36,.65) {$\scriptstyle b$};
\draw (0.08,-.2) \bdot;
\draw (0.55,-.2) node{$\scriptstyle \omega_r$};
\end{tikzpicture} 
\equiv~
\sum\limits_{c+d=r} ~
\begin{tikzpicture}[baseline = -.5mm,scale=.7,color=\clr]
\draw[-,line width=1.5pt] (0.08,-.5) to (0.08,0.04);
\draw[-,line width=1pt] (0.34,.5) to (0.08,0);
\draw[-,line width=1pt] (-0.2,.5) to (0.08,0);
\node at (-0.22,.6) {$\scriptstyle a$};
\node at (0.36,.65) {$\scriptstyle b$};
\draw (-.05,.24) \bdot;
\draw (-0.4,.2) node{$\scriptstyle \omega_{c}$};
\draw (0.6,.2) node{$\scriptstyle \omega_{d}$};
\draw (.22,.24) \bdot;
\end{tikzpicture}~
$;
\item 
$
\begin{tikzpicture}[baseline = -.5mm, scale=.7, color=\clr]
\draw[-,line width=1pt] (0.3,-.5) to (0.08,0.04);
\draw[-,line width=1pt] (-0.2,-.5) to (0.08,0.04);
\draw[-,line width=1.5pt] (0.08,.6) to (0.08,0);
\node at (-0.22,-.6) {$\scriptstyle a$};
\node at (0.35,-.6) {$\scriptstyle b$};
\draw (0.08,.2) \bdot;
\draw (0.5,.2) node{$\scriptstyle \omega_r$};  
\end{tikzpicture}
 ~\equiv~
\sum\limits_{c+d=r}~
\begin{tikzpicture}[baseline = -.5mm,scale=.7, color=\clr]
\draw[-,line width=1pt] (0.3,-.5) to (0.08,0.04);
\draw[-,line width=1pt] (-0.2,-.5) to (0.08,0.04);
\draw[-,line width=1.5pt] (0.08,.6) to (0.08,0);
\node at (-0.22,-.6) {$\scriptstyle a$};
\node at (0.35,-.6) {$\scriptstyle b$};
\draw (-.08,-.3) \bdot; \draw (.22,-.3) \bdot;
\draw (-0.5,-.3) node{$\scriptstyle \omega_{c}$};
\draw (0.7,-.3) node{$\scriptstyle \omega_{d}$};
\end{tikzpicture}
$;
\item 
  $\begin{tikzpicture}[baseline = -1mm,scale=0.6,color=\clr]
	\draw[-,thick] (0.2,-.6) to (0.2,.6);
	\draw[-,thick] (-0.2,-.6) to (-0.2,.6);
        \node at (0.2,-.75) {$\scriptstyle b$};
        \node at (-0.2,-.75) {$\scriptstyle a$};
        \draw (.2, 0) \bdot;
        \node at (0.6,0) {$\scriptstyle \omega_r$};
\end{tikzpicture}
\equiv
\begin{tikzpicture}[baseline = -1mm,scale=0.6,color=\clr]
	\draw[-,thick] (0.28,0) to[out=90,in=-90] (-0.28,.6);
	\draw[-,thick] (-0.28,0) to[out=90,in=-90] (0.28,.6);
	\draw[-,thick] (0.28,-.6) to[out=90,in=-90] (-0.28,0);
	\draw[-,thick] (-0.28,-.6) to[out=90,in=-90] (0.28,0);
        \node at (0.3,-.75) {$\scriptstyle b$};
        \draw (-.28, 0) \bdot;
        \node at (-0.7,0) {$\scriptstyle \omega_r$};
        \node at (-0.3,-.75) {$\scriptstyle a$};
\end{tikzpicture}$;
\item 
$
\begin{tikzpicture}[baseline = 7.5pt, scale=0.35, color=\clr]
\draw[-,line width=1pt,color=\cred](.8,2) to[out=down,in=90] (0.1,1) to [out=down,in=up] (.8,-.4);
\draw (.8,-.65) node{$\scriptstyle \red{u}$};
\draw[-,line width=1.2pt] (0,-.35) to (1.5,2);
\draw[-,line width=1.2pt](0,2) to  (1.5,-.35);
\node at (0, -.6) {$\scriptstyle a$};
\node at (1.5, -.6) {$\scriptstyle b$};
\end{tikzpicture}
~\equiv~
\begin{tikzpicture}[baseline = 7.5pt, scale=0.35,color=\clr]
\draw[-,line width=1pt,color=\cred](.8,2) to[out=down,in=90] (1.5,1) to [out=down,in=up] (.8,-.4);
\draw (.8,-.65) node{$\scriptstyle \red{u}$};
\draw[-,line width=1.2pt] (0,-.35) to (1.5,2);
\draw[-,line width=1.2pt](0,2) to  (1.5,-.35);
\node at (0, -.6) {$\scriptstyle a$};
\node at (1.5, -.6) {$\scriptstyle b$};
\end{tikzpicture}
$;
\item 
$\begin{tikzpicture}[baseline = 10pt, scale=0.35, color=\clr]
\draw[-,line width=1pt,color=\cred](0,-.2)to (0,2.6);
\draw (0,-.4) node{$\scriptstyle \red{u}$};
\draw[-,line width= 1.2pt](-.5,-.1) to[out=45, in=down] (.5,1.4);
\draw[-,line width= 1.2pt](.5,1.4) to[out=up, in=300] (-.5,2.4); 
\draw (-.5,-0.4) node{$\scriptstyle {r}$};
\end{tikzpicture}
\equiv~
 \begin{tikzpicture}[baseline = -1mm,scale=.7,color=\clr]
\draw[-,line width=1.2pt] (0.08,-.6) to (0.08,.5);
\node at (.08,-.8) {$\scriptstyle r$};
\draw(0.08,0) \bdot;
%\draw(-.5,0)node {$\scriptstyle g_r(u)$};
\draw[-,line width=1pt,color=\cred](.5,-.6)to (.5,.5);
\draw (.5,-.8) node{$\scriptstyle \red{u}$};
\end{tikzpicture}, \quad
 \begin{tikzpicture}[baseline = 10pt, scale=0.35, color=\clr]
\draw[-,line width=1pt,color=\cred](0,-.6)to (0,2.4);
\draw (0,-.9) node{$\scriptstyle \red{u}$};
\draw[-,line width= 1.2pt](.5,-.5) to[out=135, in=down] (-.5,1);
\draw[-,line width= 1.2pt](-.5,1) to[out=up, in=270] (.5,2.3);
\draw (.5,-0.9) node{$\scriptstyle {r}$};
\end{tikzpicture}
\equiv~
 \begin{tikzpicture}[baseline = -1mm,scale=.7,color=\clr]
\draw[-,line width=1.2pt] (0.08,-.6) to (0.08,.5);
\node at (.08,-.8) {$\scriptstyle r$};
\draw(0.08,0) \bdot;
%\draw(.7,0)node {$\scriptstyle g_{r}(u)$};
\draw[-,line width=1pt,color=\cred](-.4,-.6)to (-.4,.5);
\draw (-.4,-.8) node{$\scriptstyle \red{u}$};
\end{tikzpicture}$.
\end{enumerate}  
\end{lemma}

Recall from \eqref{dottedreduced} and Definition \ref{def:diagram} the set $\PMat_{-,-}$ of elementary (chicken foot) diagrams.

\begin{definition}
Suppose that $\lambda,\mu \in \Lambda_{\text{st}}^{1+\ell}(m)$.
A red strand dotted web diagram $\D$ from $\mu$ to $\lambda$ is  called an elementary diagram  if 
 \begin{enumerate}
     \item 
     the diagram $\overline \D $ obtained from $\D$ by removing the red strands is an elementary chicken foot diagram in  $\PMat_{\bar\lambda,\bar\mu}$, and 
     \item 
     there is at most one intersection between each red strand and any other strand.
 \end{enumerate}   
In this case, we call $\D$ a $\la \times\mu$-ornamentation of $\overline \D$. 
\end{definition}

There are usually several $\la \times\mu$-ornamentations for a given elementary diagram in $ \PMat_{\bar\lambda,\bar\mu}$, but they are all equal modulo lower degree terms in $\Hom_{\ASch}(\mu,\lambda)$ thanks to
Lemma \ref{dotmovefreely}(4), 
Lemma \ref{adamovecrossings}, and \eqref{dotmoveadaptor}. For example, the two red strand dotted web diagrams in \eqref{ex-of-dotchickenfootd-red} are two different ornamentations of the same elementary chicken foot diagram. We shall fix once for all a $\la \times\mu$-ornamentation for each elementary chicken foot diagram in $\PMat_{\bar\lambda,\bar\mu}$.

For $\lambda,\mu \in \Lambda_{\text{st}}^{1+\ell}(m)$, we introduce
\begin{align}  \label{PMatblock}
\PMat_{\lambda,\mu} = \Big\{\big(A=(A_{pq})_{0\le p,q\le \ell},  P=(P_{pq})_{0\le p,q\le \ell} \big)  \mid (A, P) \in  \PMat_{\overline{\lambda},\overline{\mu}}  \Big\}. 
\end{align}
This can be viewed as a $(1+\ell) \times (1+\ell)$-block (the blocks corresponding to the empty components are also counted here) matrix generalization (or a multicomposition generalization) of $\PMat_{-,-}$ defined earlier in \eqref{dottedreduced}, and indeed \eqref{PMatblock} for $\ell=0$ reduces to \eqref{dottedreduced}. We define a forgetful map (which is indeed a bijection) 
\begin{align}  \label{bijPM}
    \PMat_{\la,\mu} \longrightarrow \PMat_{\overline\lambda, \overline\mu},
    \qquad (A,P) \mapsto (A,P),
    %(\overline{A},\overline{P}),
\end{align}
by viewing a block matrix as a standard matrix; the map is bijective as there is a unique way to view a matrix in $\PMat_{\overline\lambda, \overline\mu}$ as a block matrix in $\PMat_{\lambda,\mu}$ where the blocks are determined by the fixed $(1+\ell)$-multicompositions. Note the notation   $\PMat_{\la,\mu}$ has different meaning as in \eqref{dottedreduced} since it is now defined for multi-compositions. 

Let $\lambda,\mu \in \Lambda_{\text{st}}^{1+\ell}(m)$. Just as an elementary chicken foot diagram $\overline{D}$ from $\overline\mu$ to $\overline\la$ is encoded in $(A,P) \in\PMat_{\overline\lambda, \overline\mu}$, an elementary diagram $D$ from $\mu$ to $\la$ which is a (fixed) ornamentation of $\overline{D}$ is encoded by $(A,P) \in\PMat_{\la,\mu}$ via the identification \eqref{bijPM}. In this way, the set of (fixed) $\la \times\mu$-ornamented elementary chicken foot diagrams in $\PMat_{\bar\lambda,\bar\mu}$ is identified with $\PMat_{\lambda,\mu}.$

\begin{example}
Let $\lambda=((9),(7))$ and $\mu=((5,5),(6))$. 
Suppose $(A,P)\in \Par\Mat_{\lambda,\mu} $ with
\[A=\left[\begin{array}{cc|c}
2 & 3 & 4 \\
\hline
3& 2 & 2 \\  
\end{array}\right],\quad  P=  \left[\begin{array}{cc|c}
\nu^1 & \nu^3 & \nu^5 \\
\hline
\nu^2& \nu^4 & \nu^6 \\  
\end{array}\right].\]
The following are two  ornamentations of $(A,P)$ which represent the same morphism up to lower degree terms:
\begin{align}
    \label{ex-of-dotchickenfootd-red-matrix}
\begin{tikzpicture}[anchorbase,scale=1.5,color=\clr]
\draw[-,line width=.6mm] (.212,.9) to (.212,.8);
\draw(.212,1) node{$\scriptstyle 9$};
\draw(1,1) node{$\scriptstyle 7$};
\draw(-1.1,-.55) node{$\scriptstyle 5$};
\draw(.4,-.6) node{$\scriptstyle 5$};
\draw(1.4,-.6) node{$\scriptstyle 6$};
\draw[-,line width=.75mm] (1,.9) to (1,.8);
\draw[-,line width=.15mm] (-1,-.396) to (.2,.8);
\draw[-,line width=.75mm](-1,-.39) to (-1.1,-.45);
\draw[-,line width=.15mm]  (.2,.8)to (.4,-.4);
\draw[-,line width=.75mm](.4,-.4) to (.4,-.5);
\draw[-,line width=.15mm]  (.2,.8)to (1.4,-.4);
\draw[-,line width=.75mm](1.4,-.4) to (1.4,-.5);
\draw[-,line width=.15mm] (1,.8) to (-1,-.39);
\draw[-,line width=.15mm] (1,.8) to (.4,-.39);
\draw[-,line width=.15mm] (1,.8) to (1.4,-.39);
\node at (-0.3,0.5) {$\scriptstyle 2$}; 
\draw (-.5,0.13)\bdot;
\node at (-.7,0.13) {$\scriptstyle  \nu^1$};
\draw (-.3,0)\bdot; \node at (-.3,-0.13) {$\scriptstyle  \nu^2$};
\draw (.35,-0.1)\bdot;\node at (.2,-0.13) {$\scriptstyle  \nu^3$};
\draw (.55,-0.1)\bdot;\node at (.7,-0.2) {$\scriptstyle  \nu^4$};
\draw (1.1,-0.1)\bdot;\node at (1,-0.2) {$\scriptstyle  \nu^5$};
\draw (1.3,-0.1)\bdot;\node at (1.5,-0.13) {$\scriptstyle  \nu^6$};
\node at (1.2,0.5) {$\scriptstyle 2$};
\node at (0.15,0.5) {$\scriptstyle 3$};
\node at (0.1,0.13) {$\scriptstyle 3$};
\node at (0.4,0.7) {$\scriptstyle 4$};
\node at (.9,.4) {$\scriptstyle 2$};
\draw[-,line width=1pt, color=\cred](.6,.8) to (0.6,.4);
\draw[-,line width=1pt, color=\cred](.6,.4) to[out=down,in=135] (1,-.4);
\node at (0.6, .9){$\red{\scriptstyle u_1}$};
\node at (1,-.5){$\red{\scriptstyle u_1}$};
%\draw[-,line width=1pt, color=\cred](0,.8) to (0,-.4);
%\node at (0, .9){$\red{\scriptstyle u_1}$};
%\node at (0,-.5){$\red{\scriptstyle u_1}$};
\end{tikzpicture}
, \quad 
\begin{tikzpicture}[anchorbase,scale=1.5,color=\clr]
\draw[-,line width=.6mm] (.212,.9) to (.212,.8);
\draw(.212,1) node{$\scriptstyle 9$};
\draw(1,1) node{$\scriptstyle 7$};
\draw(-1.1,-.55) node{$\scriptstyle 5$};
\draw(.4,-.6) node{$\scriptstyle 5$};
\draw(1.4,-.6) node{$\scriptstyle 6$};
\draw[-,line width=.75mm] (1,.9) to (1,.8);
\draw[-,line width=.15mm] (-1,-.396) to (.2,.8);
\draw[-,line width=.75mm](-1,-.39) to (-1.1,-.45);
\draw[-,line width=.15mm]  (.2,.8)to (.4,-.4);
\draw[-,line width=.75mm](.4,-.4) to (.4,-.5);
\draw[-,line width=.15mm]  (.2,.8)to (1.4,-.4);
\draw[-,line width=.75mm](1.4,-.4) to (1.4,-.5);
\draw[-,line width=.15mm] (1,.8) to (-1,-.39);
\draw[-,line width=.15mm] (1,.8) to (.4,-.39);
\draw[-,line width=.15mm] (1,.8) to (1.4,-.39);
\node at (-0.3,0.5) {$\scriptstyle 2$}; 
\draw (-.5,0.13)\bdot;
\node at (-.7,0.13) {$\scriptstyle  \nu^1$};
\draw (-.3,0)\bdot; \node at (-.3,-0.13) {$\scriptstyle  \nu^2$};
\draw (.35,-0.1)\bdot;\node at (.2,-0.13) {$\scriptstyle  \nu^3$};
\draw (.55,-0.1)\bdot;\node at (.7,-0.2) {$\scriptstyle  \nu^4$};
\draw (1.1,-0.1)\bdot;\node at (1,-0.2) {$\scriptstyle  \nu^5$};
\draw (1.3,-0.1)\bdot;\node at (1.5,-0.13) {$\scriptstyle  \nu^6$};
\node at (1.2,0.5) {$\scriptstyle 2$};
\node at (0.15,0.5) {$\scriptstyle 3$};
\node at (0.1,0.13) {$\scriptstyle 3$};
\node at (0.4,0.7) {$\scriptstyle 4$};
\node at (.9,.4) {$\scriptstyle 2$};
\draw[-,line width=1pt, color=\cred](.6,.8) to [out=down, in =135] (1,.4);
\draw[-,line width=1pt, color=\cred](1,.4) to[out=-45,in=90] (1.2,-.4);
\node at (0.6, .9){$\red{\scriptstyle u_1}$};
\node at (1.2,-.5){$\red{\scriptstyle u_1}$};
%\draw[-,line width=1pt, color=\cred](0,.8) to (0,-.4);
%\node at (0, .9){$\red{\scriptstyle u_1}$};
%\node at (0,-.5){$\red{\scriptstyle u_1}$};
\end{tikzpicture}.
\end{align} 
The forgetful map in \eqref{bijPM} sends $(A,P)$
to 
\[A=\left[\begin{array}{ccc}
2 & 3 & 4 \\
3& 2 & 2 \\  
\end{array}\right],\quad  P=  \left[\begin{array}{ccc}
\nu^1 & \nu^3 & \nu^5 \\
\nu^2& \nu^4 & \nu^6 \\  
\end{array}\right].\]
\end{example}

\begin{theorem}
  \label{thm:basisASchur}
$\PMat_{\lambda,\mu}$ forms a basis for the Hom-space $\Hom_{\ASch}(\mu,\lambda)$, for any $\lambda,\mu \in \Lambda_{\text{st}}^{1+\ell}(m)$.    
\end{theorem}

\begin{proof}
First we prove that $\PMat_{\lambda,\mu}$ spans $\Hom_{\ASch}(\mu,\lambda)$.

It suffices to show the claim  that $fg$ can be written as a linear combination of elementary diagrams up to lower degree terms (i.e, with degree $<\deg f+\deg g$)  for an arbitrary elementary diagram $g$ and a generating morphism $f$ of the following five types: 
\[ 
1_* \wkdota 1_*,\qquad  1_*\merge 1_*,\qquad  1_*\splits 1_*,   \qquad  1_*\rightcrossing1_*, \qquad   1_*\leftcrossing1_*,
\]
where $1_*$ represents suitable identity morphisms.
   
We proceed by induction on the degree $k$ of  $g\in \Hom_{\ASch}(\mu,\lambda)_{\le k}$.
Suppose $k=0$. Then any diagram with degree $0$ is generated by splits, merges and traverse-ups.
Since splits and merges can pass through red strands freely by \eqref{adaptermovemerge}, the claim holds by the argument in \cite[Lemma 4.9]{BEEO}.
Note that all ornamentations of a given elementary  chicken foot diagram represent the same morphism by \eqref{adaptermovemerge} and Lemma~ \ref{adamovecrossings}
since there are no traverse-down.
It is clear that the left multiplication (=composing vertically from the top) of the traverse-up is still an ornamentation of some degree zero elementary diagram and hence the claim holds for $k=0$.

Suppose $k>0$. Even though two different ornamentations of a given elementary chicken foot diagram may not represent the same morphism, they are equal up to some lower  degree diagrams by \eqref{adaptrermovecross}. So we shall not distinguish ornamentations of a given elementary chicken foot diagram below. In addition to splits, merges, dots, and crossings in Lemma \ref{adamovecrossings}, all other types of crossings (i.e., in \eqref{adaptrermovecross}) can also pass through red strands freely (up to lower degree terms). Keeping these in mind, we see that the claim holds for $ 1_* \wkdota 1_*,\quad  1_*\merge 1_*,\quad  1_*\splits 1_*$ by the same arguments as in the proof of \cite[Proposition~3.6]{SWweb}. % \cite[Proposition \ref{prospaningset}]{SWweb}.

It is also clear that the claim holds for the traverse-up by \eqref{adaptermovemerge}. 

It remains to consider the traverse-down. Suppose $f$ is a traverse-down joining to the $i$th vertex of an $r$-fold merge at the top of $g$. Then use \eqref{adaptermovemerge} to pass the merge through the red strand so that the resulting diagram without red strands is an elementary diagram $B$. For example, suppose $g$ is the first diagram in \eqref{ex-of-dotchickenfootd-red} and $f$ is the traverse-down joining the first vertex of $g$, then 
\[ 
  fg~=~ 
   \begin{tikzpicture}[anchorbase,scale=1.6,color=\clr]
\draw[-,line width=.6mm] (.212,.9) to (.212,.8);
\draw[-,line width=.75mm] (1,.9) to (1,.8);
\draw[-,line width=.15mm] (-1,-.396) to (.2,.8);
\draw[-,line width=.75mm](-1,-.39) to (-1.1,-.45);
\draw[-,line width=.15mm]  (.2,.8)to (.4,-.4);
\draw[-,line width=.75mm](.4,-.4) to (.4,-.5);
\draw[-,line width=.15mm]  (.2,.8)to (1.4,-.4);
\draw[-,line width=.75mm](1.4,-.4) to (1.4,-.5);
\draw[-,line width=.15mm] (1,.8) to (-1,-.39);
\draw[-,line width=.15mm] (1,.8) to (.4,-.39);
\draw[-,line width=.15mm] (1,.8) to (1.4,-.39);
\node at (-0.3,0.5) {$\scriptstyle 2$}; 
\draw (-.5,0.13)\bdot;
\node at (-.7,0.13) {$\scriptstyle  \nu^1$};
\draw (-.3,0)\bdot; \node at (-.3,-0.13) {$\scriptstyle  \nu^2$};
\draw (.35,-0.1)\bdot;\node at (.2,-0.13) {$\scriptstyle  \nu^3$};
\draw (.55,-0.1)\bdot;\node at (.64,-0.2) {$\scriptstyle  \nu^4$};
\draw (1.1,-0.1)\bdot;\node at (1.04,-0.24) {$\scriptstyle  \nu^5$};
\draw (1.3,-0.1)\bdot;\node at (1.5,-0.13) {$\scriptstyle  \nu^6$};
\node at (1.2,0.5) {$\scriptstyle 2$};
\node at (0.15,0.5) {$\scriptstyle 3$};
\node at (0.1,0.13) {$\scriptstyle 3$};
\node at (0.4,0.7) {$\scriptstyle 4$};
\node at (.9,.4) {$\scriptstyle 2$};
\draw[-,line width=1pt, color=\cred](.6,.8) to (0.6,.4);
\draw[-,line width=1pt, color=\cred](.6,.4) to[out=down,in=135] (1,-.4);
\node at (0.6, .9){$\red{\scriptstyle u_2}$};
\node at (1,-.5){$\red{\scriptstyle u_2}$};
\draw[-,line width=1pt, color=\cred](0,.8) to (0,-.4);
\draw[-,line width=1pt, color=\cred](0,.8) to (0.3,1.1);
\draw[-,line width=1.5pt] (0.2,.9) to (0, 1.1);
\node at (0.3, 1.2){$\red{\scriptstyle u_1}$};
\node at (0,-.5){$\red{\scriptstyle u_1}$};
\end{tikzpicture}
~=~
   \begin{tikzpicture}[anchorbase,scale=1.6,color=\clr]
\draw[-,line width=.6mm] (0,1.1) to (0,1);
\draw[-,line width=.75mm] (1,.9) to (1,.8);
\draw[-,line width=.15mm] (-1,-.396) to (0,1);
\draw[-,line width=.75mm](-1,-.39) to (-1.1,-.45);
\draw[-,line width=.15mm]  (0,1)to (.4,-.4);
\draw[-,line width=.75mm](.4,-.4) to (.4,-.5);
\draw[-,line width=.15mm]  (0,1)to (1.4,-.4);
\draw[-,line width=.75mm](1.4,-.4) to (1.4,-.5);
\draw[-,line width=.15mm] (1,.8) to (-1,-.39);
\draw[-,line width=.15mm] (1,.8) to (.4,-.39);
\draw[-,line width=.15mm] (1,.8) to (1.4,-.39);
\node at (-0.3,0.5) {$\scriptstyle 2$}; 
\draw (-.65,0.13)\bdot;
\node at (-.8,0.13) {$\scriptstyle \nu^1$};
\draw (-.55,-.12)\bdot; \node at (-.5,-0.23) {$\scriptstyle \nu^2$};
\draw (.32,-0.1)\bdot;\node at (.2,-0.13) {$\scriptstyle \nu^3$};
\draw (.55,-0.1)\bdot;\node at (.67,-0.2) {$\scriptstyle  \nu^4$};
\draw (1.1,-0.1)\bdot;\node at (1.04,-0.24) {$\scriptstyle  \nu^5$};
\draw (1.3,-0.1)\bdot;\node at (1.5,-0.13) {$\scriptstyle  \nu^6$};
\node at (1.2,0.5) {$\scriptstyle 2$};
\node at (0.2,0.5) {$\scriptstyle 3$};
\node at (0.1,0.13) {$\scriptstyle 3$};
\node at (0.4,0.7) {$\scriptstyle 4$};
\node at (.9,.4) {$\scriptstyle 2$};
\draw[-,line width=1pt, color=\cred](.6,.8) to (0.6,.4);
\draw[-,line width=1pt, color=\cred](.6,.4) to[out=down,in=135] (1,-.4);
\node at (0.6, .9){$\red{\scriptstyle u_2}$};
\node at (1,-.5){$\red{\scriptstyle u_2}$};
%\draw[-,line width=1pt, color=\cred](0,.8) to (0,-.4);
\draw[-,line width=1pt, color=\cred](-.2,-.4) to[out=135,in=200] (-.4,.6) to[out=45,in=down] (0.3,1.1);
%\draw[-,line width=1.5pt] (0.2,.9) to (0, 1.1);
\node at (0.3, 1.2){$\red{\scriptstyle u_1}$};
\node at (-.2,-.5){$\red{\scriptstyle u_1}$};
\end{tikzpicture}.
\]

However, the resulting red strand dotted web diagram may not be elementary since there may be two intersections between the red strand (corresponding to that in the traverse-down in $f$) and some toes of the  $r$-fold merge. Then we use \eqref{adaptorR} to replace the diagram with the ornamentation of $B$ with additional dots in the toes with not more than one intersection between the red strand and toes. Now we can use \eqref{adaptermovemerge}, \eqref{dotmoveadaptor} and arguments in the proof of  \cite[Proposition~3.6]{SWweb}
%\cite[Proposition \ref{prospaningset}]{SWweb}
to slide the dots down to the bottom, leading to some elementary diagrams. This completes the proof that $\PMat_{\lambda,\mu}$ is a spanning set for $\Hom_{\ASch}(\mu,\lambda)$.

Next, we show that $\PMat_{\lambda,\mu}$ is linearly independent.   

For any $\overline B=( A, (\nu_{ij}))\in \Mat_{\bar \lambda, \bar \mu}$, we denote by $B$ an ornamentation of $\overline B$.  
Suppose $l(\mu^{(i)})=h_i$, $l(\lambda^{(i)})=t_i$. Denote $h:=l(\bar \mu)= \sum_{1\le k\le \ell}h_k$ and $t:=l(\bar \lambda)=\sum_{1\le k\le \ell}t_k$. Write $ p_j=\sum_{1\le i\le j}\bar \mu_i$, for $1\le j\le \sum_{1\le k\le \ell}h_k$. We compute the leading term of $B$ on 
\[
m= \prod_{1\le k\le h} x_{p_{k-1}+1}^{kN} x_{p_{k-1}+2}^{kN}\cdots x_{p_k}^{kN}\in \text{Sym}_{\bar \mu},
\]
for  $N\gg 0$. Here the leading term means the top degree term of the polynomial.    
Note that by definition  the leading term of  $\omega_{a,\nu}$
on $x_b^lx_{b+1}^l\cdots x_{b+a}^{ l}$ is 
$x_b^{l+\nu'_1}x_{b+1}^{l+\nu'_2}\ldots x_{b+a}^{l+\nu'_a},$
where $\nu'$ is the dual partition of $\nu$.
Using this together with the actions of splits, merges, crossings (see \eqref{actionofcrossing}), traverse-ups and traverse-downs we conclude  that  the leading term of $B$ on $m$
is 
\[
x^{((hN)^{a_{1h}})+\nu'_{1h}+((r_{1h})^{a_{1h}})}_{1,\ldots, a_{1h}}
x^{((h-1)N)^{a_{1h-1}}+ \nu'_{1h-1}+(r_{1h-1})^{a_{1h-1}})}_{a_{1h}+1, \ldots, a_{1h}+a_{1h-1}}
\ldots x_{p_k-a_{t1}+1,\ldots, p_k}^{(N)^{a_{t1}} + \nu'_{t1} +(r_{t1}^{a_{t1}}) },\]
where 
\begin{itemize}
    \item $x_{a+1,\ldots, a+b}^{(d_1,\ldots, d_b)}=x_{a+1}^{d_1}\ldots x_{a+b}^{d_b}$,
    \item $r_{ij}$ is the number of traverse-downs which are produced by the strand $a_{ij}$.
\end{itemize}
This shows that for different $\overline B$'s, the ornamentation $B$ will produce different leading terms which are clearly linearly independent. This implies the linear independence of $\PMat_{\lambda,\mu}$. 

When combining with the spanning result for  $\PMat_{\lambda,\mu}$ earlier, we have proved that $\PMat_{\lambda,\mu}$ forms a basis for $\Hom_{\ASch}(\mu,\lambda)$.
\end{proof}

\begin{rem}
\label{rem:faithfulofF}
The argument in the proof  of the linear independence in   Theorem \ref{thm:basisASchur}
 also implies that  the functor $\mathcal F$ in Theorem \ref{thm:polyRep} is faithful.
\end{rem}

\begin{example}
Suppose $B$ is the first diagram in \eqref{ex-of-dotchickenfootd-red} with
\[\nu'_1=(2,1), \nu'_2=(3,2,1), \nu'_3=(4,3,2), \nu'_4=(5,4), \nu'_5=(6,5,4,3), \nu'_{6}=(7,6).\]
Then 
$m=x_1^Nx_2^N\cdots x_5^N x_6^{2N}\cdots x_{10}^{2N}x_{11}^{3N} \ldots x_{16}^{3N}$ and 
\[
r_{13}=1, \quad r_{ij}=0, \quad \text{for } (i,j)\neq (1,3). 
\]

Thus, the leading term of $B m$ is 
\begin{align*}
x_1^{3N+7}x_2^{3N+6}x_3^{3N+5} x_4^{3N+4} x_5^{2N+4}x_6^{2N+3} x_7^{2N+2}
x_8^{N+2}x_9^{N+1} \times
\\
\qquad x_{10}^{3N+7}x_{11}^{3N+6}
x_{12}^{2N+5}x_{13}^{2N+4}
x_{14}^{N+3} x_{15}^{N+2} x_{16}^{N+1}.
\end{align*}
\end{example}

\begin{corollary}
    The affine web category $\AW$ is a full subcategory of the affine Schur category $\ASch$.
\end{corollary}

\begin{proof}
    We have a natural functor $\AW \rightarrow \ASch$, which matches objects and morphisms in the same names. The statement follows by noting that the elementary diagram basis for $\AW$ in Theorem~ \ref{basisAW} and the one for $\ASch$ in Theorem~ \ref{thm:basisASchur} clearly matches.
\end{proof}

%% file: Section4_cycSchurCat.tex
\section{Cyclotomic Schur categories}
 \label{sec:cycSchur}

In this section, we first review (a degenerate version of) some basic results about the cyclotomic Schur algebras $\Sc_{\bfu}$ following \cite{DJM98}. 
Then we introduce the cyclotomic Schur category $\Sch_{\bfu}$ associated to a fixed vector $\bfu\in \kk^{\ell}$. We construct a functor $\mathcal G$ from $\Sch_{\bfu}$ to $\ScatDJM_{\bfu}$ and formulate some first properties of $\mathcal G$ which connect to the cellular basis for the cyclotomic Schur algebras $\Sc_{\bfu}$.

\subsection{Cyclotomic Hecke algebras}\label{subsec-cychecke}

 Fix 
\begin{equation}
\label{cyclotomic-polynomial}
    \bfu=(u_1,\ldots, u_{\ell})\in \kk^{\ell},
    \qquad \text{ for } \ell \ge 1.
\end{equation} 
The degenerate cyclotomic Hecke algebra $\mathcal{H}_{m,\bfu}$ with parameter $\bfu$ is the quotient of $\widehat {\mathcal{H}}_m$ by the two-sided ideal generated by $\prod_{1\le i\le \ell}(x_1-u_i)$. We are going to recall the cellular basis and permutation modules of $\mathcal{H}_{m,\bfu}$ in \cite[Theorem 6.3]{AMR06} (see also  \cite{DJM98} for the quantum case).

Recall the set $\Lambda_{\text{st}}^{1+\ell}(m)$ of strict  $(1+\ell)$-multicompositions of $m$ used in Section~\ref{sec:ASchur}. We will identify the set $\Lambda_{\text{st}}^{\ell}(m)$ of 
strict $\ell$-multicompositions of $m$ with the subset 
$\Lambda_{\text{st}}^{\emptyset,\ell}(m)$ of $\Lambda_{\text{st}}^{1+\ell}(m)$, where 
\begin{align}
\label{def:barlambdaum}
\Lambda_{\text{st}}^{\emptyset,\ell}(m) &:=\{\lambda\in \Lambda_{\text{st}}^{1+\ell}(m)
\mid \lambda^{(0)}=\emptyset\},
\\
\label{def:barlambdau}
\Lambda_{\text{st}}^{\emptyset,\ell} &:=\bigcup_{m\in \N} \Lambda_{\text{st}}^{\emptyset,\ell}(m).    
\end{align}

For any multicomposition $\lambda=(\lambda^{(1)}, \ldots,\lambda^{(\ell)})\in \Lambda_{\text{st}}^\ell(m)$, set 
$a_i=\sum_{1\le j\le i}|\lambda^{(j)}|$, for $1\le i\le \ell-1$, and define 
\begin{equation}\label{def-pilambda}
 \pi_\lambda= \pi_{a_1,1} \pi_{a_2,2} \cdots \pi_{a_{\ell-1},\ell-1},   
\end{equation}
where $\pi_{a,i}=\prod_{1\le j\le a}(x_j-u_{i+1})$.
For example, for $\lambda=(\lambda^{(1)},\lambda^{(2)},\lambda^{(3)})$ with $\lambda^{(1)}=(2),\lambda^{(2)}=(3),\lambda^{(3)}=(1) $, $\pi_\lambda=\prod_{1\le j\le 2}(x_j-u_2) \prod_{1\le j\le 5}(x_j-u_3)$.

Recall the Young subgroup $\mathfrak S_\lambda$ of $\mathfrak S_m$ and the element $\x_\lambda=\sum_{w\in  \mathfrak S_\lambda}w$. Set 
\begin{equation}  \label{def-mlambda}
    m_\lambda= \pi_{\lambda} \x_{\lambda}
    (=\x_\lambda\pi_\lambda).
\end{equation}
The second equality in \eqref{def-mlambda} follows since each element in  $ \mathfrak S_{|\lambda^{(1)}|}\times \ldots \times \mathfrak S_{|\lambda^{(\ell)}|}$ commutes with $\pi_\lambda$. 

There is an anti-involution $*$ of $\mathcal{H}_{m,\bfu}$ which fixes all generators $x_i$'s and $s_j$'s. Note that $m_\lambda^*=m_\lambda$.
For any $\lambda\in \Lambda_{\text{st}}^\ell(m)$, we refer to the right $\mathcal{H}_{m,\bfu}$-module 
\[
M^\lambda=m_\lambda \mathcal{H}_{m,\bfu}
\]
as a permutation module.

An $\ell$-multicomposition $\la$ is an $\ell$-multipartition if each of its components $\la^{(i)}$'s is a partition. Let $\Par^\ell(m)$  be the set of all $\ell$-multipartitions  of $m$. We have $\Par^\ell(m)\subset \Lambda_{\text{st}}^\ell(m)$.

We define a dominance order $\unlhd$ on $\Lambda_{\text{st}}^\ell(m)$ as follows: $\lambda\unlhd \mu$ if 
$$ 
\sum_{i=1}^{j-1}|\lambda^{(i)}|+\sum_{h=1}^l\lambda^{(j)}_h  \le \sum_{i=1}^{j-1}|\mu^{(i)}|+\sum_{h=1}^l\mu^{(j)}_h,
\qquad
\text{for all } 1\le j\le \ell \text{ and all } l.
$$
 We denote $\la \lhd \mu$ or $\mu \rhd \la$ if $\la \unlhd \mu$ and $\la \neq \mu$. Recall the Young diagram  $Y(\mu)$ for any composition $\mu$.  For any $\ell$-multicomposition  $\lambda$, the Young diagram is $Y(\lambda)=(Y(\lambda^{(1)}), \ldots, Y(\lambda^{(\ell)}))$. 
A $\lambda$-tableau $\t=(\t_1,\ldots, \t_{\ell})$ is obtained from $Y(\lambda)$ by inserting the numbers $1,2,\ldots,m$ into each box without repetition. A $\lambda$-tableau $\t$ is standard if 
$\lambda$ is a multipartition and the entries in each component $\t_i$ are row and column increasing.
Let $\std(\lambda)$ be the set of all standard $\lambda$-tableaux.

For any $\ell$-multicomposition $\lambda$ of $m$, the symmetric group $\mathfrak S_m$ acts from the right on the set of $\lambda$-tableaux by permuting the entries in each tableau. 
Let $\t^\lambda$ be the $\lambda$-tableau obtained by 
inserting the numbers $1,2,\ldots,m$  from  left to right along the rows of $Y(\lambda^{(1)})$ and then $Y(\lambda^{(2)})$ and so on. 
 For any $\t\in \std(\lambda)$, denote $\mathfrak S_\lambda=\mathfrak S_{\lambda^{(1)}}\times \ldots\times \mathfrak S_{\lambda^{(\ell)}}$ and 
\begin{align}
   \label{eq:dt}
    d(\t) &= \text{the minimal length coset representative in }
    \mathfrak S_\lambda \backslash \mathfrak S_m \text{ such that } \t=\t^\lambda d(\t).
\end{align}
 
\begin{example}
\label{example:tlambda}
If $\lambda=((3,2), (2),(2,1))$ then 
   \[\t^\lambda=\left(\ytableaushort{{1}{2}{3},{4}{5}},\quad  \ytableaushort{{6}{7}}, \quad \ytableaushort{{8}{9},{10}}  \right). \] 
\end{example}

For $\lambda\in \Par^\ell(m)$ and $\s,\t \in\std(\lambda)$, we define 
 \begin{align}  \label{def-mst}
 m_{\s\t}:= d(\s)^* m_\lambda d(\t).  
 \end{align}
 In particular, we have 
 \begin{align*}
     m_\lambda=m_{\t^\lambda\t^\lambda}, \qquad \text{ and }\quad  m_{\s\t}^*=m_{\t\s}\in \mathcal{H}_{m,\bfu}.
 \end{align*}
 
% \begin{proposition}\cite[Theorem 6.3]{AMR06}
%   The algebra  $\mathcal{H}_{m,f}$ has a cellular basis $\{ m_{\s\t}\mid \s,\t\in \std(\lambda), \lambda\in \Lambda^+(m,r)\}$.
% \end{proposition}
 
To give the basis for a permutation module, we need to recall the notion of $\lambda$-tableaux $\mathbf S$ of type $\mu$, for $\lambda\in \Par^\ell(m)$ and $\mu\in \Lambda_{\text{st}}^\ell(m)$. A $\la$-tableau $\mathbf S$ is obtained from $Y(\lambda)$ by inserting ordered pairs $(i,p)$, also denoted by $i_p$, into $Y(\lambda)$, where $1\le p\le \ell$ and $i$ is a positive integer.  Moreover, $\mathbf S$ has type $\mu$ if the number of times of $(i,p)$ as an entry in $\mathbf S$ is $\mu^{(p)}_i$, for all $i,p$.  

\begin{example}
\label{exam:taleumu}
Suppose $\lambda=((3,2), (2), (2,1))$,  $\mu=((2,3), (1),(2,1^2))$ and 
\[
\mathbf S=\left(\ytableaushort{{1_1}{1_1}{2_1},{2_1}{2_1}},\quad  \ytableaushort{{1_2}{2_3}}, \quad \ytableaushort{{1_3}{1_3},{3_3}}  \right). \]
Then the $\lambda$-tableau  $\mathbf S$ is of type $\mu$.
\end{example}
 Given two ordered pairs $(i_1,p_1 )$ and $(i_2,p_2)$, we say $(i_1,p_1)<(i_2,p_2)$ if $p_1<p_2$, or $p_1=p_2$ and $i_1<i_2$.

Suppose $\lambda\in \Par^\ell(m)$ and $\mu\in \Lambda_{\text{st}}^\ell(m)$. A  $\lambda$-tableau $\mathbf S=(\mathbf S^{(1)}, \ldots, \mathbf S^{(\ell)})$ of type $\mu$ is called semistandard (cf. \cite[Definition (4.4)]{DJM98}) if 
\begin{enumerate}
        \item 
        the entries of each row of $\mathbf S^{(p)}$
         are weakly increasing, for $1\le p\le \ell$; and 
        \item 
        the entries of each column of $\mathbf S^{(p)}$ are strictly increasing, for $1\le p\le \ell$; and 
        \item 
        if $(i,p)$ appears in $\mathbf S^{(j)}$, for $1\le j\le \ell$,  then we must have $p\ge j$.
\end{enumerate}
For example, the tableau $\mathbf S$ in Example \ref{exam:taleumu} is semistandard. Denote by $\SST(\lambda,\mu)$ the set of all semistandard $\lambda$-tableaux of type $\mu$. It is known that $\SST(\lambda,\mu) = \emptyset$ unless $\lambda\unrhd \mu$ (e.g., \cite[\S 2]{Math03}).

Suppose that $\lambda\in \Par^\ell(m)$, $\mu \in \Lambda_{\text{st}}^\ell(m)$ and  $\s\in \std(\lambda)$. Let
  $\mu(\s)$ be the $\lambda$-tableau of type $\mu$
  which is obtained from $\s$ by replacing each entry $n$ in $\s$ by $(i,j)$ (or just $i_j$) if $n$ is in row $i$ of the $j$th component of  $\t^\mu$. For example, for $\s=\t^\lambda$  in Example \ref{example:tlambda} and  $\mu=((2,3), (1),(2,1^2))$,  we have 
  \[\mu(\t^\lambda)=\left(\ytableaushort{{1_1}{1_1}{2_1},{2_1}{2_1}},\quad  \ytableaushort{{1_2}{1_3}}, \quad \ytableaushort{{1_3}{2_3},{3_3}}  \right). \]  
  For any $\mathbf S\in \SST(\lambda,\mu)$ and $\t\in \std(\lambda)$, let 
 \begin{equation}
     m_{\mathbf S\t}:= \sum_{\s\in \mu^{-1}(\mathbf S)}m_{\s\t}
 \end{equation}
 where  
\begin{align}  \label{def-mu-1}
\mu^{-1}(\mathbf S):=\{\s\in \std(\lambda)\mid \mu(\s)=\mathbf S\}. 
\end{align}
  Then the permutation module $M^\mu$  has a $\kk$-basis (cf. \cite[Theorem~ 4.14]{DJM98}) 
  \begin{equation}
      \label{basis-permutation}
  \{m_{\mathbf S\t}\mid \mathbf S\in \SST(\lambda,\mu), \t\in \std(\lambda), \lambda\in \Par^\ell(m)\}.
  \end{equation}

 \subsection{Cyclotomic Schur algebras}
 
 We can define the cyclotomic Schur algebra as in \cite{DJM98} thanks to $\Par^\ell(m)\subset \Lambda_{\text{st}}^\ell(m)$. Suppose that $\bfu =(u_1,\ldots, u_{\ell})\in \kk^{\ell}$ is fixed as for the cyclotomic Hecke algebra $\mathcal{H}_{m,\bfu}$. The cyclotomic Schur algebra is the endomorphism algebra 
     \begin{equation}
        \Sc_{m,\bfu}:= \End_{\mathcal{H}_{m,\bfu}}\Big(\bigoplus _{\mu\in \Lambda_{\text{st}}^\ell(m)}M^\mu\Big).
     \end{equation}
We further denote 
   \begin{align}
   \Sc_\bfu=\bigoplus_{m\in \N}\Sc_{m,\bfu}.  
   \end{align}

 Then $\Sc_{m,\bfu}$ is a degenerate version of $\Sc(\Lambda)$ in \cite[Definition (6.1)]{DJM98} with $\Lambda=\Lambda_{\text{st}}^\ell(m)$.
 Moreover, 
 $\Sc_{m,\bfu}\cong \oplus_{\mu,\nu\in\Lambda_{\text{st}}^\ell(m) }\Hom_{\mathcal{H}_{m,\bfu}}(M^\nu,M^\mu)$.
 We may regard  any $\phi\in \Hom_{\mathcal{H}_{m,\bfu}}(M^\nu,M^\mu)$ as an element of $\Sc_{m,\bfu}$ by extension of zero.

For any $\mu\in \Lambda_{\text{st}}^\ell(m)$, let $1_\mu\in \Hom_{\mathcal{H}_{m,\bfu}}(M^\mu,M^\mu)$ be the 
identity morphism of $M^\mu$.
Then the algebra $\Sc_{\bfu}$ is a locally unital algebra with 
$\{1_\mu\mid \mu\in \Lambda_{\text{st}}^{\ell}\}$ as a family of mutually orthogonal idempotents such that 
\[
\Sc_{\bfu}=\bigoplus _{\mu,\nu\in\Lambda_{\text{st}}^{\ell}} 1_\nu \Sc_{\bfu}1_\mu.
\]
The algebra $\Sc_{\bfu}$ can be identified with a small $\kk$-linear category $ \ScatDJM_\bfu$, with the object set $\Lambda_{\text{st}}^{\ell}$ and morphisms $\Hom_{\ScatDJM_\bfu}(\mu,\nu):=1_\nu \Sc_{\bfu}1_\mu$, which is  $ \Hom_{\mathcal{H}_{m,\bfu}}(M^\mu,M^\nu) $ (respectively, $0$ ) if $\mu\in \Lambda_{\text{st}}^\ell(m), \nu \in\Lambda_{\text{st}}(m',\ell)$ with $m=m'$ (respectively, $m\neq m'$).

We are going to recall the cellular basis of $\Sc_{m,\bfu}$ given in \cite{DJM98}.
    Suppose $\lambda\in \Par^\ell(m)$ and $ \mu,\nu\in \Lambda_{\text{st}}^\ell(m)$. For any $\mathbf S\in \SST(\lambda,\mu)$ and $\mathbf T\in \SST(\lambda,\nu)$, let 
    \begin{equation}
        m_{\mathbf S\mathbf T}=\sum_{\s\in \mu^{-1}(\mathbf S),\t\in\nu^{-1}(\mathbf T) }m_{\s\t}.
    \end{equation}
   Set $\phi_{\mathbf S\mathbf T}\in \Hom_{\mathcal{H}_{m,\bfu}}(M^\nu,M^\mu)$ by 
 \begin{equation} \label{phi}
     \phi_{\mathbf S\mathbf T}(m_\nu h)=m_{\mathbf S\mathbf T}h
 \end{equation}
 for $h\in \mathcal{H}_{m,\bfu}$.   
The following is a degenerate analog of {\cite[Theorems~ (6.6),(6.12)]{DJM98}}, where the proof is essentially the same as the quantum analog.
 \begin{proposition} 
     \label{prop:DJM}
 The algebra $\Sc_{m,\bfu}$ is a free $\kk$-module with basis 
    \begin{equation}
    \label{basisofdjm}
    \{\phi_{\mathbf S\mathbf T}\mid \mathbf S\in \SST(\lambda,\mu), \mathbf T\in \SST(\lambda,\nu), \mu,\nu\in \Lambda_{\text{st}}^\ell(m), \lambda\in \Par^\ell(m)\}. 
    \end{equation}
    Moreover, $\Sc_{m,\bfu}$ is a cellular algebra with \eqref{basisofdjm} as its cellular basis, where the required anti-involution $*$ is given by 
    $\phi^*_{\mathbf S\mathbf T}=\phi_{\mathbf T\mathbf S}$.
\end{proposition}
We remark that the above cellular basis of $\Sc_{m,\bfu}$ is not an analogue of a double coset basis, but is the analogue of Green's codeterminant basis for the Schur algebra.
\subsection{Definition of cyclotomic Schur categories}

 Fix $\bfu \in \kk^\ell$, for $\ell \ge 1$, as in \eqref{cyclotomic-polynomial}. This allows us to identify the objects in \eqref{equ:object-of-affineschur} with elements in  $\Lambda_{\text{st}}^{1+\ell}$. 
The cyclotomic Schur category we are going to define is a small category with the set of objects given by $\Lambda_{\text{st}}^{1+\ell}$. Recall  the identity morphism $1_\mu$ of $\mu$ is drawn as 
\begin{equation}
\label{coloredidentity}
\begin{tikzpicture}[baseline = 10pt, scale=0.5, color=\clr]
\draw[-,line width=1.2pt] (-5,0.5)to[out=up,in=down](-5,2.2);
\draw (-4,1) node {$\ldots$};
\draw (-5,0) node{$ \scriptstyle{\mu^{(0)}_1}$}; 
\draw[-,line width=1.2pt] (-3,0.5)to[out=up,in=down](-3,2.2);
\draw (-3,0) node{$\scriptstyle {\mu^{(0)}_{h_0}}$};
\draw[-,line width=1pt, color=\cred](-2,0.2) to (-2,2.3);
\draw (-2,-.2) node{$\scriptstyle \red{u_1}$};
\draw[-,line width=1.2pt] (-1,0.5)to[out=up,in=down](-1,2.2);
\draw (0,1) node {$\ldots$};
\draw (-1,0) node{$ \scriptstyle{\mu^{(1)}_1}$};   \draw[-,line width=1.2pt] (1,0.5)to[out=up,in=down](1,2.2);
\draw (1,0) node{$\scriptstyle {\mu^{(1)}_{h_1}}$}; 
\draw[-,line width=1pt, color=\cred](2,0.2) to  (2,2.3);
\draw (2,-.2) node{$\scriptstyle \red{u_2}$};
\draw (3,.5) node {$\ldots$};
\draw[-,line width=1pt, color=\cred](4,0.2) to (4,2.3);
\draw (4,-.2) node{$\scriptstyle \red{u_i}$};
\draw[-,line width=1.2pt] (5,0.5)to[out=up,in=down](5,2.2);
\draw (6,1) node {$\ldots$};
\draw (5,0) node{$ \scriptstyle{\mu^{(i)}_1}$};   \draw[-,line width=1.2pt] (7,0.5)to[out=up,in=down](7,2.2);
\draw (7,0) node{$\scriptstyle {\mu^{(i)}_{h_i}}$}; 
\draw[-,line width=1pt, color=\cred](8,0.2) to (8,2.3);
\draw (10,-.2) node{$\scriptstyle \red{u_\ell}$};
\draw (8.3,-.2) node{$\scriptstyle \red{u_{i+1}}$};
\draw (9,1) node {$\ldots$};
\draw[-,line width=1pt, color=\cred] (10,0.2) to (10,2.3);
\draw[-,line width=1.2pt] (11,0.5)to[out=up,in=down](11,2.2);
\draw (12,1) node {$\ldots$};
\draw (11,0) node{$ \scriptstyle{\mu^{(\ell)}_1}$};   \draw[-,line width=1.2pt] (13,0.5)to[out=up,in=down](13,2.2);
\draw (13,0) node{$\scriptstyle {\mu^{(\ell)}_{h_\ell}}$};
\end{tikzpicture},                   
\end{equation}
where $h_j=l(\mu^{(j)})$, $0\le j\le \ell$.
Here the subscript $u_i$ of the red line indicates it is the $i$th red line from the left to right and hence it is between $\mu^{(i-1)}$ and $\mu^{(i)}$. 
We may omit the subscript $u_i$ if it is clear from the context. 
 
Any diagram we will use is obtained from the identity morphism in \eqref{coloredidentity} by replacing some parts with some other diagrams.  
To simplify the notation, we will only draw the different parts to denote the whole diagram if it is clear for us to draw the whole diagram from this.
 For example, a diagram of the form 
 \[
 \begin{tikzpicture}[baseline = 10pt, scale=0.5, color=\clr]
\draw[-,line width=1.2pt] (-8,0.5)to[out=up,in=down](-8,2.2);
\draw (-7,1) node {$\ldots$};
\draw (-8,0) node{$ \scriptstyle{\mu^{(0)}_1}$};   \draw[-,line width=1.2pt] (-6,0.5)to[out=up,in=down](-6,2.2);
\draw (-6,0) node{$\scriptstyle {\mu^{(0)}_{h_0}}$};
\draw[-,line width=1pt, color=\cred](-5,0.2) to (-5,2.3);
\draw (-5,-.2) node{$\scriptstyle \red{u_1}$};
\draw[-,line width=1.2pt] (-4,0.5)to[out=up,in=down](-4,2.2);
\draw (-3,1) node {$\ldots$};
\draw (-4,0) node{$ \scriptstyle{\mu^{(1)}_1}$};
\draw[-,line width=1pt] (-2,1)to[out=up,in=down](-2,1.8) to (2,1.8)to (2,1) to (-2,1);
\draw (-2,0) node{$ \scriptstyle{\mu^{(i)}_j}$}; 
\draw (2,0) node{$ \scriptstyle{\mu^{(i)}_l}$};
\draw[-,line width=1.2pt](-1.7,.5) to(-1.7,1);
\draw[-,line width=1.2pt](1.5,.5) to(1.5,1);
\draw[-,line width=1.2pt](-1.7,1.8) to(-1.7,2.2);
\draw[-,line width=1.2pt](1.5,1.8) to(1.5,2.2);
\draw (0,.6) node {$\ldots$};
\draw (0,2) node {$\ldots$};
\draw[-,line width=1pt, color=\cred](4,0.2) to (4,2.3);
\draw (4,-.2) node{$\scriptstyle \red{u_\ell}$};
\draw (3,1) node {$\ldots$};
\draw[-,line width=1.2pt] (5,0.5)to[out=up,in=down](5,2.2);
\draw (6,1) node {$\ldots$};
\draw (5,0) node{$ \scriptstyle{\mu^{(\ell)}_1}$};   \draw[-,line width=1.2pt] (7,0.5)to[out=up,in=down](7,2.2);
\draw (7,0) node{$\scriptstyle {\mu^{(\ell)}_{h_\ell}}$};
\end{tikzpicture},  
\]
can be expressed concisely as 
\[
\begin{tikzpicture}[baseline = 10pt, scale=0.5, color=\clr]
\draw[-,line width=1pt] (-2,1)to[out=up,in=down](-2,1.8) to (2,1.8)to (2,1) to (-2,1);
\draw (-2,0) node{$ \scriptstyle{\mu^{(i)}_j}$}; 
\draw (2,0) node{$ \scriptstyle{\mu^{(i)}_l}$};
\draw[-,line width=1.2pt](-1.7,.5) to(-1.7,1);
\draw[-,line width=1.2pt](1.5,.5) to(1.5,1);
\draw[-,line width=1.2pt](-1.7,1.8) to(-1.7,2.2);
\draw[-,line width=1.2pt](1.5,1.8) to(1.5,2.2);
\draw (0,.6) node {$\ldots$};
\draw (0,2) node {$\ldots$};
\end{tikzpicture}~.
\]

\begin{definition}\label{def-of-cyc-Schur}
The cyclotomic Schur category $\Sch_{\bfu}$ with parameter $\bfu \in \kk^\ell$ is a $\kk$-linear category with objects being strict $(1+\ell)$-multicompositions $\mu =(\mu^{(0)},\mu^{(1)},\ldots,\mu^{(\ell)}) \in \Lambda_{\text{st}}^{1+\ell}$.
The generating morphisms are the splits, merges,  crossings, and dotted strands (for $a,b \ge 1$)
\begin{equation} \label{merge i}
  \begin{tikzpicture}[baseline = 10pt, scale=0.4, color=\clr] 
                \draw[-,line width=1pt] (-2,0.2) to (-1.5,1);
	\draw[-,line width=1pt ] (-1,0.2) to (-1.5,1);
	\draw[-,line width=1.5pt] (-1.5,1) to (-1.5,1.8);
                \draw (-1,0) node{$\scriptstyle {b}$};
                \draw (-2,0) node{$\scriptstyle {a}$};
                \draw (-1.5,2.2) node{$\scriptstyle {a+b}$};
                \end{tikzpicture}, 
\quad \begin{tikzpicture}[baseline = 10pt, scale=0.4, color=\clr] 
                \draw[-,line width=1pt] (-2,0.2) to (-1,1.8);
	\draw[-,line width=1pt ] (-1,0.2) to (-2,1.8); 
                \draw (-1,0) node{$\scriptstyle {b}$};
                \draw (-2,0) node{$\scriptstyle {a}$};
                \draw (-1,2.2) node{$\scriptstyle {a}$};
                 \draw (-2,2.2) node{$\scriptstyle {b}$};
                \end{tikzpicture},   \quad 
  \begin{tikzpicture}[baseline = 10pt, scale=0.4, color=\clr]
                \draw[-,line width=1pt] (-2,1.8) to (-1.5,1);
	\draw[-,line width=1pt ] (-1,1.8) to (-1.5,1);
	\draw[-,line width=1.5pt] (-1.5,1) to (-1.5,0.2);
                \draw (-1,2.2) node{$\scriptstyle {b}$};
                \draw (-2,2.2) node{$\scriptstyle {a}$};
                \draw (-1.5,0) node{$\scriptstyle {a+b}$};
                \end{tikzpicture}, 
                \quad  
 \wkdota,             
 \end{equation}
and traverse-ups and traverse-downs:
\begin{equation}
\label{chameleonsgen}
\begin{tikzpicture}[baseline = 10pt, scale=.8, color=\clr]
 \draw[-,line width=1.2pt] (-0.3,.3) to (.3,1);
\draw[-,line width=1pt,color=\cred] (0.3,.3) to (-.3,1);
\draw(0.3,.15) node {$\scriptstyle \red{u_i}$};
\draw (-.3,0.15) node{$\scriptstyle a$};
\draw (.3,1.15) node{$\scriptstyle {a}$};
\end{tikzpicture}, 
                \qquad
\begin{tikzpicture}[baseline = 10pt, scale=.8, color=\clr]
 \draw[-,line width=1pt,color=\cred] (-0.3,.3) to (.3,1);
 \draw(-.3,.15) node {$\scriptstyle \red{u_i}$};
\draw[-,line width=1.2pt] (0.3,.3) to (-.3,1);
\draw (-.3,1.15) node{$\scriptstyle a$};
\draw (.3,0.15) node{$\scriptstyle {a}$};
\end{tikzpicture}.     
\end{equation}
 The generators \eqref{merge i}--\eqref{chameleonsgen}  are subject to the  local relations in  \eqref{webassoc}--\eqref{intergralballon}, \eqref{adaptorR}--\eqref{adaptrermovecross} and the following additional relations:
\begin{equation}
\label{cyclotomicpoly}
1_{\lambda}=0 \text{ unless } \lambda =(\emptyset,\la^{(1)},\ldots,\la^{(\ell)}). 
\end{equation}

Finally, we impose commuting relations via obvious diagrammatic isotopes that do not change the combinatorial type of the diagram (see also the commuting relations \cite[(2.8)]{BCNR} for any monoidal category). For example,  
\begin{equation}
\label{commurelation}
     ~\begin{tikzpicture}[baseline = 10pt, scale=0.4, color=\clr] 
               \draw[-,line width=1pt] (-2,0.2) to (-1.5,1);
	\draw[-,line width=1pt] (-1,0.2) to (-1.5,1);
\draw[-,line width=1.5pt] (-1.5,1) to (-1.5,1.8);
 \draw (-1,0) node{$\scriptstyle {b}$};
\draw (-2,0) node{$\scriptstyle {a}$};
\draw (-1.5,2.2) node{$\scriptstyle {a+b}$};
               \end{tikzpicture}~
\begin{tikzpicture}[baseline = 10pt, scale=0.4, color=\clr] 
\draw[-,line width=1pt] (-2,0.2) to (-1,1.8);
\draw[-,line width=1pt] (-1,0.2) to (-2,1.8); 
\draw (-1,0) node{$\scriptstyle {c}$};
 \draw (-2,0) node{$\scriptstyle {d}$};
\draw (-1,2.2) node{$\scriptstyle {d}$};
\draw (-2,2.2) node{$\scriptstyle {c}$};
\end{tikzpicture}~ ~
:= ~  
~\begin{tikzpicture}[baseline = 10pt, scale=0.4, color=\clr] 
 \draw[-,line width=1pt] (-2,0.7) to (-1.5,1.5);
\draw[-,line width=1pt] (-2,0.7) to (-2,-.3);  \draw[-,line width=1pt] (-1,0.7) to (-1,-.3);   
\draw[-,line width=1pt] (-1,0.7) to (-1.5,1.5);
\draw[-,line width=1.5pt] (-1.5,1.5) to (-1.5,2.3);
 \end{tikzpicture}
~\begin{tikzpicture}[baseline = 10pt, scale=0.4, color=\clr] 
\draw[-,line width=1pt] (-2,-.3) to (-1,.7);
\draw[-,line width=1pt] (-1,.7) to (-1,2.3);
\draw[-,line width=1pt] (-2,.7) to (-2,2.3);
\draw[-,line width=1pt] (-1,-0.3) to (-2,.7); 
\end{tikzpicture}~
~=~ 
~\begin{tikzpicture}[baseline = 10pt, scale=0.4, color=\clr] 
 \draw[-,line width=1pt] (-2,-0.3) to (-1.5,.6);  
\draw[-,line width=1pt] (-1,-0.3) to (-1.5,.6);
\draw[-,line width=1.5pt] (-1.5,.6) to (-1.5,2.3);
 \end{tikzpicture}~
\begin{tikzpicture}[baseline = 10pt, scale=0.4, color=\clr] 
 \draw[-,line width=1pt] (-2,.8) to (-1,2.3);
 \draw[-,line width=1pt] (-1,.8) to (-1,-.3);
\draw[-,line width=1pt] (-2,.8) to (-2,-.3);
\draw[-,line width=1pt] (-1,.8) to (-2,2.3); 
\end{tikzpicture}~.
\end{equation} 
\end{definition}

\begin{rem} We do not use the last commuting relations elsewhere since we define affine Schur and web categories as monoidal categories, where it is implicit. We add them in the definition because the cyclotomic Schur category is not monoidal. 
On the other hand, the cyclotomic Schur category $\Sch_\bfu$ can be viewed as the quotient of the $\kk$-linear category $\ASch$ by the relations in \eqref{cyclotomicpoly} and additional relations  $1_\lambda=0$
for all $\lambda\notin \Lambda_{\text{st}}^{1+\ell}$ (associated to the given $\mathbf u$). 
%The slight difference here is that this quotient of $\ASch$ has additional objects not in $\Lambda_{\text{st}}^{1+\ell}$ but  with $\Hom(\mathbf a,\mathbf b)=0$ if $\mathbf a\notin \Lambda_{\text{st}}^{1+\ell}$ or $\mathbf b\notin \Lambda_{\text{st}}^{1+\ell} $ (Therefore, the endomorphism algebras for the two kinds of cyclotomic quotients are the same). 
\end{rem}

Thanks to \eqref{cyclotomicpoly}, the morphism space 
$\Hom_{\Sch_\bfu}(\lambda,\mu)=0$ unless 
$\lambda, \mu\in\Lambda_{\text{st}}^{\emptyset,\ell}$; see \eqref{def:barlambdau}. 
Recall the morphisms $g_{r}(u)$ introduced in \eqref{def-gau}, for $r \ge 1$. For $1\le i\le \ell$ and $r\in\Z_{> 0}$, we define 
\begin{align}
  \label{eq:gri}
g_{r,i}=\prod_{1\le j\le i}g_{r}(u_j).
\end{align}
Then we may draw $g_{r,i}$ as a diagram
\begin{tikzpicture}[baseline = 10pt, scale=0.4, color=\clr]
\draw[-,line width=1pt] (-1,2) to (-1,0.2);
\draw (-1,1) \bdot;
\draw (-0.2,1) node{$\scriptstyle {g_{r,i}}$};
\draw (-1,-.2) node{$\scriptstyle {r}$}; 
\end{tikzpicture} 
which is the vertical concatenation of 
\begin{tikzpicture}[baseline = 10pt, scale=0.4, color=\clr]
\draw[-,line width=1pt] (-1,2) to (-1,0.2);
\draw (-1,1) \bdot;
\draw (0.2,1) node{$\scriptstyle {g_{r}(u_j)}$};
\draw (-1,-.2) node{$\scriptstyle {r}$}; 
\end{tikzpicture}, for $1\le j\le i$.
The following result shows that the black strand with concatenation of $g_{r,i}$ can be slide to the left of all the $i$ red strands.  
\begin{lemma}
\label{lem:cycpolyvanish}
The following relation holds in $\Sch_{\bfu}$:
\begin{equation}
\label{cyclotomicpolyi}
~
\begin{tikzpicture}[baseline = 10pt, scale=0.4, color=\clr]
\draw[-,line width =1pt,color=\cred] (-4.5,.2) to (-4.5,2);
\draw (-4.6,-.2) node{$\scriptstyle \red{u_1}$};    
\draw[-,line width =1pt,color=\cred] (-3.5,.2) to (-3.5,2);
\draw (-3.4,-.2) node{$\scriptstyle \red{u_2}$};
\draw(-2.6,.8) node{$\ldots$};
\draw[-,line width =1pt,color=\cred] (-2,.2) to (-2,2);
\draw (-2,-.2) node{$\scriptstyle \red{u_{i}}$};
\draw[-,line width=1pt] (-1,2) to (-1,0.2);
\draw (-1,1) \bdot;
\draw (-0.2,1) node{$\scriptstyle {g_{r,i}}$};
\draw (-1,-.2) node{$\scriptstyle {r}$}; 
\end{tikzpicture}
~1_*
=0, 
\qquad 1\le i\le \ell,
\end{equation}
where $1_*$ represents any suitable identity morphisms. 
\end{lemma}

\begin{proof}
We prove by induction on $i$. The case for $i=1$ follows from \eqref{adaptorL} and \eqref{cyclotomicpoly}.  Suppose $2\le i\le \ell$. Note by \eqref{eq:gri} that $g_{r,i} =g_{r,i-1} g_r(u_i)$. By the inductive assumption on $i-1$, we have 
\[
\begin{tikzpicture}[baseline = 10pt, scale=0.4, color=\clr]
\draw[-,line width =1pt,color=\cred] (-4,.2) to (-4,2);
\draw (-4.1,-.2) node{$\scriptstyle \red{u_1}$};    
\draw[-,line width =1pt,color=\cred] (-3.5,.2) to (-3.5,2);
\draw (-3.4,-.2) node{$\scriptstyle \red{u_2}$};
\draw(-2.6,.8) node{$\ldots$};
\draw[-,line width =1pt,color=\cred] (-2,.2) to (-2,2);
\draw (-2,-.2) node{$\scriptstyle \red{u_i}$};
\draw[-,line width=1pt] (-1,2) to (-1,0.2);
\draw (-1,1) \bdot;
\draw (-.2,1) node{$\scriptstyle {g_{r,i}}$};
\draw (-1,-.2) node{$\scriptstyle {r}$}; 
\end{tikzpicture}
~\overset{\eqref{adaptorL}}=~
\begin{tikzpicture}[baseline = 10pt, scale=0.4, color=\clr]
\draw[-,line width =1pt,color=\cred] (-4,.2) to (-4,2);
\draw (-4.1,-.2) node{$\scriptstyle \red{u_1}$};  
\draw[-,line width=1,color=\cred](-2,2) to [out=down,in=up] (-1,1.5) to[out=down,in=up] (-2.3,0);
\draw[-,line width =1pt,color=\cred] (-3.5,.2) to (-3.5,2);
\draw (-3.3,-.2) node{$\scriptstyle \red{u_2}$};
\draw(-2.6,.8) node{$\ldots$};
%\draw[-,line width =1pt,color=\cred] (-2,.2) to (-2,2);
\draw (-2.5,-.2) node{$\scriptstyle \red{u_i}$};
\draw[-,line width=1pt] (-1.5,2) to (-1.5,0);
\draw (-1.5,.5) \bdot;
\draw (-.3,.5) node{$\scriptstyle {g_{r,i-1}}$};
\draw (-1.5,-.2) node{$\scriptstyle {r}$}; 
\end{tikzpicture}
~\overset{\eqref{dotmoveadaptor}}=~
\begin{tikzpicture}[baseline = 10pt, scale=0.4, color=\clr]
\draw[-,line width =1pt,color=\cred] (-4,.2) to (-4,2);
\draw (-4,-.2) node{$\scriptstyle \red{u_1}$};  
\draw[-,line width=1,color=\cred](-2,2) to [out=down,in=up] (-1,1.5) to[out=down,in=up] (-2.3,0);
\draw[-,line width =1pt,color=\cred] (-3.5,.2) to (-3.5,2);
\draw (-3.3,-.2) node{$\scriptstyle \red{u_2}$};
\draw(-2.6,.8) node{$\ldots$};
%\draw[-,line width =1pt,color=\cred] (-2,.2) to (-2,2);
\draw (-2.5,-.2) node{$\scriptstyle \red{u_i}$};
\draw[-,line width=1pt] (-1.3,2) to (-1.3,0);
\draw (-1.3,1.4) \bdot;
\draw (-2.45,1.4) node{$\scriptstyle {g_{r,i-1}}$};
\draw (-1.3,-.2) node{$\scriptstyle {r}$}; 
\end{tikzpicture}
=0,
\]
where the identity morphism $1_*$ on the right of each diagram is omitted for simplicity.
The lemma is proved. 
\end{proof}

\subsection{A functor $\mathcal G$ from $\Sch_{\bfu}$ to $\ScatDJM_{\bfu}$}
\label{subsec:functorG}

Recall the set $\Lambda_{\text{st}}^{\emptyset,\ell} $ in \eqref{def:barlambdau} which can be identified with $\Lambda_{\text{st}}^{\ell}$ via $\lambda =(\emptyset, \lambda^{(1)}, \ldots, \lambda^{(\ell)}) \in \Lambda_{\text{st}}^{\emptyset,\ell} \leftrightarrow(\lambda^{(1)}, \ldots, \lambda^{(\ell)})\in \Lambda_{\text{st}}^{\ell}$.
Moreover, the algebra $\ScatDJM_{\bfu}$ is identified with its category version which is also denoted by $\ScatDJM_{\bfu}$.
We may extend the object set of  $\ScatDJM_{\bfu}$ to $\Lambda_{\text{st}}^{1+\ell}$ such that $ \Hom_{\Sc_\bfu}(\lambda,\mu)=0$ if $\lambda\notin \Lambda_{\text{st}}^{\emptyset,\ell} $ or $\mu\notin \Lambda_{\text{st}}^{\emptyset,\ell} $.

 Associated with each generator of $\Sch_{\bfu}$ in \eqref{merge i}--\eqref{chameleonsgen}, we introduce certain elements in $\ScatDJM_{\bfu}$ denoted by the same symbols in (1)--(4) below. 
 We suppose the any generator in \eqref{merge i}--\eqref{chameleonsgen} is in $\Hom_{\Sch_\bfu}(\mu,\tilde \mu)$, for $ \mu ,\tilde \mu \in  \Lambda_{\text{st}}^{\emptyset,\ell} $. Moreover, in the following we will omit the $1_*$ on the left and right of the generators. For example, we just use $ \merge$ to denote $1_* \merge1_*$ if it is clear in the context. 
\begin{enumerate}
\item 
Consider
$
\merge 
\in \Hom_{\Sch_{\bfu}}(\mu,\tilde\mu)
$, where $\tilde \mu$ is obtained from $\mu$  by contracting $a=\mu^{(i)}_j$ and $b=\mu^{(i)}_{j+1}$ as $\mu^{(i)}_j+\mu^{(i)}_{j+1}$ for some $i$th component and some $j$. 
Then
$\pi_{\mu}=\pi_{\tilde \mu}$
and 
$
\x_{\tilde\mu}=\x_{\mu}\sigma_{a,b}=\sigma_{a,b}^*\x_\mu
$, where 
$\sigma_{a,b}$
is given as in \eqref{eq:sigmaab}. Moreover, we have 
$
\sigma_{a,b}\pi_{\mu}=\pi_{\mu}\sigma_{a,b}
$
and
$
m_{\tilde \mu}=\sigma_{a,b}^*m_{\mu}=m_\mu \sigma_{a,b} 
$.
Then we define 
$
\blue{\rot{Y}}\in \Hom_{\mathcal{H}_{m,\bfu}}(M^\mu,M^{\tilde \mu})
$
such that 
 \[
 \blue{\rot{Y}}: M^\mu \longrightarrow M^{\tilde \mu} , \quad m_{\mu}\mapsto m_{\tilde \mu}. 
 \]
 That is, $\blue{\rot{Y}}$ is given by left multiplication of $\sigma_{a,b}^* $.
 
 For the generator 
 $\splits \in \Hom_{\Sch_{\bfu}}(\tilde\mu,\mu)$, 
we define 
$
\blue{\text{Y}}\in \Hom_{\mathcal{H}_{m,f}}(M^{\tilde \mu},M^{\mu}) 
$ 
such that 
\[
\blue{\text{Y}}: M^{\tilde \mu}\longrightarrow M^\mu, \quad m_{\tilde \mu } \mapsto m_{\tilde \mu}, 
\]
i.e.,  $\blue{\text{Y}}$ is the inclusion from $M^{\tilde \mu}$ to $M^\mu$. 
 
\item 
Consider the traverse-up   
  $
  \upliftuip\in \Hom_{\Sch_{\bfu}}(\mu,\tilde \mu)
  $ 
  with $1\le i \le \ell-1$ and $a=\mu^{(i)}_{h_i}$, where $h_i=l(\mu^{(i)})$. 
  Then $\tilde \mu$ and $\mu$ differ only on the $i$th and $(i+1)$st components, i.e., $\tilde \mu^{(i)}$ is obtained from $\mu^{(i)}$ with its last part $\mu^{(i)}_{h_i}$ removed while $\tilde \mu^{(i+1)}$ is obtained from $\mu^{(i+1)}$ by inserting a new part $\mu^{(i)}_{h_i}$ to the leftmost. 
  We further have
  \[
  \mathfrak S_\mu=\mathfrak S_{\tilde\mu}, 
  \quad
\x_\mu=\x_{\tilde \mu}, 
\quad \text{ and }
\quad
\pi_\mu=  \prod_{1\le j\le \mu^{(i)}_{h_i}}(x_{t+j}-u_{i+1}) \pi_{\tilde \mu},
\]
where 
$t=\sum_{1\le j\le i}|\mu^{(i)}_j|-\mu^{(i)}_{h_i}$.
Then 
\[
m_{\mu}=\prod_{1\le j\le \mu^{(i)}_{h_i}}(x_{t+j}-u_{i+1}) m_{\tilde \mu}=m_{\tilde \mu} \prod_{1\le j\le \mu^{(i)}_{h_i}}(x_{t+j}-u_{i+1}). 
\]
We define the morphism 
$\upliftuip\in \Hom_{\mathcal{H}_{m,\bfu}}(M^\mu,M^{\tilde \mu})$ 
such that 
\[
\upliftuip: M^\mu\longrightarrow M^{\tilde \mu}, \quad m_{\mu}\mapsto m_{\mu},  
\]
i.e., $\upliftuip$ is the inclusion from $M^\mu$ to $M^{\tilde \mu}$.

With the same $\mu, \tilde\mu$ above, for the generator 
$ \downliftuip\in \Hom_{\Sch_{\bfu}}(\tilde \mu,\mu)$, we define the associated morphism
$\downliftuip\in \Hom_{\mathcal{H}_{m,f}}(M^{\tilde \mu}, M^\mu)$ such that    
     \[
     \downliftuip: M^{\tilde\mu} \longrightarrow M^\mu, \quad m_{\tilde \mu}\mapsto m_{\mu}, 
     \]      
i.e., $\downliftuip$  
is given by the left multiplication of
$\prod_{1\le j\le \mu^{(i)}_{h_i}}(x_{t+j}-u_{i+1})  $.              
\item                
 Consider the dot generator   
 $
 \wkdotr
~\in \Hom_{\Sch_{\bfu}}(\mu,\mu)
$,
where $r=\mu^{(i)}_j$ for some $i,j$ with $i\ge1$. Set 
$s=\sum_{1\le h\le i-1}|\mu^{(h)}|+\sum_{1\le l\le j-1}\mu^{(i)}_l$.
Then 
\[\prod_{s+1\le k\le s+r} x_k \x_\mu=\x_\mu \prod_{s+1\le k\le s+r} x_k, 
\text{ and }  
\prod_{s+1\le k\le s+r}x_km_\mu=m_\mu \prod_{s+1\le k\le s+r}x_k.
\]
We define 
$
\wkdotr
\in \Hom_{\mathcal{H}_{m,\bfu}}(M^\mu,M^\mu)
$
such that  
\[
\wkdotr:
 M^\mu\longrightarrow M^\mu, \quad m_\mu\mapsto \prod_{s+1\le k\le s+\mu^{(i)}_j}x_km_\mu, 
 \]
i.e., the left multiplication by $\prod_{s+1\le k\le s+\mu^{(i)}_j}x_k$.           
 In general, the map associated to $\wkdota$ with $r<a$ is determined by \eqref{splitmerge},  i.e.,  the composition of  maps for $
\begin{tikzpicture}[baseline = -.5mm,scale=.8,color=\clr]
	\draw[-,line width=1.5pt] (0.08,-.3) to (0.08,0.04);
	\draw[-,line width=1pt] (0.28,.4) to (0.08,0);
	\draw[-,line width=1pt] (-0.12,.4) to (0.08,0);
        \node at (-0.22,.5) {$\scriptstyle r$};
        \node at (0.36,.55) {$\scriptstyle b$};
        \node at (0.1,-.45){$\scriptstyle a$};
\end{tikzpicture}$,
$
\begin{tikzpicture}[baseline = 3pt, scale=0.4, color=\clr]
\draw[-,line width=1.2pt] (0,0) to[out=up, in=down] (0,1.4);
\draw(0,0.6) \bdot; 
\draw (-.6,0.6) node {$\scriptstyle \omega_r$};
\node at (0,-.3) {$\scriptstyle r$};
\draw[-,line width=1.2pt](0.4,0) to (0.4,1.4);
  \draw(0.4,-.3)node {$\scriptstyle b$};
\end{tikzpicture}
$ 
and 
$\begin{tikzpicture}[baseline = -.5mm,scale=.8,color=\clr]
	\draw[-,line width=1pt] (0.28,-.3) to (0.08,0.04);
	\draw[-,line width=1pt] (-0.12,-.3) to (0.08,0.04);
	\draw[-,line width=1.5pt] (0.08,.4) to (0.08,0);
        \node at (-0.13,-.45) {$\scriptstyle r$};
        \node at (0.35,-.4) {$\scriptstyle b$};\node at (0.08,.55){$\scriptstyle a$};\end{tikzpicture} $ with $b=a-r$.               
                
\item 
Consider    
$
\begin{tikzpicture}[baseline = 10pt, scale=0.4, color=\clr] 
\draw[-,line width=1pt] (-2,0.2) to (-1,1.8);
\draw[-,line width=1pt] (-1,0.2) to (-2,1.8); 
\draw (-1,0) node{$\scriptstyle {b}$};
\draw (-2,0) node{$\scriptstyle {a}$};
\draw (-1,2.2) node{$\scriptstyle {a}$};
\draw (-2,2.2) node{$\scriptstyle {b}$};
\end{tikzpicture}
\in \Hom_{\Sch_{\bfu}}(\mu,\tilde \mu)
$,
where $\tilde \mu $
is obtained from $\mu$ by swapping 
$a=\mu^{(i)}_j$ 
and 
$b=\mu^{(i)}_{j+1}$ 
for some $i,j$ with $i\ge 1$.
Then $\pi_\mu=\pi_{\tilde \mu}$.
We define the morphism 
$
\begin{tikzpicture}[baseline = 10pt, scale=0.4, color=\clr] 
\draw[-,line width=1pt] (-2,0.2) to (-1,1.8);
\draw[-,line width=1pt] (-1,0.2) to (-2,1.8); 
\draw (-1,0) node{$\scriptstyle {b}$};
\draw (-2,0) node{$\scriptstyle {a}$};
\draw (-1,2.2) node{$\scriptstyle {a}$};
\draw (-2,2.2) node{$\scriptstyle {b}$};
\end{tikzpicture}
\in \Hom_{\mathcal{H}_{m,\bfu}}(M^\mu,M^{\tilde \mu}) 
$ 
determined by 
\[
\begin{tikzpicture}[baseline = 10pt, scale=0.4, color=\clr] 
\draw[-,line width=1pt] (-2,0.2) to (-1,1.8);
\draw[-,line width=1pt] (-1,0.2) to (-2,1.8); 
\draw (-1,0) node{$\scriptstyle {b}$};
\draw (-2,0) node{$\scriptstyle {a}$};
\draw (-1,2.2) node{$\scriptstyle {a}$};
\draw (-2,2.2) node{$\scriptstyle {b}$};
\end{tikzpicture}: M^\mu \longrightarrow M^{\tilde \mu}, \quad m_\mu \mapsto \pi_{\mu} \phi\Big(
\begin{tikzpicture}[baseline = 10pt, scale=0.4, color=\clr] 
\draw[-,line width=1pt] (-2,0.2) to (-1,1.8);
\draw[-,line width=1pt] (-1,0.2) to (-2,1.8); 
\draw (-1,0) node{$\scriptstyle {b}$};
\draw (-2,0) node{$\scriptstyle {a}$};
\draw (-1,2.2) node{$\scriptstyle {a}$};
\draw (-2,2.2) node{$\scriptstyle {b}$};
\end{tikzpicture}
\Big)(\x_\mu),    
\]
where $\phi$ is the isomorphism given in 
Proposition ~\ref{cor-isom-schur}.
Note that by \eqref{crossgen} and  Proposition~ \ref{cor-isom-schur}, the morphism  
$
\begin{tikzpicture}[baseline = 10pt, scale=0.4, color=\clr] 
\draw[-,line width=1pt] (-2,0.2) to (-1,1.8);
\draw[-,line width=1pt] (-1,0.2) to (-2,1.8); 
\draw (-1,0) node{$\scriptstyle {b}$};
\draw (-2,0) node{$\scriptstyle {a}$};
\draw (-1,2.2) node{$\scriptstyle {a}$};
\draw (-2,2.2) node{$\scriptstyle {b}$};
\end{tikzpicture}
$ 
is determined by 
$\blue{\text{Y}}$'s 
and
$\blue{\rot{Y}}$'s.            
\end{enumerate}
Finally,  for any generator in \eqref{merge i}--\eqref{chameleonsgen} viewed as in $\Hom_{\Sch_{\bfu}}(\lambda,\mu)$ for some $\lambda\notin \Lambda_{\text{st}}^{\emptyset,\ell}$ or 
 $\mu\notin\Lambda_{\text{st}}^{\emptyset,\ell}$, the associated morphisms in $\Sc_{\bfu}$ are defined to be zero.
 
\begin{theorem}
 \label{thm:G}
There is a functor $\mathcal G: \Sch_{\bfu}\rightarrow \ScatDJM_{\bfu}$, which sends an object $\mu$ to $\mu$, and the generating morphisms in  \eqref{merge i}--\eqref{chameleonsgen} to the corresponding morphisms 
$\blue{\text{Y}}$, 
$\blue{\rot{Y}}$, 
$\upliftui$, 
$\downliftui$,
$\wkdota$, 
$\begin{tikzpicture}[baseline = 10pt, scale=0.4, color=\clr] 
\draw[-,line width=1pt] (-2,0.2) to (-1,1.8);
\draw[-,line width=1pt] (-1,0.2) to (-2,1.8); 
\draw (-1,0) node{$\scriptstyle {b}$};
\draw (-2,0) node{$\scriptstyle {a}$};
\draw (-1,2.2) node{$\scriptstyle {a}$};
\draw (-2,2.2) node{$\scriptstyle {b}$};
\end{tikzpicture}
$, respectively.
\end{theorem} 

\begin{proof}
It suffices to check that all the defining relations, \eqref{webassoc}--\eqref{intergralballon}, \eqref{adaptorR}--\eqref{adaptermovemerge}, \eqref{cyclotomicpoly}, and the commuting relations of $\Sch_{\bfu}$ (see Definition~\ref{def-of-cyc-Schur}) are satisfied under the functor $\mathcal G$. Note first that \eqref{webassoc}--\eqref{splitmerge}  are satisfied by Proposition \ref{cor-isom-schur}.
The relations \eqref{cyclotomicpoly} and  \eqref{intergralballon}  are satisfied by definition. The commuting relations may also be checked easily by definitions. The relation \eqref{adaptermovemerge} can be checked by using the fact that
\[
\sigma_{a,b}\prod_{s+1\le k\le a+b}(x_k-u) = \prod _{s+1\le k\le a+b} (x_k-u) \sigma_{a,b}
\]
for any $u$. 

It remains to check the relations \eqref{dotmovecrossing}--\eqref{dotmovesplitss} and \eqref{adaptorR}--\eqref{adaptrermovecross}. Since the permutation modules and their morphisms are all defined over the polynomial ring $\Z[\tilde u_1,\ldots,\tilde u_\ell]$, by the same type of argument in the proof of Theorem~ \ref{thm:polyRep}, we are reduced to verifying these relations over $\Z[\tilde u_1,\ldots,\tilde u_\ell]$  and then are reduced to checking the relations \eqref{dotmovecrossingC}, \eqref{chamvanish1 C} and \eqref{adaptrermovecross C} over $\C$ by the proof of \cite[Lemma~ 5.3]{SWweb} and Lemmas \ref{adpterLRa=1}--\ref{chraterzero}. 

\underline{Relations \eqref{dotmovecrossingC}}.  Suppose that the two $1$'s in the labels of the first relation in \eqref{dotmovecrossingC} represent $\mu^{(i)}_j$ and $\mu^{(i)}_{j+1}$ for some $j$.  Let $t=\sum_{1\le h\le i-1}|\mu^{(h)}|+\sum_{1\le l\le j}\mu^{(i)}_l$.
Then the left (respectively, right) hand side is given by left multiplication of $s_t x_t $ (respectively, $x_{t+1}s_{t}+1$).
They are equal by \eqref{xs}. The argument for the second relation in \eqref{dotmovecrossingC} is the same.

\underline{Relation \eqref{adaptrermovecross C}}.
Suppose the two $1$'s in the labels of each diagram of \eqref{adaptrermovecross C} represent $\mu^{(i)}_{h_i}$ and $\mu^{(i+1)}_{1}$, respectively. 
Then the left hand side of \eqref{adaptrermovecross C} under $\mathcal G$ is given by left multiplication of $ (x_{t+1}-u_{i+1})s_{t+1}$, while the right hand side is 
$s_{t+1}(x_{t+2}-u_{i+1})+1$, where $t=\sum_{1\le j\le i}|\mu^{(j)}|-\mu^{(i)}_{h_i}$. They are equal by \eqref{xs}.

\underline{Relations \eqref{chamvanish1 C}}. 
For the first relation in  \eqref{chamvanish1 C}, suppose that the label $1$ therein represents $\mu^{(i)}_{h_i}$.
Then this relation follows since morphisms in both sides of the relation in \eqref{chamvanish1 C}  are given by left multiplication of $x_{t+1}-u_{i+1}$ (respectively, 0), if $i\ge 2$ (respectively, $i=1$), where $t=\sum_{1\le j\le i}|\mu^{(j)}|-\mu^{(i)}_{h_i}$.
For the second relation in  \eqref{chamvanish1 C}, suppose  the label $1$ therein represents $\mu^{(i)}_1$ with $i\ge 1$.  Then the right-hand side is  given by 
left multiplication of $x_{t+1}-u_{i+1}$, where $t=\sum_{1\le j\le i-1}|\mu^{(j)}|$, which is equal to the left-hand side for $i\ge2$.
Suppose $i=1$. Then the left-hand side is $0$ by definition. The right-hand side is also $0$ since $(x_{1}-u_{1}) m_\mu=0$ for any $\mu\in
\Lambda_{\text{st}}^{\ell}=\Lambda_{\text{st}}^{\emptyset,\ell} $ with $\mu^{(1)}\neq \emptyset$. 

This completes the proof.
\end{proof}

\subsection{The SST-diagrams}
  \label{subsec:SST}

Recall we identify $\Lambda_{\text{st}}^{\ell}$ with the subset $\Lambda_{\text{st}}^{\emptyset,\ell} $ of $\Lambda_{\text{st}}^{1+\ell}$ by identifying $\mu=(\mu^{(1)}, \ldots, \mu^{(\ell)})$ with $(\emptyset, \mu^{(1)}, \ldots, \mu^{(\ell)})$.
Suppose that $\lambda\in \Par^\ell(m)$ and $\mu\in \Lambda_{\text{st}}^\ell(m)$. Recall that applying the forgetful map to $\mu$ leads to the   composition $\bar \mu$ of $m$
\[
\bar \mu =(\mu^{(1)}_1,\ldots, \mu^{(1)}_{h_1}, \mu^{(2)}_1,\ldots, \mu^{(2)}_{h_2},\ldots \mu^{(r)}_1,\ldots, \mu^{(r)}_{h_{\ell}}),
\]
where 
$ h_i=l(\mu^{(i)})$, $1\le i\le \ell$. We may identify $\bar \mu $ as a strict composition by omitting the empty components (i.e., for those $h_i=0$).
Associated to any semistandard tableau 
$\mathbf T\in \SST(\lambda,\mu)$  
we define an element $[\mathbf T] \in 
\Hom_{\Sch_{\bfu}}(\lambda,\mu)$ in Step (a)--(b) as follows.

\begin{enumerate}
\item[(a)]
Associated to $\mathbf T$, we have a matrix $A_\mathbf T \in\Mat_{\bar\mu,\bar\lambda} $. Here we index the row  and column of any matrix in $\Mat_{\bar\mu,\bar\lambda}$
in the order 
\begin{align*}
 &
 (1,1),(1,2),\ldots, (1, l(\mu^{(1)}), (2,1), \ldots, (2, l(\mu^{(2)})), \ldots, (\ell,1),\ldots, (\ell, l(\mu^{(\ell)})), \text{ and }
  \\
 & 
 (1,1),(1,2),\ldots, (1, l(\lambda^{(1)}), (2,1), \ldots, (2, l(\lambda^{(2)})), \ldots, (\ell,1),\ldots, (\ell, l(\lambda^{(\ell)})), 
\end{align*}
 respectively.
Then the matrix $A_{\mathbf T}=(a_{(p,i),(q,j)})$ is defined such that 
$a_{(p,i),(q,j)}$ 
is the number of $i_p$ in $j$th row of
$\mathbf T^{(q)}$. Thus, we have a reduced chicken foot diagram $[A_{\mathbf T}]$  of shape $A_{\mathbf T}$.

\item[(b)] 
Let $(A_\mathbf T, (\emptyset))\in \PMat_{\bar \mu,\bar \lambda}$ be the elementary  CFD associated to $[A_\mathbf T]$ which has no dots. Then we choose an ornamentation  $(A_\mathbf T, (\emptyset))\in\PMat_{\mu,\lambda}$ (under the identification via the forgetful map in \eqref{bijPM})
of $(A_\mathbf T, (\emptyset))$. To simplify notation, we denote the resulting ornamentation $(A_\mathbf T, (\emptyset))$ by $[{\mathbf T}]$. 
\end{enumerate}
We also denote by $\div$ the anti-involution on $\Sch_{\bfu}$ induced by $\div$ on $\ASch$ in \eqref{anti-autocyc}. Applying the symmetry $\div$ gives us $[{\mathbf T}]^\div \in \Hom_{\Sch_{\bfu}}(\mu,\la)$. The diagrams of the form $[{\mathbf T}]$ or $[{\mathbf T}]^\div$ are called {\em SST-diagrams (or SST-morphisms)}. 

By Condition (3) in the  definition of semi-standard tableaux in \S \ref{subsec-cychecke}, the diagram $[{\mathbf T}]$ does not involve the traverse-downs $
\begin{tikzpicture}[baseline = 10pt, scale=.8, color=\clr]
 \draw[-,line width=1pt,color=\cred] (-0.3,.3) to (.3,1);
\draw[-,line width=1.2pt] (0.3,.3) to (-.3,1);
\draw(-.3,0.15) node{$\scriptstyle \red{u}$};
\draw (.3, 0.15) node{$\scriptstyle a$};
\end{tikzpicture}
$; similarly, the diagram $[{\mathbf T}]^\div$ does not involve the traverse-ups $\begin{tikzpicture}[baseline = 10pt, scale=.8, color=\clr]
 \draw[-,line width=1.2pt] (-0.3,.3) to (.3,1);
\draw[-,line width=1pt,color=\cred] (0.3,.3) to (-.3,1);
\draw(-.3,0.15) node{$\scriptstyle a$};
\draw (.3, 0.15) node{$\scriptstyle \red{u}$};
\end{tikzpicture}.
$
In particular, there is no crossings in either diagram $[{\mathbf T}]$ or $[{\mathbf T}]^\div$ of the form 
$\begin{tikzpicture}[baseline = 2mm,scale=.7, color=\clr]
\draw[-,line width=1.2pt] (0,0) to (.6,1);
\draw[-,line width=1pt] (0,1) to (.6,0);
\draw[-, thick, color=\cred] (.4,0) to (.4,1);
\draw(.4,-.2)node {$\scriptstyle u_i$};
\end{tikzpicture} 
\text{ or } 
\begin{tikzpicture}[baseline = 2mm,scale=.7, color=\clr]
\draw[-,line width=1.2pt] (0,0) to (.6,1);
\draw[-,line width=1pt] (0,1) to (.6,0);
\draw[-, thick, color=\cred] (.2,0) to (.2,1);
\draw(.2,-.2)node {$\scriptstyle u_i$};
\end{tikzpicture}
$
in $[\mathbf T]$. Then by the relation 
$
\mathord{
\begin{tikzpicture}[baseline = -1mm,scale=0.6,color=\clr]
	\draw[-,thick] (0.45,.6) to (-0.45,-.6);
	\draw[-,thick] (0.45,-.6) to (-0.45,.6);
        \draw[-,thick] (0,-.6) to[out=90,in=-90] (-.45,0);
        \draw[-,thick] (-0.45,0) to[out=90,in=-90] (0,0.6);
        \node at (0,-.77) {$\scriptstyle b$};
        \node at (0.5,-.77) {$\scriptstyle c$};
        \node at (-0.5,-.77) {$\scriptstyle a$};
\end{tikzpicture}
}
=
\mathord{
\begin{tikzpicture}[baseline = -1mm,scale=0.6,color=\clr]
	\draw[-,thick] (0.45,.6) to (-0.45,-.6);
	\draw[-,thick] (0.45,-.6) to (-0.45,.6);
        \draw[-,thick] (0,-.6) to[out=90,in=-90] (.45,0);
        \draw[-,thick] (0.45,0) to[out=90,in=-90] (0,0.6);
        \node at (0,-.77) {$\scriptstyle b$};
        \node at (0.5,-.77) {$\scriptstyle c$};
        \node at (-0.5,-.77) {$\scriptstyle a$};
\end{tikzpicture}
}\:
$
and the relations in Lemma~ \ref{adamovecrossings},  all ornamentations of $(A_\mathbf T, (\emptyset)) $ represent the same morphism   
 $[\mathbf T]\in \Hom_{\Sch_{\bfu}}(\lambda,\mu)$.

\begin{example}
Suppose that $\ell=3$ and $m=12$. Let 
$\lambda=((3,2),(2,2),(2,1))\in \Par^3(12)$ and 
$\mu=((2,1), (1,1,2),(1,1,1,2))\in \Lambda_{\text{st}}^3(12)$.
Let $\mathbf T\in \SST(\lambda,\mu)$ be
\[
\mathbf T= \bigg(\; \ytableaushort{{1_1}{1_1}{1_2},{2_1}{1_3}}, \;\ytableaushort{{2_2}{3_2},{3_2}{4_3}}, \; \ytableaushort{{2_3}{3_3},{4_3}} \; \bigg).
\]
Then 
 $\bar\mu=(2,1,1,1,2,1,1,1,2)$,
 $\bar \lambda=(3,2,2,2,2,1)$ 
 and 
 \[
 A_{\mathbf T}=
 \begin{pmatrix}
     2&0&0&0&0&0\\
     0&1&0&0&0&0\\
     1&0&0&0&0&0\\
     0&0&1&0&0&0\\
     0&0&1&1&0&0\\
     0&1&0&0&0&0\\
     0&0&0&0&1&0\\
     0&0&0&0&1&0\\
     0&0&0&1&0&1\\
 \end{pmatrix}
 \in \Mat_{\bar\mu,\bar \lambda}.
 \] 
 Then the diagram $[\mathbf T]$ is drawn as
 \[
 \begin{tikzpicture}[baseline = 10pt, scale=0.6, color=\clr]
\draw[-,line width=1.5pt](-4,-.3) to  (-4,0);
\draw[-, line width=1pt,color=\cred](-5,-.3) to (-5,2.2);
\draw (-5,-.6) node{$\scriptstyle \red{u_1}$};
\draw[-, line width=1pt](-4,0) to node[left]{$\scriptstyle 2$} (-4,2);
\draw[-, line width=1pt](-4,0) to node[right]{$\scriptstyle 1$} (-1,2);
\draw[-,line width=1.5pt](-3,-.3) to (-3,0);
\draw[-,line width=1pt](-3,0) to node[left]{$\scriptstyle 1$} (-3,2);
\draw[-,line width=1pt](-3,0) to node[below]{$\scriptstyle 1$} (3,2);
\draw[-, line width=1.5pt](-1,-.3) to (-1,0);
\draw[-, line width=1pt](-1,0) to node[left]{$\scriptstyle1$} (0,2);
\draw[-, line width=1pt](-1,0) to (1,2);
\draw[-, line width=1.5pt](0,-.3) to (0,0);
\draw[-,line width=1pt](0,0) to node [below] {$\scriptstyle1$} (1,2);
\draw[-,line width=1pt](0,0) to node [below] {$\scriptstyle1$} (6,2);
\draw[-, line width=1.5pt](3,-.3) to (3,0);
\draw[-, line width=1.5pt](1,2) to (1.2,2.2);
\draw[-, line width=1.5pt](6,2) to (6.2,2.2);
\draw[-,line width=1pt](3,0) to node [below] {$\scriptstyle1$} (4,2);
\draw[-,line width=1pt](3,0) to (6,2);
\draw(5,1) node {$\scriptstyle1$};
\draw[-,line width=1pt](4,-.3) to node [below] {$\scriptstyle1$} (5,2);
\draw[-, line width=1pt,color=\cred](-2,-.3) to (-2,2.2);
\draw (-2,-.6) node{$\scriptstyle \red{u_2}$};
\draw[-, line width=1pt,color=\cred](2,-.3) to (2,2.2);
\draw (2,-.6) node{$\scriptstyle \red{u_3}$};
\draw(-4,-.6) node{$\scriptstyle 3$};
\draw(-3,-.6) node{$\scriptstyle 2$};
\draw(-1,-.6) node{$\scriptstyle 2$};
\draw(0,-.6) node{$\scriptstyle 2$};
\draw(3,-.6) node{$\scriptstyle 2$};
\draw(4,-.6) node{$\scriptstyle 1$};
 \end{tikzpicture}.
 \]
It is clear that no traverse-down appears locally in the above diagram. 
\end{example}

\subsection{Property of the functor $\mathcal G$}

Recall from Theorem \ref{thm:G} the functor $\mathcal G: \Sch_{\bfu}\rightarrow \ScatDJM_{\bfu}$. Recall the cellular basis $\{ \phi_{\mathbf S\mathbf T}\}$ from \eqref{phi}--\eqref{basisofdjm} for the cyclotomic Schur algebras. 

\begin{proposition}
 \label{pro:mapsdjm}
Suppose that $\lambda\in \emph{\Par}^\ell(m)$ and $\mu,\nu\in \Lambda_{\text{st}}^\ell(m).$ Then the functor $\mathcal G$ sends $[\mathbf T]\circ [\mathbf S]^\div$ to $\phi_{\mathbf T\mathbf S}$, for any $\mathbf S\in \emph{\SST}(\lambda,\mu)$ and $\mathbf T\in \emph{\SST}(\lambda,\nu)$.
\end{proposition}

\begin{proof}
Let $\mathbf T^\lambda:=\lambda(\t^\lambda) \in \SST(\lambda,\lambda)$ be the unique semistandard $\lambda$-tableau of type $\lambda$. Then $\phi_{\mathbf T^\lambda \mathbf T^\lambda}$ is the identify morphism in $\Hom_{\mathcal H_{m,\bfu}}(M^\lambda,M^\lambda)$.
Moreover, by definition  $[\mathbf T^\lambda]$ is the identity morphism in $\Hom_{\Sch_{\bfu}}(\lambda,\lambda)$, i.e, the result holds for $ \mathbf T=\mathbf S=\mathbf T^\lambda$ with $\mu=\nu=\lambda$.

Next, we prove the result for $\mathbf S=\mathbf T^\lambda$ with $\nu=\lambda$ in Steps (1)-(3) below.
Let $h_q=|\lambda^{(q)}|$ for $1\le q\le \ell$.

(1) Suppose that $\mu^{(p)}=(1^{h_p})$, for $1\le p\le \ell$.
In this case, we must have $|\mu^{-1}(\mathbf T)|=1$ and $(\mu^{(p)})^{-1}(\mathbf T^{(p)}) \in  \std(\lambda^{(p)})$, $1\le p\le \ell$; see \eqref{def-mu-1} for notation $\mu^{-1}(\cdot)$. Recall the notation $d(\cdot)$ from \eqref{eq:dt}, $m_\la$ from \eqref{def-mlambda} and $m_{*,*}$ from \eqref{def-mst}. Then $m_{\mathbf T,\mathbf T^\lambda}=d(\mu^{-1}(\mathbf T))^* m_\lambda$ and $\phi_{\mathbf T\mathbf T^\lambda}$
is given by the left multiplication of 
\[
d(\mu^{-1}(\mathbf T))^*= \prod_{1\le p\le \ell} d((\mu^{(p)})^{-1}(\mathbf T^{(p)}))^*.
\]
On the other hand, we may draw $ [\mathbf T]$ as follows.
Suppose $\mu^{-1}(\mathbf T)=\t ^\lambda$. Then 
\begin{equation}
\label{equ:case111cl}
[\mu(\t^\lambda)]= 
\begin{tikzpicture}[baseline = 10pt, scale=0.6, color=\clr]
\draw[-,line width=1.5pt](-4,-.3) to (-4,0);
\draw (-4,-.7) node{$\scriptstyle \lambda^{(1)}_1$};
\draw[-,thin](-4,0) to (-5,1.8);
\draw[-, line width=1pt,color=\cred](-5.5,-.3) to (-5.5,2.2);
\draw (-5.5,-.6) node{$\scriptstyle \red{u_1}$};
\draw (-4,1) node{$\ldots$};
\draw[-,thin](-4,0) to (-3,1.8);
\draw (-4,2) node{$\ldots$};
\draw (-5,2) node{$\scriptstyle 1$};
\draw (-3,2) node{$\scriptstyle 1$};
\draw (-2.5,1) node{$\ldots$};
\draw[-,line width=1.5pt](-1,-.3) to (-1,0);
\draw (-1,-.7) node{$\scriptstyle \lambda^{(1)}_{h_1}$};
\draw[-,thin](-1,0) to (-2,1.8);
\draw (-1,1) node{$\ldots$};
\draw[-,thin](-1,0) to (0,1.8);
\draw (-1,2) node{$\ldots$};
\draw (-2,2) node{$\scriptstyle 1$};
\draw (0,2) node{$\scriptstyle 1$};
\end{tikzpicture}
\begin{tikzpicture}[baseline = 10pt, scale=0.6, color=\clr]
\draw[-, line width=1pt,color=\cred](-6,-.3) to (-6,2.2);
\draw (-6,-.6) node{$\scriptstyle \red{u_2}$};
\draw[-,line width=1.5pt](-4,-.3) to (-4,0);
\draw (-4,-.7) node{$\scriptstyle \lambda^{(2)}_1$};
\draw[-,thin](-4,0) to (-5,1.8);
\draw (-4,1) node{$\ldots$};
\draw[-,thin](-4,0) to (-3,1.8);
\draw (-4,2) node{$\ldots$};
\draw (-5,2) node{$\scriptstyle 1$};
\draw (-3,2) node{$\scriptstyle 1$};
\draw (-2.5,1) node{$\ldots$};
\draw[-,line width=1.5pt](-1,-.3) to (-1,0);
\draw (-1,-.7) node{$\scriptstyle \lambda^{(2)}_{h_2}$};
\draw[-,thin](-1,0) to (-2,1.8);
\draw (-1,1) node{$\ldots$};
\draw[-,thin](-1,0) to (0,1.8);
\draw (-1,2) node{$\ldots$};
\draw (-2,2) node{$\scriptstyle 1$};
\draw (0,2) node{$\scriptstyle 1$};
\draw (1.8,1) node{$\ldots$};
\draw[-, line width=1pt,color=\cred](2.5,-.3) to (2.5,2.2);
\draw (2.5,-.6) node{$\scriptstyle \red{u_\ell}$};
\draw[-,line width=1.5pt](4,-.3) to (4,0);
\draw (4,-.7) node{$\scriptstyle \lambda^{(\ell)}_1$};
\draw[-,thin](4,0) to (5,1.8);
\draw (4,1) node{$\ldots$};
\draw[-,thin](4,0) to (3,1.8);
\draw (4,2) node{$\ldots$};
\draw (5,2) node{$\scriptstyle 1$};
\draw (3,2) node{$\scriptstyle 1$};
\draw (5.5,1) node{$\ldots$};
\draw[-,line width=1.5pt](7,-.3) to (7,0);
\draw (7,-.7) node{$\scriptstyle \lambda^{(\ell)}_{h_\ell}$};
\draw[-,thin](7,0) to (6,1.8);
\draw (7,1) node{$\ldots$};
\draw[-,thin](7,0) to (8,1.8);
\draw (7,2) node{$\ldots$};
\draw (6,2) node{$\scriptstyle 1$};
\draw (8,2) node{$\scriptstyle 1$};
 \end{tikzpicture}
\end{equation}
and hence $\mathcal G([\mu(\t^\lambda)])=\phi_{\mu(\t^\lambda) \mathbf T^\lambda}$ since both sides are the map determined by sending $m_\lambda$ to $m_\lambda$.
In general, $[\mathbf T]$ is obtained from 
$\mu(\t^\lambda)$ by vertical composition of the permutation diagram of $d(\mu^{-1}(\mathbf T))^*$, where the permutation diagram of any $w\in\mathfrak S_m$ is obtained via the isomorphism  $\End_{\AW}(1^m)\cong\hat {\mathcal H}_{m}$  (cf. \cite[Proposition~3.13]{SWweb}).
Then we have 
\begin{equation}
\label{equ:casecloumn}
 \mathcal G([\mathbf T])=\mathcal G(d(\mu^{-1}(\mathbf T))^* \circ [\mu(\t^\lambda)])= \mathcal G (d(\mu^{-1}(\mathbf T))^*) \circ \mathcal G([\mu(\t^\lambda)])=\phi_{\mathbf T\mathbf T^\lambda},   
\end{equation}
since $\mathcal G (d(\mu^{-1}(\mathbf T))^*) $ is given by left multiplication of $d(\mu^{-1}(\mathbf T))^*$ by definition.
This completes the proof in the case $\mu^{(p)}=(1^{h_p})$, for $1\le p\le \ell$.

(2) Suppose that $ \mu^{(p)}=(1^{t_p})$, for some $t_p\in \N$ and for $1\le p\le \ell$. Recall that $ \SST(\lambda,\mu)\neq \emptyset$ only if 
$\lambda \rhd \mu$. This implies that 
$\sum_{1\le k\le p}t_k\le \sum_{1\le k\le p}h_k$, for $1\le p\le \ell$. In this case, we also have $|\mu^{-1}(\mathbf T)|=1$. Moreover, we have
\[ 
[\mu(\t^\lambda)]= 
\begin{tikzpicture}[baseline = 10pt, scale=0.6, color=\clr]
\draw[-,line width=1.5pt](-4,-.3) to (-4,0);
\draw (-4,-.7) node{$\scriptstyle \lambda^{(1)}_1$};
\draw[-,thin](-4,0) to (-5,1.8);
\draw (-4,1) node{$\ldots$};
\draw[-,thin](-4,0) to (-3,1.8);
\draw (-4,2) node{$\ldots$};
\draw (-5,2) node{$\scriptstyle 1$};
\draw (-3,2) node{$\scriptstyle 1$};
\draw (-2.5,1) node{$\ldots$};
\draw[-,line width=1.5pt](-1,-.3) to (-1,0);
\draw (-1,-.7) node{$\scriptstyle \lambda^{(1)}_{h_1}$};
\draw[-,thin](-1,0) to (-2,1.8);
\draw (-1,1) node{$\ldots$};
\draw[-,thin](-1,0) to (0,1.8);
\draw (-1,2) node{$\ldots$};
\draw (-2,2) node{$\scriptstyle 1$};
\draw (0,2) node{$\scriptstyle 1$};
\draw[-, line width=1pt,color=\cred](1,-.3) to (-1,1.8);
\draw (1,-.6) node{$\scriptstyle \red{u_2}$};
\draw (1.8,1) node{$\ldots$};
\draw[-, line width=1pt,color=\cred](-5.5,-.3) to (-5.5,1.8);
\draw (-5.5,-.6) node{$\scriptstyle \red{u_1}$};
\draw[-,line width=1.5pt](4,-.3) to (4,0);
\draw (4,-.7) node{$\scriptstyle \lambda^{(\ell-1)}_{\ell-1}$};
\draw[-,thin](4,0) to (5,1.8);
\draw (4,1) node{$\ldots$};
\draw[-,thin](4,0) to (3,1.8);
\draw (4,2) node{$\ldots$};
\draw (5,2) node{$\scriptstyle 1$};
\draw (3,2) node{$\scriptstyle 1$};
\draw[-, line width=1pt,color=\cred](5.5,-.3) to (3.8,1.8);
\draw (5.5,-.6) node{$\scriptstyle \red{u_\ell}$};
\draw (5.5,1) node{$\ldots$};
\draw[-,line width=1.5pt](7,-.3) to (7,0);
\draw (7,-.7) node{$\scriptstyle \lambda^{(\ell)}_{h_\ell}$};
\draw[-,thin](7,0) to (6,1.8);
\draw (7,1) node{$\ldots$};
\draw[-,thin](7,0) to (8,1.8);
\draw (7,2) node{$\ldots$};
\draw (6,2) node{$\scriptstyle 1$};
\draw (8,2) node{$\scriptstyle 1$};
 \end{tikzpicture}\]
i.e., $\mu(\t^\lambda)$ is the  ornamentation of the same underlying CFD in \eqref{equ:case111cl}.
Then by definition $\mathcal G(\mu(\t^\lambda))=\phi_{\mu(\t^\lambda)\mathbf T^\lambda}$ since both are the inclusion map sending $m_\lambda$ to $m_\lambda$.
 The case for general  $\mathbf T\in \SST(\lambda,\mu)$ follows by repeating \eqref{equ:casecloumn}.  

(3) Let $\mathbf v=(t_1,\ldots, t_\ell)$ be as given in Step (2) above. Define $ \Lambda(\mathbf v)$ to be the subset of $\Lambda_{\text{st}}^\ell(m)$ consisting of $\mu$ such that $|\mu^{(p)}|=t_p$, for $1\le p\le \ell$.
Then $ \Lambda_{\text{st}}^\ell(m)=\bigsqcup_{\mathbf t}\Lambda(\mathbf v)$.
We now proceed by induction on the order $\lhd$ of each set $\Lambda(\mathbf v)$.
The minimal case is $\mu^{(p)}=(1^{t_p})$, $1\le p\le \ell$, and in this case the result is proven in (2). In general, there are some $p,i$ such that 
\[
\mu^{(p)}_i>1, \text{ and }
\mu^{(k)}_s=1 \text{ for all } k< p \text{ or } k=p \text{ and } s<i. 
\]
We shall separate the discussion into two subcases (i)--(ii) below.
\begin{itemize}
    \item[(i)] Suppose that $a_{(p,i), (q,j)}= \mu^{(p)}_i$, for some $\lambda^{(q)}_j$. 
    Then  $a_{(p,i), (q',j')}=0$ for all $(q',j')\neq (q,j)$ and there is no merge at the vertex $\mu^{(p)}_i$ in $[\mathbf T]$.
    Moreover, there is a component in $\mathbf T$ of the form  
    \[\ytableaushort{{i_p}{i_p}{\ldots}{i_p}}. \]
Let $\mu'\in \Lambda_{\text{st}}^\ell(m)$ be obtained from $\mu$ by splitting $\mu^{(p)}_i$ to $(\mu^{(p)}_i-1,1)$. Then $\mu'\lhd \mu$.
Let $\mathbf T'\in \SST(\lambda,\mu')$ be obtained from $\mathbf T$ by changing the last $i_p$  of the above component to $(i+1)_p$ and all other $k_p$ to $(k+1)_p$ for all $k>i$.
Then 
$\mu^{(p)}_i[\mathbf T]= \pi \circ [\mathbf T'] $, where 
\[\pi= 1_*
\begin{tikzpicture}[baseline = -.5mm,color=\clr]
\draw[-,line width=1pt] (0.28,-.3) to (0.08,0.04);
\draw[-,line width=1pt] (-0.12,-.3) to (0.08,0.04);
\draw[-,line width=1.5pt] (0.08,.4) to (0.08,0);
%\node at (-0.22,-.4) {$\scriptstyle a$};
\node at (0.35,-.4) {$\scriptstyle 1$};
\node at (0.06,.7){$\scriptstyle \mu^{(p)}_i$};
\end{tikzpicture}
1_* .\]
Note that $\mu(\s)=\mathbf T$ if and only if $\mu'(\s)=\mathbf T'$ for any $\s\in \std(\lambda)$.
Then by the inductive assumption, we have 
\begin{align*}
\mathcal G(\mu^{(p)}_i[\mathbf T] ) 
&= \mathcal G( \pi \circ [\mathbf T'] )
= \mathcal G(\pi) \circ \mathcal G([\mathbf T'])
\\
&= \sum_{w\in (\mathfrak S_{\mu}/\mathfrak S_{\mu'})_{\text{min}}}w\phi_{\mathbf T' \mathbf T^\lambda}=
\sum_{w\in (\mathfrak S_{\mu}/\mathfrak S_{\mu'})_{\text{min}}}w \phi_{\mathbf T \mathbf T^\lambda}=\mu^{(p)}_i\phi_{\mathbf T \mathbf T^\lambda},
\end{align*}
 since $w m_{\s,\t^\lambda}=m_{\s,\t^\lambda}$ for any $\s \in \mu^{-1}(\mathbf T)$ and  any $w\in\mathfrak S_{\mu}/\mathfrak S_{\mu'}$. This implies that $\mathcal G([\mathbf T])=\phi_{\mathbf T \mathbf T^\lambda}$ by first working over $\Z[\tilde u_1,\ldots, \tilde u_\ell]$ and then applying base change argument as before. 

\item[(ii)] Suppose that 
$\{a_{(p,i), (q,j)}\neq 0\mid 1\le q\le \ell, 1\le j\le l(\bar\lambda)\}$ for a given $(p,i)$ has cardinality $h\ge 2$ and is denoted by $\{a_{(p,i),(p_1,i_1)},\ldots, a_{(p,i), (p_h,i_h)}\}$.
Then $[\mathbf T]$ has a component of the form 
 \[\ytableaushort{{}{}{}{}{}{}{\ldots}{}{}{}{}{i_p}{i_p}{\ldots}{i_p},
 {}{}{}{}{\ldots}{}{i_p}{i_p}{\ldots}{i_p},
 {}{}{\ldots}{}{}{}{}{}{},
 {}{}{i_p}{i_p}{\ldots}{i_p}}, \]
where there are $h$ rows and the number $i_p$ in each row is $a_{(p,i),(p_k,i_k)}$ for $1\le k\le h$.
Write $b= a_{(p,i), (p_h,i_h)}$ and $a= \mu^{(p)}_i-b$. 
Let $\mu'$ be obtained from $\mu$ by splitting $\mu^{(p)}_i$
to $(a,b)$. Thus, $\mu'\lhd \mu$.
Define $\mathbf T'\in \SST(\lambda,\mu')$ to be obtained from 
$\mathbf T$ by replacing $i_p$ in the last row with $(i+1)_p$ and  $k_p$ with $(k+1)_p$ for $k>i$.
Then 
$ [\mathbf T]= \pi\circ [\mathbf T']$, where 
\[\pi= 1_*
\begin{tikzpicture}[baseline = -.5mm,color=\clr]
	\draw[-,line width=1pt] (0.28,-.3) to (0.08,0.04);
	\draw[-,line width=1pt] (-0.12,-.3) to (0.08,0.04);
	\draw[-,line width=1.5pt] (0.08,.4) to (0.08,0);
        \node at (-0.22,-.4) {$\scriptstyle a$};
        \node at (0.35,-.4) {$\scriptstyle b$};
        \node at (0.06,.7){$\scriptstyle \mu^{(p)}_i$};\end{tikzpicture} 
1_*. \]
Note that $\s'\in (\mu')^{-1}(\mathbf T')$ if and only if $\s=\s'w^*\in \mu^{-1}(\mathbf T)$ for some $w\in(\mathfrak S_{\mu}/\mathfrak S_{\mu'})_{\text{min}}$. 
This implies that
\begin{align*}
  \phi_{\mathbf T \mathbf T^\lambda}
  &=\sum_{w\in (\mathfrak S_{\mu}/\mathfrak S_{\mu'})_{\text{min}}}w \phi_{\mathbf T' \mathbf T^\lambda} \\
  &= \mathcal G(\pi) \circ \mathcal G([\mathbf T']), \quad \text{by inductive assumption on $\mu'$}\\
  &= \mathcal G(\pi \circ [\mathbf T'])= \mathcal G([\mathbf T]).
\end{align*}
\end{itemize}

The same argument shows that $ \mathcal G([\mathbf S]^\div)= \phi_{\mathbf T^\lambda \mathbf S}$.
Then, we have $\mathcal G([\mathbf T]\circ [\mathbf S]^\div)=\phi_{\mathbf T \mathbf T^\lambda}\circ \phi_{\mathbf T^\lambda\mathbf S}= \phi_{\mathbf T\mathbf S}$. This completes the proof.
\end{proof}

\begin{rem}
    We remark that there is no dots in the basis given in Proposition \ref{pro:mapsdjm}.
\end{rem}

%% file: Section5_basescycSchur.tex
\section{Bases for a cyclotomic Schur category}
\label{sec:basis-cycschur}

In this section we construct an elementary diagram basis and a double SST basis for the cyclotomic Schur category $\Sch_{\bfu}$. We then show that $\Sch_{\bfu}$ is isomorphic to $\ScatDJM_{\bfu}$, which matches the double SST basis of $\Sch_{\bfu}$ with the cellular basis of $\ScatDJM_{\bfu}$.

\subsection{ A spanning set of $\Sch_{\bfu}$}
 
Recall the spanning set $\PMat_{\lambda,\mu}$ of $\Hom_{\ASch}(\mu,\lambda)$. For any $\lambda,\mu\in\Lambda_{\text{st}}^{\emptyset,\ell}(m)$,
an element $(A,P)$ of $\PMat_{\lambda,\mu}$ can be written as 
a pair of $\ell \times \ell$-block matrices $(A=(A_{pq})_{p,q=1}^\ell, P=(P_{pq})_{p,q=1}^\ell)$ via the bijective map \eqref{bijPM}.  We may also further express such a pair more explicitly as 
\[
\big(A=(a_{(p,i),(q,j)}), P=(\eta_{(p,i),(q,j)}) \big)
\]
with each of the block matrices given by $A_{pq} =(a_{(p,i),(q,j)})_{i,j}$ and $P_{pq} =(\eta_{(p,i),(q,j)})_{i,j}$, where $\eta_{(p,i),(q,j)}$ is a partition consisting of parts not exceeding $a_{(p,i),(q,j)}$. 

%For $1\le h\le \ell$ and $1\le k\le \ell$,  let $\PMat_{\lambda^{(h)},\mu^{(k)}}^\flat $ be the subset of $\PMat_{\lambda^{(h)},\mu^{(k)}} $ consists of all $(A=(a_{ij}), P=(\eta_{ij})) $ such that $l(\eta_{hk})\le \min\{h,k\}-1$,  for all  $1\le i\le l(\lambda^{(h)}), 1\le j\le l(\mu^{(k)})$.

For $\lambda,\mu\in \Lambda_{\text{st}}^{\emptyset,\ell}(m)$, we define the set of bounded-partition-enhanced $\ell\times \ell$-block $\N$-matrices
\begin{align}
     \label{PM0}
\PMat_{\lambda,\mu}^\flat:= \big\{ (A,P)\in\PMat_{\lambda,\mu}\mid   l(\eta_{(p,i),(q,j)})\le \min\{p,q\}-1, \forall i,j,p,q 
 \big \}.
\end{align}

\begin{proposition}  
\label{pro:spanofcycshcurcate}
For $\lambda,\mu\in \Lambda_{\text{st}}^\ell(m)$, the Hom-space $\Hom_{\Sch_{\bfu}}(\mu,\lambda)$
is spanned by $\PMat_{\lambda,\mu}^\flat$.
\end{proposition}

\begin{proof}
We prove by induction on the degree $k$  of $(A,P)$ in $\PMat_{\lambda,\mu}$. Let $k_0$ be the minimal degree of a diagram in $\PMat_{\lambda,\mu}$. Note that by Definition \ref{def:degree} $k_0$ is not always zero since the traverse-down has positive degree.   The cases $P=(\emptyset)$, i.e., each matrix entry of $P$ is $\emptyset$ are already in $\PMat_{\lambda,\mu}^\flat$ since $l(\emptyset)=0$ satisfies the condition in \eqref{PM0}. This includes 
the basic case $k=k_0$  since dots are of positive degree and hence $P$ must be $(\emptyset)$ to attain minimum.   Suppose  $P_{pq}\neq (\emptyset)$, for some $p,q$ (and hence $k>k_0$). We have 
\begin{align}
\label{componentone}
\begin{tikzpicture}[baseline = -1mm,scale=.8,color=\clr]
\draw[-,line width=1.2pt] (0.08,-.6) to (0.08,.5);
\node at (.08,-.8) {$\scriptstyle a$};
\draw(0.08,0) \bdot;
\draw(.6,0)node {$\scriptstyle \omega_{r}$};
\draw[-,line width=1pt,color=\cred](-.4,-.6)to (-.4,.5);
\draw (-.4,-.8) node{$\scriptstyle \red{u_1}$};
\end{tikzpicture}
\overset{\eqref{splitmerge}}=
\begin{tikzpicture}[baseline = -1mm,scale=.8,color=\clr]
\draw[-,line width=1.5pt] (0.08,-.6) to (0.08,-.3);
\draw[-,line width=1.5pt] (0.08,.5) to (0.08,.8);
\draw[-,thick] (0.1,-.31) to [out=45,in=-45] (0.1,.51);
\draw[-,thick] (0.06,-.31) to [out=135,in=-135] (0.06,.51);
\draw(-.1,0.2) \bdot;
\draw(-.4,0.2)node {$\scriptstyle \omega_r$};
\node at (-.3,-.15) {$\scriptstyle r$};
\node at (.6,-.15) {$\scriptstyle a-r$};  
\draw[-,line width=1pt,color=\cred](-.8,-.6)to (-.8,.5);
\draw (-.8,-.8) node{$\scriptstyle \red{u_1}$};
\end{tikzpicture}
\overset{\text{Lem }\ref{dotmovefreely}(6)}\equiv
\begin{tikzpicture}[baseline = 10pt, scale=0.4, color=\clr]
\draw[-,line width=1pt,color=\cred](-.15,-.6)to(-.15,2.4);
\draw (-.15,-.9) node{$\scriptstyle \red{u_1}$};
\draw[-,thick](.5,-.2) to[out=135, in=down] (-.5,1);
\draw[-,thick](-.5,1) to[out=up, in=270] (.5,2);
\draw[-,thick] (.5,-.2) to [out=right,in=right] 
(0.5,2);
\draw[-,line width=1.5pt](.5, -.5) to (.5,-.2);
\draw[-,line width=1.5pt](.5,2) to (.5,2.3);
\draw (1.5,-0.2) node{$\scriptstyle {a-r}$};
\draw (.53,-0.85) node{$\scriptstyle {a}$};
\end{tikzpicture}
\overset{\eqref{cyclotomicpoly}}=0.
\end{align}
% Recall that 
% \begin{align}
% \label{symmetric}
% \mathord{
% \begin{tikzpicture}[baseline = -1mm,scale=0.6,color=\clr]
% 	\draw[-,thick] (0.28,0) to[out=90,in=-90] (-0.28,.6);
% 	\draw[-,thick] (-0.28,0) to[out=90,in=-90] (0.28,.6);
% 	\draw[-,thick] (0.28,-.6) to[out=90,in=-90] (-0.28,0);
% 	\draw[-,thick] (-0.28,-.6) to[out=90,in=-90] (0.28,0);
%         \node at (0.3,-.75) {$\scriptstyle b$};
%         \node at (-0.3,-.75) {$\scriptstyle a$};
% \end{tikzpicture}
% }&=
% \mathord{
% \begin{tikzpicture}[baseline = -1mm,scale=0.6,color=\clr]
% 	\draw[-,thick] (0.2,-.6) to (0.2,.6);
% 	\draw[-,thick] (-0.2,-.6) to (-0.2,.6);
%         \node at (0.2,-.75) {$\scriptstyle b$};
%         \node at (-0.2,-.75) {$\scriptstyle a$};
% \end{tikzpicture}
% }\:,\end{align}

Thus, we may use Lemma \ref{dotmovefreely} to move the dots between the red strands of $u_1$ and $u_2$ next to the red strand of $u_1$, and 
then use \eqref{componentone} to reduce to a case with $P_{p,1}=(\emptyset)$ and $P_{1,q}=(\emptyset)$.

In general, suppose that $P_{pq}=(\eta_{ij})$ with 
$l(\eta_{ij})=h\ge t:=\min\{ p, q\}$ for some $\eta_{ij}$. Let $\bar {\eta}$ denote the partition obtained from $\eta_{ij}$ by deleting its first part, say $r$. We have
\begin{align*}
\begin{tikzpicture}[baseline = -1mm,scale=.8,color=\clr]
\draw[-,line width=1.2pt] (0.08,-.6) to (0.08,.5);
\node at (.08,-.8) {$\scriptstyle a$};
\draw(0.08,0) \bdot;
\draw(.6,0)node {$\scriptstyle \omega_{\eta_{ij}}$};
\draw[-,line width=1pt,color=\cred](-.4,-.6)to (-.4,.5);
\draw (-.4,-.8) node{$\scriptstyle \red{u_t}$};
\end{tikzpicture}
&\overset{\eqref{splitmerge}}=
\begin{tikzpicture}[baseline = -1mm,scale=.8,color=\clr]
\draw[-,line width=1.5pt] (0.08,-.8) to (0.08,-.3);
\draw[-,line width=1.5pt] (0.08,.5) to (0.08,.8);
\draw[-,thick] (0.1,-.31) to [out=45,in=-45] (0.1,.51);
\draw[-,thick] (0.06,-.31) to [out=135,in=-135] (0.06,.51);
\draw(-.1,0.2) \bdot;
\draw(0.08,-.5) \bdot;
\draw(.5,-.5)node {$\scriptstyle \omega_{\bar{\eta}}$};
\draw(-.4,0.2)node {$\scriptstyle \omega_{r}$};
\node at (-.3,-.15) {$\scriptstyle r$};
\node at (.6,-.15) {$\scriptstyle a-r$};  
\draw[-,line width=1pt,color=\cred](-.8,-.6)to (-.8,.5);
\draw (-.8,-.8) node{$\scriptstyle \red{u_t}$};
\end{tikzpicture}
\overset{\eqref{adaptorL}}\equiv
\begin{tikzpicture}[baseline = 10pt, scale=0.4, color=\clr]
\draw[-,line width=1pt,color=\cred](0,-.6)to (0,2.4);
\draw (0,-.9) node{$\scriptstyle \red{u_t}$};
\draw[-,thick](.5,-.2) to[out=135, in=down] (-.5,1);
\draw[-,thick](-.5,1) to[out=up, in=270] (.5,2);
\draw[-,thick] (.5,-.2) to [out=right,in=right] 
(0.5,2);
\draw[-,line width=1.5pt](.5, -.9) to (.5,-.2);
\draw[-,line width=1.5pt](.5,2) to (.5,2.3);
\draw (1.8,0.2) node{$\scriptstyle {a-r}$};
\draw(0.5,-.5) \bdot;
\draw(1.2,-.6)node {$\scriptstyle \omega_{\bar{\eta}}$};
\end{tikzpicture}
\\
&
\overset{\eqref{dotmovesplitss}}{\underset{\eqref{dotmoveadaptor}}{\equiv}}
\sum_{\overset{\mathbf t,\mathbf v}{\underset{l(\mathbf t)\le h-1}{l(\mathbf v)\le p-1}}} 
\begin{tikzpicture}[baseline = 10pt, scale=0.4, color=\clr]
\draw[-,line width=1pt,color=\cred](0,-.6)to (0,2.4);
\draw (0,-.9) node{$\scriptstyle \red{u_t}$};
\draw[-,thick](.5,-.2) to[out=135, in=down] (-.5,1);
\draw[-,thick](-.5,1) to[out=up, in=270] (.5,2);
\draw[-,thick] (.5,-.2) to [out=right,in=right] 
(0.5,2);
\draw[-,line width=1.5pt](.5, -.9) to (.5,-.2);
\draw[-,line width=1.5pt](.5,2) to (.5,2.3);
\draw (1.6,0) node{$\scriptstyle {a-r}$};
\draw(-0.5,1) \bdot;
\draw(-1,1)node {$\scriptstyle \omega_{\mathbf t}$};
\draw(1.15,1) \bdot;
\draw(1.8,1)node {$\scriptstyle \omega_{\mathbf  v}$};
\end{tikzpicture}\\
&\equiv
\sum_{\overset{\mathbf z,\mathbf v}{\underset{l(\mathbf z)< t-1}{l(\mathbf v)\le h-1}}} 
\begin{tikzpicture}[baseline = 10pt, scale=0.4, color=\clr]
\draw[-,line width=1pt,color=\cred](0,-.6)to (0,2.4);
\draw (0,-.9) node{$\scriptstyle \red{u_t}$};
\draw[-,thick](.5,-.2) to[out=135, in=down] (-.5,1);
\draw[-,thick](-.5,1) to[out=up, in=270] (.5,2);
\draw[-,thick] (.5,-.2) to [out=right,in=right] 
(0.5,2);
\draw[-,line width=1.5pt](.5, -.9) to (.5,-.2);
\draw[-,line width=1.5pt](.5,2) to (.5,2.3);
\draw (1.6,0) node{$\scriptstyle {a-r}$};
\draw(-0.5,1) \bdot;
\draw(-1,1)node {$\scriptstyle \omega_{\mathbf z}$};
\draw(1.15,1) \bdot;
\draw(1.8,1)node {$\scriptstyle \omega_{\mathbf  v}$};
\end{tikzpicture}, ~\quad \text{by induction on $t-1$}
\\
&\overset{\text{Lem} \ref{dotmovefreely}}{\underset{\eqref{dotmoveadaptor}}{\equiv}}
\sum_{\overset{\mathbf z,\mathbf v}{\underset{l(\mathbf z)< t-1}{l(\mathbf v)\le h-1}}} 
\begin{tikzpicture}[baseline = 10pt, scale=0.4, color=\clr]
\draw[-,line width=1pt,color=\cred](-1.4,-.6)to (-1.4,2.4);
\draw (-1.4,-.9) node{$\scriptstyle \red{u_t}$};
\draw[-,thick](.5,-.2) to[out=left, in=left] (.5,2);
%\draw[-,thick](-.3,1) to[out=up, in=270] (.5,2);
\draw[-,thick] (.5,-.2) to [out=right,in=right] 
(0.5,2);
\draw[-,line width=1.5pt](.5, -.9) to (.5,-.2);
\draw[-,line width=1.5pt](.5,2) to (.5,2.3);
\draw (1.6,0) node{$\scriptstyle {a-r}$};
\draw(-0.15,1) \bdot;
\draw(-.7,1)node {$\scriptstyle \omega_{r}$};
\draw(0,.25) \bdot;
\draw(-.2,-.2)node {$\scriptstyle \omega_{\mathbf z}$};
\draw(1.15,1) \bdot;
\draw(1.8,1)node {$\scriptstyle \omega_{\mathbf  v}$};
\end{tikzpicture}\\
&\equiv
\sum_{\mathbf y, l(\mathbf y)\le h-1}
\begin{tikzpicture}[baseline = -1mm,scale=.8,color=\clr]
\draw[-,line width=1.2pt] (0.08,-.6) to (0.08,.5);
\node at (.08,-.8) {$\scriptstyle a$};
\draw(0.08,0) \bdot;
\draw(.4,0)node {$\scriptstyle \omega_{\mathbf y}$};
\draw[-,line width=1pt,color=\cred](-.4,-.6)to (-.4,.5);
\draw (-.4,-.8) node{$\scriptstyle \red{u_t}$};
\end{tikzpicture}\\
& \equiv
\sum_{\mathbf y, l(\mathbf y)\le t-1}
\begin{tikzpicture}[baseline = -1mm,scale=.8,color=\clr]
\draw[-,line width=1.2pt] (0.08,-.6) to (0.08,.5);
\node at (.08,-.8) {$\scriptstyle a$};
\draw(0.08,0) \bdot;
\draw(.4,0)node {$\scriptstyle \omega_{\mathbf y}$};
\draw[-,line width=1pt,color=\cred](-.4,-.6)to (-.4,.5);
\draw (-.4,-.8) node{$\scriptstyle \red{u_t}$};
\end{tikzpicture}, \quad \text{by induction on $h-1$},
\end{align*}
where the second last ``$\equiv$" follows from induction assumption on $r$ and \cite[Corollary~3.12]{SWweb}. 
%[Corollary \ref{cor:blambdamu}]{SWweb}.
This together with Lemma \ref{dotmovefreely}(4) shows inductively that a general morphism in the Hom-space $\Hom_{\Sch_{\bfu}}(\mu,\lambda)$ can be written as a linear combination of $\PMat_{\lambda,\mu}^\flat$. The proposition is proved.
\end{proof}

\subsection{The double SST basis}

In this subsection, we shall construct a double SST basis for the path algebra of the cyclotomic Schur category $\Sch_{\bfu}$:
 \begin{align} \label{Schurm}
 \Sch_\bfu(m):= \bigoplus_{\mu,\nu\in \Lambda_{\text{st}}^\ell(m)}\Hom_{\Sch_{\bfu}}(\mu,\nu),
 \qquad \text{ for } m \ge 1.
 \end{align}
 We then establish an algebra isomorphism $\Sch_\bfu(m) \cong \Sc_{m,\bfu}$ and identify the double SST basis for $\Sch_\bfu(m)$ with the cellular basis $\{\phi_{\bf S\bf T}\}$ for $\Sc_{m,\bfu}$ in \eqref{basisofdjm}. 
 In particular, under the isomorphism between Schur category and web category \cite{BEEO}, the $\ell=1$ case follows from the so-called straightening algorithm for codeterminants, see \cite[Lemma 7.4]{Bru24}.
 
For any $\lambda\in \Par^\ell(m)$, let $\Sch^{\rhd\lambda}(m)$ be the span of all morphisms factoring  through any object $\gamma\in \Par^\ell(m)$ with $\gamma\rhd \lambda$. Equivalently, $\Sch^{\rhd\lambda}(m)$ is the two-sided ideal of the algebra $\Sch_\bfu(m)$ generated by all $1_{\gamma}$, for $\gamma\in \Par^\ell(m)$ with $\gamma\rhd \lambda$. The following is a technical preparation for the basis theorem for $\Sch_\bfu(m)$. 
 
\begin{proposition}
 \label{pro:spancycdjm}
For any $\mu\in \Lambda^\ell_{\text{st}}(m)$ and $\lambda\in \emph{\Par}^\ell(m)$, we have
\[
\Hom_{\Sch_{\bfu}}(\lambda,\mu)+\Sch^{\rhd\lambda}(m)= \kk\text{-span} \big\{ [\mathbf T] + \Sch^{\rhd\lambda}(m)\mid \mathbf T\in \emph{\SST}(\lambda,\mu) \big\}.
\]
\end{proposition}
 
\begin{proof}
Denote the span on the right-hand side in the formula above by $\text{Span}_{\lambda,\mu}$. Thanks to  Proposition \ref{pro:spanofcycshcurcate}, it suffices to show that 
$(A,P)  +\Sch^{\rhd\lambda}(m)\in \text{Span}_{\lambda,\mu}$ for any 
$(A,P)\in\Mat_{\mu,\lambda}^\flat$.
We proceed by induction on $\ell$ and on the order $\rhd$ of the set $\Par^\ell(m)$.
For $\ell=1$, the result follows from 
the cellular basis of the Schur algebra (e.g., \cite{DJM98}) and  the isomorphism between the web category and the Schur category in \cite{BEEO}.

Suppose $\ell\ge 2$.  
If $\lambda=((m), \emptyset,\ldots,\emptyset)$, then 
 $\Hom_{\Sch_{\bfu}}(\lambda,\mu)$ is spanned by 
\[\{(A,(\emptyset))\in \Par\Mat_{\bar\mu,\bar\lambda}\}.\]
Since $\lambda$ is a row in the first component,  any $A$
above is obtained as $A_{\mathbf T}$ for some $ \mathbf T\in \SST(\lambda,\mu)=\SST(\bar\lambda,\bar \mu)$. 
Thus, $(A,(\emptyset))=[\mathbf T]\in \text{Span}_{\lambda,\mu}$. 
\vspace{2mm}

{\bf Claim.} For any $A\in \Mat_{\bar\mu,\bar \lambda}$, we have
\begin{enumerate}
    \item if  $P\neq (\emptyset)$, then $(A,P)+ \Sch^{\rhd\lambda}(m) = 0$  up to lower degree terms; 
    \item $(A,(\emptyset))\in \Sch^{\rhd\lambda}(m)$ if $a_{(p,i),(q,j)}\neq 0$ for  some $q>p$. 
\end{enumerate}
Here the set $\Mat_{\bar\mu,\bar\lambda}$ is given below \eqref{exam-CFD}.
We prove Claim (1) first.
Suppose $P_{pq}\neq \emptyset$ for some $p,q$.
By assumption, $(A,P)$ factors through a morphism 
\[ 
1_*~ 
\begin{tikzpicture}[baseline = 10pt, scale=0.5, color=\clr]
\draw[-,line width =1pt,color=\cred] (-4.5,.2) to (-4.5,2);
\draw (-4.5,-.2) node{$\scriptstyle \red{u_q}$};    
\draw[-] (-3.5,.2) to (-3.5,2);
\draw (-3.5,-.2) node{$\scriptstyle {\lambda^{(q)}_1}$};
\draw(-2.6,.8) node{$\ldots$};
\draw[-] (-2,.2) to (-2,2);
\draw (-2,-.2) node{$\scriptstyle \lambda^{(q)}_{j-1}$};
\draw[-] (0,.2) to (-1,2);
\draw[-,line width=1.5pt](0,.2) to (0,0);
\draw[-,](0,.2) to (1,2);
\draw (-.5,1) \bdot;
\draw (-1,1) node{$\scriptstyle {\omega_r}$};
\draw (-0.5,.3) node{$\scriptstyle r$}; 
\end{tikzpicture}
~1_* 
\in \Hom_{\Sch_{\bfu}}(\lambda,\nu)
\]
for some $\nu\in \Lambda_{\text{st}}^\ell(m)$, $1\le q\le \ell$, and  $j$ with $1\le r\le \lambda^{(q)}_{j}$.
We have 
\begin{align*}
1_*\begin{tikzpicture}[baseline = 10pt, scale=0.5, color=\clr]
\draw[-,line width =1pt,color=\cred] (-4.5,.2) to (-4.5,2);
\draw (-4.5,-.2) node{$\scriptstyle \red{u_q}$};    
\draw[-] (-3.5,.2) to (-3.5,2);
\draw (-3.5,-.2) node{$\scriptstyle {\lambda^{(q)}_1}$};
\draw(-2.6,.8) node{$\ldots$};
\draw[-] (-2,.2) to (-2,2);
\draw (-2,-.2) node{$\scriptstyle \lambda^{(q)}_{j-1}$};
\draw[-] (0,.2) to (-1,2);
\draw[-,line width=1.5pt](0,.2) to (0,0);
\draw[-,](0,.2) to (1,2);
\draw (-.5,1) \bdot;
\draw (-1,1) node{$\scriptstyle {\omega_r}$};
\draw (-0.5,.3) node{$\scriptstyle r$}; 
\end{tikzpicture}1_*
&
\overset{\text{Lem} \ref{dotmovefreely}}\equiv
1_*\begin{tikzpicture}[baseline = 10pt, scale=0.5, color=\clr]
\draw[-,line width =1pt,color=\cred] (-4.5,.2) to (-4.5,2);
\draw (-4.5,-.2) node{$\scriptstyle \red{u_q}$}; 
\draw[-] (-3.5,1) to[out=up,in=down] (0,2);
\draw[-] (-3.5,1) to[out=down,in=up] (1,.2);
\draw[-] (-2.5,.2) to (-2.5,2);
\draw (-2.5,-.4) node{$\scriptstyle {\lambda^{(q)}_1}$};
\draw(-2,.8) node{$\ldots$};
\draw[-] (-1,.2) to (-1,2);
\draw (-1,-.4) node{$\scriptstyle \lambda^{(q)}_{j-1}$};
%\draw[-] (1,.2) to (0,2);
\draw[-,line width=1.5pt](1,.2) to (1,0);
\draw[-,](1,.2) to (2,2);
\draw (-3.5,1) \bdot;
\draw (-4,1) node{$\scriptstyle {\omega_r}$};
\draw (-3.5,.3) node{$\scriptstyle r$}; 
\end{tikzpicture}1_*
\\
&\overset{\eqref{adaptorL}}{\equiv}
1_*\begin{tikzpicture}[baseline = 10pt, scale=0.5, color=\clr]
\draw[-,line width =1pt,color=\cred] (-4.5,0) to (-4.5,2);
\draw (-4.5,-.4) node{$\scriptstyle \red{u_q}$}; 
\draw[-] (-5,1) to[out=45,in=down] (0,2);
\draw[-](-5,.9) to (-5,1);
\draw[-] (-5,.9) to[out=-45,in=135] (1,.2);
\draw[-] (-2.5,.2) to (-2.5,2);
\draw (-2.5,-.4) node{$\scriptstyle {\lambda^{(q)}_1}$};
\draw(-2,.8) node{$\ldots$};
\draw[-] (-1,.2) to (-1,2);
\draw (-1,-.4) node{$\scriptstyle \lambda^{(q)}_{j-1}$};
%\draw[-] (1,.2) to (0,2);
\draw[-,line width=1.5pt](1,.2) to (1,0);
\draw[-,](1,.2) to (2,2);
%\draw (-3.5,1) \bdot;
%\draw (-4,1) node{$\scriptstyle {\omega_r}$};
\draw (-3.8,.1) node{$\scriptstyle r$}; 
\end{tikzpicture}
1_*
\end{align*}
which factors through $1_{\gamma}$
with 
\[
\gamma= \Big(\lambda^{(1)}, \ldots, (\lambda^{(q-1)},r), (\lambda^{(q)}_1,\ldots, \lambda^{(q)}_{j-1},\lambda^{(q)}_{j}-r, \lambda^{(q)}_{j+1},\ldots), \lambda^{(q+1)}, \ldots, \lambda^{(\ell)} \Big).
\]
Denote by $\text{sort}(\gamma)$ the multipartition obtained by component-wise sorting of the multicomposition $\gamma$. It follows by $ \gamma\rhd \lambda$ that $1_{\text{sort}(\gamma)}\in \Sch^{\rhd\lambda}(m)$. Then Claim (1) follows by noting that the two-sided ideal generated by $1_\gamma$ is the same as $1_{\text{sort}(\gamma)}$ since any (thick) permutation diagram is invertible by \eqref{braid}. 

Next, we prove Claim (2). By assumption, $(A,(\emptyset))$ factors through a morphism 
\[ 
1_*~ 
\begin{tikzpicture}[baseline = 10pt, scale=0.5, color=\clr]
\draw[-,line width =1pt,color=\cred] (-4.5,.2) to (-4.5,2);
\draw (-4.5,-.2) node{$\scriptstyle \red{u_q}$};    
\draw[-] (-3.5,.2) to (-3.5,2);
\draw (-3.5,-.2) node{$\scriptstyle {\lambda^{(q)}_1}$};
\draw(-2.6,.8) node{$\ldots$};
\draw[-] (-2,.2) to (-2,2);
\draw (-2,-.2) node{$\scriptstyle \lambda^{(q)}_{j-1}$};
\draw[-] (0,.2) to [out=135,in=down] (-5,2);
\draw[-,line width=1.5pt](0,.2) to (0,0);
\draw[-,](0,.2) to (1,2);
%\draw (-.5,1) \bdot;
%\draw (-1,1) node{$\scriptstyle {\omega_r}$};
\draw (-0.5,.3) node{$\scriptstyle r$}; 
\end{tikzpicture}
~1_* 
\in \Hom_{\Sch_{\bfu}}(\lambda,\nu),
\]
where $\nu\in \Lambda_{\text{st}}^\ell(m)$   with $r=a_{(p,i), (q,j)}$.
Then the above morphism factors through $1_\gamma$ as in the proof of Claim~ (1). Hence Claim (2) follows. 

By Claims (1)--(2) (and induction on degrees when applying Claim (1)), we may and shall assume that 
$ (A, P)$  satisfies the following conditions
\begin{itemize}
    \item $P=(\emptyset)$,
    \item $a_{(p,i),(q,j)}=0$ for  all $q>p$.
\end{itemize}
Then we decompose 
$ (A, (\emptyset))$ as the composition of two morphisms as follows. 
Let $b_i=\sum\limits_{j, 1\le q\le \ell-1} a_{(\ell,i),(q,j)}$, for $1\le i\le l(\mu^{(\ell)})$.
First, we define $\tilde \mu\in \Lambda_{\text{st}}^{\ell}(m)$ such that $\tilde\mu^{(p)}=\mu^{(p)}$, for $1\le p\le \ell-2$, and 
\begin{itemize}
    \item $\tilde \mu^{(\ell-1)}=(\mu^{(\ell-1)}, b_1,\ldots, b_{l(\mu^{(\ell)})})$,
    \item $\tilde \mu^{(\ell)}=\mu^{(\ell)}-(b_1,b_2,\ldots, b_{l(\mu^{(\ell)})})$. 
\end{itemize}
Let $c= |\tilde \mu^{(\ell)}|$,
$\hat\mu=(\tilde \mu^{(1)}, \ldots, \tilde \mu^{(\ell-1)}) $ and $\hat \lambda=(\lambda^{(1)}, \ldots, \lambda^{(\ell-1)})$. 
Then we have $ \hat \mu\in \Lambda_{\text{st}}^{\ell-1}(m-c)$ and $\hat \lambda\in \Par^{\ell-1}(m-c)$.
We define
 a matrix $\hat A\in \Mat_{\bar{\hat \mu}, \bar{\hat\lambda}}$ as a submatrix of $A$ by deleting the last $l(\lambda^{(\ell)})$ columns. Moreover, we omit the zero rows corresponding to $b_i=0$ for $1\le i\le l(\mu^{(\ell)})$. 
Then, we have 
\[
(A, (\emptyset))= D\circ \Big( 
(\hat A, \emptyset) 
\begin{tikzpicture}[baseline = 10pt, scale=0.5, color=\clr]
\draw[-,line width =1pt,color=\cred] (0,.2) to (0,2);
\draw(0,-.1) node{$\scriptstyle \red{u_{\ell}}$};
\end{tikzpicture} [A_{\ell,\ell}]\Big),
\]
where $A_{\ell,\ell}$ is the $(\ell,\ell)$th block of $A$ and  
\[D= 1_* 
\begin{tikzpicture}[baseline = 10pt, scale=0.5, color=\clr]
\draw[-,line width =1pt,color=\cred] (0,.2) to (0,2);
\draw(0,-.2) node{$\scriptstyle \red{u_{\ell}}$};
\draw[-,line width=1.2pt](1,.2) to (1,2);
\draw[-,line width=1pt](1,2) to (-4,.2);
\draw(-4,-.2)node{$\scriptstyle b_1$};
\draw[-,line width=1.2pt](2,.2) to (2,2);
\draw[-,line width=1pt](2,2) to (-3,.2);
\draw(-3,-.2)node{$\scriptstyle b_2$};
\draw(3,.8) node{$\ldots$};
\draw(-2,.2) node{$\ldots$};
\draw[-,line width=1.2pt](4,.2) to (4,2);
\draw[-,line width=1pt](4,2) to (-1,.2);
%\draw(-1,0)node{$\scriptstyle b_{l(\mu^{(\ell)})}$};
\end{tikzpicture}\in \Hom_{\Sch_{\bfu}}(\tilde\mu,\mu).
 \]
By inductive assumption on $\ell-1$ and the result for $\ell=1$, we have that 
\[
\Big( (\hat A, \emptyset) 
\begin{tikzpicture}[baseline = 10pt, scale=0.5, color=\clr]
\draw[-,line width =1pt,color=\cred] (0,.2) to (0,2);
\draw(0,-.1) node{$\scriptstyle \red{u_{\ell}}$};
\end{tikzpicture} [A_{\ell,\ell}] 
\Big)
+ \Sch^{\rhd\lambda}(m) \in \text{Span}_{\lambda, \tilde \mu}. \]
Now the result follows by noting that 
$D\circ [\mathbf T']$ is equal to some $[\mathbf T]\in \SST(\lambda,\mu)$, for any $\mathbf T'\in \SST(\lambda,\tilde \mu)$, where $\mathbf T$ is obtained from $\mathbf T'$ by substituting 
$(l(\mu^{(\ell-1)})+i)_{\ell-1}$ with $i_\ell$, for $1\le i\le l(\mu^{(\ell)})$.
\end{proof}

Recall each $T\in\SST (\la, \nu)$ gives rise to an SST diagram in $\Hom_{\Sch_{\bfu}}(\la,\nu)$; see \S\ref{subsec:SST}. Recall the anti-involution  $\div$ of $\Sch_{\bfu}$ induced by the symmetry $\div$ in \eqref{anti-autocyc}.

\begin{theorem} [Double SST basis]
 \label{thm:SSTbasis}
Let $m\ge 1$. 
\begin{enumerate}
\item 
For $\mu,\nu\in \Lambda_{\text{st}}^\ell(m)$, $\Hom_{\Sch_{\bfu}}(\mu,\nu)$ has a basis given  by 
\begin{align}  \label{doubleSST Hom}
\bigcup_{\lambda\in \Par^\ell(m)}
\big\{ [\mathbf T]\circ [\mathbf S]^\div \mid \mathbf T\in \SST(\lambda,\nu), \mathbf S\in \SST(\lambda,\mu) \big\}. 
\end{align}
\item 
The algebra $\Sch_\bfu(m)$ in \eqref{Schurm} has a basis given  by 
\begin{align}  \label{doubleSST}
\bigcup_{\overset{\lambda\in \Par^\ell(m)}{\mu,\nu\in \Lambda_{\text{st}}^\ell(m)} }
\big\{ [\mathbf T]\circ [\mathbf S]^\div \mid \mathbf T\in \SST(\lambda,\nu), \mathbf S\in \SST(\lambda,\mu) \big\}. 
\end{align}
\item 
We have an algebra isomorphism $\mathcal G: \Sch_\bfu(m)\cong \Sc_{m,\bfu}$, which matches the bases \eqref{doubleSST} and \eqref{basisofdjm}.    
\end{enumerate}
\end{theorem}

\begin{proof}
Parts (1) and (2) are clearly equivalent by the definition of \eqref{Schurm}, and we prove (2). Applying $\div$ to the equation in Proposition \ref{pro:spancycdjm}  results a similar equation for $\Hom_{\Sch_{\bfu}}(\mu, \lambda)+\Sch^{\rhd\lambda}(m)$. Combining  the equation in Proposition \ref{pro:spancycdjm} (for $\mu=\nu$) and the \
 equation for $\Hom_{\Sch_{\bfu}}(\mu, \lambda)+\Sch^{\rhd\lambda}(m)$  we obtain that $\Sch_\bfu(m)1_\lambda \Sch_\bfu(m) +\Sch^{\rhd\lambda}(m) $ is spanned by 
 \[ 
 \big\{[\mathbf T]\circ [\mathbf S]^\div  +\Sch^{\rhd\lambda}(m) \mid \mathbf T\in \SST(\lambda,\nu), \mathbf S\in \SST(\lambda,\mu), \mu,\nu\in \Lambda_{\text{st}}^\ell(m) \big\}.
 \]
 Therefore, it follows by induction on $\la$ that $\Sch_\bfu(m)$ is spanned by the elements in \eqref{doubleSST}. By Proposition \ref{pro:mapsdjm}, this spanning set \eqref{doubleSST} is mapped by $\mathcal G: \Sch_\bfu(m)\rightarrow \Sc_{m,\bfu}$ to the cellular basis \eqref{basisofdjm} for $\Sc_{m,\bfu}$ and hence must be linearly independent. This proves that \eqref{doubleSST} forms a basis for $\Sch_\bfu(m)$.
 
(3) Follows since the homomorphism $\mathcal G: \Sch_\bfu(m)\rightarrow \Sc_{m,\bfu}$ matches the 2 bases. 
\end{proof}

\begin{rem}
\label{rem:triangularbasis}
Theorem \ref{thm:SSTbasis} shows that $\Sch_{\bfu}$ is  a weakly triangular category (see \cite{GRS}), or equivalently, the path algebra of $\Sch_{\bfu}$,
\[
A_\bfu
:= \bigoplus_{\lambda,\mu\in \Lambda_{\text{st}}^\ell} \Hom_{\Sch_\bfu}(\lambda,\mu)
= \bigoplus_{\lambda,\mu\in \Lambda_{\text{st}}^\ell} 1_\mu A1_\lambda,
\]
admits a triangular basis (see \cite{BS}). In fact, for any $\mu\in \Lambda_{\text{st}}^\ell(m)$ and $ \lambda\in \Par^\ell(m')$, $\gamma\in \Par^\ell(m'')$, we define  
\begin{enumerate}
    \item $X_{\mu,\lambda}=\emptyset=Y(\lambda,\mu)$ unless $m=m'$,
    \item $H(\lambda,\gamma)=\emptyset$ unless $ \lambda=\gamma$. In the later case, $H(\lambda,\lambda)=\{1_\lambda\}$,
   \item for $m=m'$,  $X_{\mu,\lambda}=\{ [\mathbf T]\mid \mathbf T\in \SST(\lambda,\mu)\}$  and 
   $Y_{\lambda,\mu}=\{ [\mathbf T]^\div\mid \mathbf T\in \SST(\lambda,\mu)\}$.
\end{enumerate}
Since $\SST(\lambda,\mu)=\emptyset $ unless $\mu\unlhd\lambda$, by Theorem \ref{thm:SSTbasis}(1), we see that $(A_{\bfu}, \Lambda_{\text{st}}^\ell, \Par^\ell, \lhd)$ has a triangular basis
\[
\{xhy\mid x\in X(\mu,\lambda), h\in H(\lambda,\gamma), y\in Y(\gamma,\nu), \mu,\nu\in \Lambda_{\text{st}}^\ell, \lambda,\gamma\in \Par^\ell\}.
\]
Moreover, by \cite[Corollary 5.39]{BS} and \cite[Theorem 3.10]{GRS} the category $A_{\bfu}$-lfdmod of locally finite dimensional representations of $A_\bfu$ is a (upper finite) highest weight category since $\bar 1_\lambda A_{\unlhd \lambda} \bar 1_\lambda$ has a basis $\{\bar 1_\lambda\}$, where
$A_{\unlhd \lambda}:= A_{\bfu}/(1_\mu\mid \mu \ntrianglelefteq \lambda)$. The standard module $\Delta(\lambda)$, for $\lambda\in\Par^\ell$, is obtained by applying the standardzation functor $A_{\unlhd \lambda} \bar 1_{\lambda} \otimes_{\bar 1_\lambda A_{\unlhd\lambda }\bar 1_\lambda} ? $ to the one-dimensional $\bar 1_\lambda A_{\unlhd\lambda }\bar 1_\lambda$-module $\kk_\lambda$ and it admits a simple head $L(\lambda)$.
\end{rem}

Via the isomorphism in Theorem \ref{thm:SSTbasis}(3),  we have the following simple corollary of \cite[Corollary 5.39]{BS} and \cite[Theorem 3.10]{GRS}; this result was known in \cite[Corollary~ 6.18]{DJM98} by a very different approach.

\begin{corollary}
The cyclotomic Schur algebra  $\Sc_{m,\bfu}$ is quasi-hereditary. The set $\{L(\lambda) \mid \la \in \Par^\ell(m)\}$ forms a complete set of non-isomorphic simple $\Sc_{m,\bfu}$-modules.  
\end{corollary}

\subsection{Elementary diagram basis and higher level RSK}

For $\mu,\nu\in \Lambda_{\text{st}}^\ell(m)$, we denote 
\[
\SST_{\nu,\mu}^2 := \bigcup_{\lambda\in \Par^\ell(m)}\{(\mathbf S, \mathbf T) \mid \mathbf S \in \SST(\lambda,\nu), \mathbf T\in \SST(\lambda,\mu) \}.
\]
We present the next two theorems together. 

 \begin{theorem}
 \label{thm:RSK}
 Suppose $\mu,\nu\in \Lambda_{\text{st}}^\ell(m)$. There exists a bijective map  
   \begin{align*}
    \varphi_\ell : \PMat^\flat_{\nu,\mu}& \longrightarrow \SST^{\,2}_{\nu,\mu}. 
    %,\qquad (A,P) \mapsto (\mathbf S, \mathbf T).
\end{align*}
\end{theorem}

 \begin{theorem}
 \label{thm:basis2CycSchur}
    Suppose $\mu,\nu\in \Lambda_{\text{st}}^\ell(m)$. Then $\PMat_{\nu,\mu}^\flat$ forms a basis for $\Hom_{\Sch_{\bfu}}(\mu,\nu)$.
\end{theorem}

Our proofs of Theorems \ref{thm:RSK} and \ref{thm:basis2CycSchur} below are interdependent: Theorem~\ref{thm:basis2CycSchur} follows readily from two main results above (i.e., Proposition~ \ref{pro:spanofcycshcurcate} and Theorem \ref{thm:SSTbasis}) together with a much weak fact from Theorem~ \ref{thm:RSK} that $|\PMat^\flat_{\nu,\mu}| \leq |\SST^2_{\nu,\mu}|$. On the other hand, for Theorem~ \ref{thm:RSK} we take the level one map $\varphi_1$ to be the celebrated Robinson-Shensted-Knuth (RSK) correspondence. Both the construction of the map $\varphi_\ell$ and the proof that $\varphi_\ell$ is bijective for general $\ell$ in the next subsection rely essentially on the category $\Sch_{\bfu}$. It will be of considerable interest to construct a version of the bijection $\varphi_\ell$ combinatorially.
 
Denote by
\[
\PMat^\flat:= \bigcup_{m\ge 0} \bigcup_{\mu,\nu\in \Lambda_{\text{st}}^\ell(m)}  \PMat^\flat_{\nu,\mu}
\]
the set of bounded-partition-enhanced $\ell\times \ell$-block $\N$-matrices (cf. \eqref{PM0}), and denote by 
\[
\SST^2:= \bigcup_{m\ge 0} \bigcup_{\mu,\nu\in \Lambda_{\text{st}}^\ell(m)} \SST^2_{\nu,\mu}
\]
the set of pairs of semi-standard tableaux of the same $\ell$-multipartition shape.

\begin{conjecture} [Higher level RSK correspondence]
 \label{conjRSK}
    There exists a combinatorial construction of a (level $\ell$) one-to-one correspondence $\psi_\ell:\PMat^\flat \longrightarrow \SST^{\,2}$ (or more explicitly, bijections $\psi_\ell:\PMat^\flat_{\nu,\mu} \longrightarrow \SST^{\,2}_{\nu,\mu}$, for all $\mu,\nu\in \Lambda_{\text{st}}^\ell(m)$), so that the level one map $\psi_1$ is the RSK correspondence and the level $\ell$ map $\psi_\ell$ extends the level $(\ell-1)$ map $\psi_{\ell-1}$. 
\end{conjecture}

\begin{rem}
    The above conjecture has been proved by Eriksson \cite{Er26}, and Theorem~ \ref{thm:RSK} follows from this. We have kept our original proof of Theorem~ \ref{thm:RSK} below to make the paper self-contained (but skipped a long subsection of examples explaining our proof). 
\end{rem}
\subsection{Proof of Theorems \ref{thm:RSK} and \ref{thm:basis2CycSchur} }
\label{subsec:varphi}

We construct a map $\varphi_\ell : \PMat^\flat_{\nu,\mu} \rightarrow \SST^2_{\nu,\mu}$ inductively on $\ell$.
%and then show that $\varphi_\ell$ is injective. 

Suppose first $\ell=1$. We define the bijective map
\begin{equation}
    \label{bijectionlevelone}
 \varphi_1: \PMat_{\nu,\mu}^\flat=\Mat_{\nu,\mu} \longrightarrow \SST_{\nu,\mu}^2   
\end{equation}
for $\mu,\nu \in \Lambda_{\text{st}}(m)$ to be the well-known RSK correspondence; any bijection will do too. 

Suppose $\ell>1$ and we have defined the map $\varphi_{\ell-1}$. To define the map $\varphi_\ell$, we need some more notation. 

Given any $\mu,\nu \in \Lambda^\ell_{\text{st}}(m)$, let $h_p=l(\nu^{(p)})$ and $t_p= l(\mu^{(p)})$, for  $1\le p\le \ell$.
Recall $\big(A=(A_{pq}), P=(P_{pq}) \big)\in \Par\Mat_{\nu,\mu}$ consists of a pair of $\ell\times \ell$ block matrices under the bijective map \eqref{bijPM}, where 
\begin{align}  \label{APeta}
A_{pq}=(a_{(p,i), (q,j)})_{1\le i\le h_p, 1\le j\le t_q}, \quad P_{pq}=(\eta_{(p,i), (q,j)})_{1\le i\le h_p, 1\le j\le t_q}.
\end{align}
We further introduce a shorthand notation 
\begin{align}  \label{Peta}
P_{\ell\ell}=(\eta^{ij})_{1\le i\le h_\ell, 1\le j\le t_\ell}, 
\end{align}
where each $\eta^{ij} =(\eta^{ij}_1, \eta^{ij}_2, \eta^{ij}_3, \ldots)$ is a partition. We denote by $\overline{\eta}^{ij} =(\eta^{ij}_2, \eta^{ij}_3, \ldots)$ the remainder of $\eta_{ij}$ after the removal of its largest part. 
We define a new pair of matrices 
\begin{align}  \label{eq:Aprime}
A'_{\ell\ell}=\left( \eta^{ij}_1\right)_{1\le i\le h_\ell, 1\le j\le t_\ell }, 
\quad \text{and } 
P'_{\ell\ell} =\left( \bar\eta^{ij}\right)_{1\le i\le h_\ell, 1\le j\le t_\ell },
\end{align}
of the same size as for $A_{\ell\ell}$ and $P_{\ell\ell}.$

Replacing the $(\ell,\ell)$th block $(A_{\ell\ell},P_{\ell\ell})$ in the pair $(A,P)$ by a new $(\ell,\ell)$th block $(A'_{\ell\ell},P'_{\ell\ell})$ gives us a new pair $(\widetilde A, \widetilde P)$. Combining the $(\ell-1)$st and $\ell$th row (and respectively, column) blocks of $(\widetilde A, \widetilde P)$ as a single $(\ell-1)$st row (and respectively, column) block, we shall view $(\widetilde A =(\widetilde A_{pq}), \widetilde P =(\widetilde P_{pq}) )$ as a pair of $(\ell-1)\times (\ell-1)$-block matrices from now on. 

In this way, we have $(\widetilde A, \widetilde P) \in \PMat_{\tilde{\nu},\tilde{\mu}}$, for suitable $\tilde\mu,\tilde \nu \in \Lambda^{\ell-1}_{\text{st}}$, by noting that $l(\overline{\eta}^{ij}) = l(\eta^{ij})-1 \leq \ell -2$ if $\eta^{ij} \neq \emptyset$. 
More explicitly, $\tilde\mu,\tilde \nu$ are obtained from $\mu =(\mu^{(1)}, \ldots, \mu^{(\ell)}), \nu =(\nu^{(1)}, \ldots, \nu^{(\ell)})$ as follows:
\begin{itemize}
    %\item $b_i=\sum_{q\le \ell-1, t} a_{(\ell,i),(q,t)} +\sum_{1\le j\le h_\ell} \eta^{ij}_1$, for $1\le i\le h_\ell$,
    %\item $c_j=\sum_{p\le \ell-1, t} a_{(p,t),(\ell,j)} +\sum_{1\le i\le t_\ell} \eta^{ij}_1$, for $1\le j\le t_\ell$,
    \item 
    $\tilde\nu =(\nu^{(1)}, \ldots, \nu^{(\ell-2)}, \tilde\nu^{(\ell-1)}) \in \Lambda^{\ell-1}_{\text{st}}$
    with
     \[
     \tilde \nu^{(\ell-1)}=( \nu^{(\ell-1)}, b_1, \ldots, b_{h_\ell})\in \Lambda_{\text{st}}, 
     \text{ where }
     b_i=\sum_{q\le \ell-1, t} a_{(\ell,i),(q,t)} +\sum_{1\le j\le h_\ell} \eta^{ij}_1, \forall i; 
     \] 
 \item $\tilde\mu =(\mu^{(1)}, \ldots, \mu^{(\ell-2)}, \tilde\mu^{(\ell-1)})\in \Lambda^{\ell-1}_{\text{st}}$
  with
   \[
   \tilde \mu^{(\ell-1)}=( \mu^{(\ell-1)}, c_1, \ldots, c_{t_\ell})\in \Lambda_{\text{st}}, 
   \text{ where }
   c_j=\sum_{p\le \ell-1, t} a_{(p,t),(\ell,j)} +\sum_{1\le i\le t_\ell} \eta^{ij}_1, \forall j. 
   \]
\end{itemize}
It is understood that the $b_k$'s which are equal to $0$ are removed to view $\tilde \nu^{(\ell-1)}$ as a strict composition, and similar remarks apply in other similar situations.

Define
\[
\hat \nu=\nu^{(\ell)}-(b_1, b_2,\ldots, b_{h_\ell}), 
 \text{ and }   \hat \mu = \mu^{(\ell)}-(c_1, c_2,\ldots, c_{t_\ell}).
 \]
(Here we may view  $\hat \mu $ and $\hat \nu$ as strict compositions by omitting zero parts.) Denote by $|B|$ the sum of all the entries of a matrix $B$. The following is clear from the definitions above. 

 \begin{lemma}
 \label{lem:AA}
     We have 
\begin{enumerate}
\item $d:=|\hat \nu|=|\hat \mu|= |A_{\ell\ell}| - |A_{\ell\ell}'|$,
%\sum_{1\le i\le h_\ell, 1\le j\le t_\ell } a_{ij}-\sum_{1\le i\le h_\ell, 1\le j\le t_\ell } \eta^{ij}_1$, where
\item  $\hat{A} :=A_{\ell\ell} -A'_{\ell\ell} \in \Mat_{\hat \nu, \hat \mu}$,
\item $|\tilde \nu|=|\tilde \mu|=m-d$,
\item If $(A,P)\in \Par\Mat_{\nu,\mu}^{\,\flat}$, then 
$(\widetilde A,\widetilde P)\in \Par\Mat_{\tilde\nu,\tilde\mu}^{\,\flat}$,
\item the assignment $(A,P) \mapsto \big((\widetilde A,\widetilde P), \hat{A} \big)$ is injective.
\end{enumerate}
\end{lemma}

By the inductive assumption on the maps $\varphi_{\ell-1}$ and $\varphi_1$, we have
\begin{equation}
\label{mapell-1andone}
    \varphi_{\ell-1}(\tilde A,\tilde P) = (\tilde{\mathbf S}, \tilde{\mathbf T}) \quad \text{ and } \quad  \varphi_1 (\hat{A}) =(\hat{\mathbf S}, \hat {\mathbf T}),
\end{equation}
where 
$(\tilde {\mathbf S},\tilde {\mathbf T})\in \SST(\tilde\lambda, \tilde \nu)\times \SST(\tilde\lambda,\tilde\mu)$ and 
$(\hat {\mathbf S}, \hat{\mathbf T})\in \SST(\hat\lambda, \hat\nu )\times \SST(\hat \lambda, \hat\mu )$, for some 
$\tilde \lambda\in \Par^{\ell-1}( m-d )$ and $\hat \lambda\in \Par(d)$. 

Let $\lambda\in \Par^\ell(m)$ be such that 
$\lambda^{(i)}= \tilde\lambda^{(i)}$ for $1\le i\le \ell-1$ and 
$\lambda^{(\ell)}= \hat \lambda$. Define 
$(\mathbf S,\mathbf T)\in \SST(\lambda,\nu)\times \SST(\lambda,\mu)$ whose  $\ell$ components are specified as follows:
\begin{itemize}
    \item
    $\mathbf S^{(\ell)}$ and $\mathbf T^{(\ell)}$ are copies of 
    $\hat{\mathbf S} $ and $\hat {\mathbf T}$ (by replacing $i$ with $i_\ell$ for all $i$), respectively; 
    \item 
    $\mathbf S^{(p)}$ is obtained from 
    $\tilde{\mathbf S}^{(p)}$ by substituting 
    $(h_{\ell-1}+i)_{\ell-1}$ with $i_{\ell} $ for $1\le i\le h_\ell $, $1\le p\le \ell-1$;
   \item  
   similarly,   $\mathbf T^{(q)}$ is obtained from 
    $\tilde{\mathbf T}^{(q)}$ by substituting 
    $(t_{\ell-1}+j)_{\ell-1}$ with $j_\ell$ for
    $1\le j\le t_\ell $, $1\le q\le \ell-1$. 
\end{itemize}
This allows us to define the map $\varphi_\ell : \Par\Mat^\flat_{\nu,\mu} \rightarrow \SST^2_{\nu,\mu}, 
(A,P)  \mapsto (\mathbf S, \mathbf T),$ as stated in Theorem \ref{thm:RSK}. 

The map $\varphi_\ell$ as defined satisfies the following key property.

\begin{lemma}  \label{lem:inj}
The map $\varphi_\ell: \Par\Mat^{\,\flat}_{\nu,\mu} \rightarrow \SST^{\,2}_{\nu,\mu}$ defined above is injective.
\end{lemma}

\begin{proof}
    Suppose that $(A,P), (B,Q)\in \Par\Mat_{\nu,\mu}^\flat$.
     Let $( \tilde{\mathbf S},\tilde {\mathbf T})$ and $(\hat {\mathbf S}, \hat {\mathbf T})$ be given in \eqref{mapell-1andone} for $(A,P)$.
     Similarly, associated to $(B,Q)$, we have $(\widetilde B, \widetilde Q), B_{\ell\ell}, B'_{\ell\ell}, \hat{B}$, and we denote
     \[  
     \varphi_{\ell-1}(\tilde B,\tilde Q) = (\tilde{\mathbf S}_1, \tilde{\mathbf T}_1) 
     \quad \text{ and } \quad
     \varphi_1 (\hat{B}) =(\hat{\mathbf S}_1, \hat {\mathbf T}_1).
     \]
     
  Suppose   $ \varphi_\ell (A,P)=(\mathbf S, \mathbf T)=\varphi_\ell(B,Q)=(\mathbf S_1, \mathbf T_1)$.
By the construction of $(\mathbf S,\mathbf T)$ and $(\mathbf S_1, \mathbf T_1) $  above, we have 
\[ 
(\tilde{\mathbf S},\tilde{\mathbf T})= (\tilde{\mathbf S}_1,\tilde{\mathbf T}_1) 
\quad \text{ and } \quad 
(\hat{\mathbf S},\hat {\mathbf T})= (\hat{\mathbf S}_1, \hat {\mathbf T}_1).
\]
Recall $\varphi_1$ is injective. By the inductive assumption that $\varphi_{\ell-1}$ is injective, we have 
  \[ 
  (\tilde A,\tilde P)= (\tilde B,\tilde Q) 
  \quad \text{and } \quad
  \hat{A} = \hat{B}.
  \]
This implies by Lemma~\ref{lem:AA}(6) that $(A,P)= (B,Q)$ and hence $\varphi_\ell$ is injective.
\end{proof}

Now we are ready to prove Theorems \ref{thm:RSK} and \ref{thm:basis2CycSchur}.

\begin{proof}  [Proof of Theorems \ref{thm:RSK} and \ref{thm:basis2CycSchur}]

By Theorem \ref{thm:SSTbasis}, the $\kk$-module $\Hom_{\Sch_{\bfu}}(\mu,\nu)$ is free with basis $\SST^2_{\nu,\mu}$, and by Proposition~ \ref{pro:spanofcycshcurcate}, $\Hom_{\Sch_{\bfu}}(\mu,\nu)$ admits a second spanning set $\PMat^\flat_{\nu,\mu}$  of cardinality $|\PMat^\flat_{\nu,\mu}| \leq |\SST^2_{\nu,\mu}|$ by Lemma \ref{lem:inj}. It follows by a standard argument that $|\PMat^\flat_{\nu,\mu}| = |\SST^2_{\nu,\mu}|$ and  $\PMat_{\nu,\mu}^\flat$ is a basis for $\Hom_{\Sch_{\bfu}}(\mu,\nu)$. Therefore, the map $\varphi_\ell$ is bijective by Lemma~\ref{lem:inj} again. 
\end{proof}

\subsection{Proof of the basis theorem for $\W_{\bfu}$}
\label{subsec:proofbasis}

Now we give a proof of \cite[Theorem~ 4.3]{SWweb},  
%\cite[Theorem~ \ref{basis-cyc-web}]{SWweb}, 
the basis theorem for the cyclotomic web category $\W_\bfu$, for $\bfu \in \kk^\ell$. 

\begin{proof} 
 There is an obvious functor $\mathcal F: \AW\rightarrow \Sch_{\bfu}$ of $\kk$-linear categories, which sends an object $\lambda\in \Lambda_{\text{st}}$ to $(\emptyset^\ell,\lambda)\in \Lambda_{\text{st}}^{1+\ell}$ and sends any morphism $g$ to 
 $\begin{tikzpicture}[baseline = 10pt, scale=0.5, color=\clr]
\draw[-,line width =1pt,color=\cred] (-4.5,.2) to (-4.5,2);
\draw (-4.5,-.2) node{$\scriptstyle \red{u_1}$};    
\draw[-,line width =1pt,color=\cred] (-3.5,.2) to (-3.5,2);
\draw (-3.5,-.2) node{$\scriptstyle \red{u_2}$};
\draw(-2.6,.8) node{$\ldots$};
\draw[-,line width =1pt,color=\cred] (-2,.2) to (-2,2);
\draw (-2,-.2) node{$\scriptstyle \red{u_{\ell}}$};
%\draw[-,line width=1pt] (-1,2) to (-1,0.2);
%\draw (-1,1) \bdot;
\draw (-1.2,1) node{$g$};
%\draw (-1,-.2) node{$\scriptstyle {r}$}; 
\end{tikzpicture}
$. 
Thanks to Lemma~ \ref{lem:cycpolyvanish}, $\mathcal F$ factors through $\W_\bfu$ so we have a functor $\mathcal F: \W_\bfu \rightarrow \Sch_{\bfu}$. Since $\mathcal F$ sends the spanning set $\PMat^\ell_{\nu,\mu}$ of $\Hom_{\W_{\bfu}}(\mu,\nu)$ onto the elementary diagram basis  $\PMat_{(\emptyset^{\ell},\nu),(\emptyset^\ell,\mu)}^\flat$ of the corresponding Hom-space for $\Sch_\bfu$, the spanning set $\PMat^\ell_{\nu,\mu}$ must be a basis. 
\end{proof}

The following can be read off from the above proof.

\begin{corollary}
\label{cor:fullsubcategory}
$\W_\bfu$ is a full subcategory of the cyclotomic Schur category $\Sch_\bfu$ with objects $\{(\emptyset^\ell,\mu)\mid \mu\in \Lambda_{\text{st}}\}$.
\end{corollary}

%% file: Section6_zerochar.tex
\section{Affine Schur category over $\C$}
\label{sec:C}

Throughout this section, we set $\kk$ to be any field of characteristic zero, e.g., $\C$. Under such an assumption on $\kk$, we greatly simplify the presentation of  $\ASch$ and its cyclotomic quotients. This also helps us to understand better the origins of more involved general relations in the original presentation of  $\ASch$ (for general $\kk$, e.g., $\kk =\Z$). 

\subsection{The affine web category over $\C$}
 \label{affineweboverC}

We give a reduced version of the affine web category, following \cite[Section 5]{SWweb}.

\begin{definition} 
\label{def-affine-webC}
The affine web category $\AWC$ is the strict $\kk$-linear monoidal category generated by objects $a\in \mathbb Z_{\ge 1}$. The morphisms are generated by 
\begin{align}
\label{merge+split+crossing C}
\begin{tikzpicture}[baseline = -.5mm,scale=.8,color=\clr]
	\draw[-,line width=1pt] (0.28,-.3) to (0.08,0.04);
	\draw[-,line width=1pt] (-0.12,-.3) to (0.08,0.04);
	\draw[-,line width=1.5pt] (0.08,.4) to (0.08,0);
        \node at (-0.22,-.4) {$\scriptstyle a$};
        \node at (0.35,-.4) {$\scriptstyle b$};\node at (0,.55){$\scriptstyle a+b$};\end{tikzpicture} 
&:(a,b) \rightarrow (a+b),&
\begin{tikzpicture}[baseline = -.5mm,scale=.8,color=\clr]
	\draw[-,line width=1.5pt] (0.08,-.3) to (0.08,0.04);
	\draw[-,line width=1pt] (0.28,.4) to (0.08,0);
	\draw[-,line width=1pt] (-0.12,.4) to (0.08,0);
        \node at (-0.22,.5) {$\scriptstyle a$};
        \node at (0.36,.5) {$\scriptstyle b$};
        \node at (0.1,-.45){$\scriptstyle a+b$};
\end{tikzpicture}
&:(a+b)\rightarrow (a,b),&
\begin{tikzpicture}[baseline=-.5mm,scale=.8,color=\clr]
	\draw[-,line width=1pt] (-0.3,-.3) to (.3,.4);
	\draw[-,line width=1pt] (0.3,-.3) to (-.3,.4);
        \node at (0.3,-.42) {$\scriptstyle b$};
        \node at (-0.3,-.42) {$\scriptstyle a$};
         \node at (0.3,.55) {$\scriptstyle a$};
        \node at (-0.3,.55) {$\scriptstyle b$};
\end{tikzpicture}
&:(a,b) \rightarrow (b,a),
\end{align}
%(called the merges, splits and crossings, respectively), 
and 
\begin{equation}
\label{dotgeneratorC}
 \begin{tikzpicture}[baseline = 3pt, scale=0.5, color=\clr]
\draw[-,thick] (0,0) to[out=up, in=down] (0,1.4);
\draw(0,0.6) \bdot;
\node at (0,-.3) {$\scriptstyle 1$};
\end{tikzpicture} \;,
\end{equation}
subject to the following relations \eqref{webassoc C}--\eqref{dotmovecrossingC}, for $a,b,c,d \in \Z_{\ge 1}$ with $d-a=c-b$:
\begin{align}
\label{webassoc C}
\begin{tikzpicture}[baseline = 0,color=\clr]
	\draw[-,thick] (0.35,-.3) to (0.08,0.14);
	\draw[-,thick] (0.1,-.3) to (-0.04,-0.06);
	\draw[-,line width=1pt] (0.085,.14) to (-0.035,-0.06);
	\draw[-,thick] (-0.2,-.3) to (0.07,0.14);
	\draw[-,line width=1.5pt] (0.08,.45) to (0.08,.1);
        \node at (0.45,-.41) {$\scriptstyle c$};
        \node at (0.07,-.4) {$\scriptstyle b$};
        \node at (-0.28,-.41) {$\scriptstyle a$};
\end{tikzpicture}
&=
\begin{tikzpicture}[baseline = 0, color=\clr]
	\draw[-,thick] (0.36,-.3) to (0.09,0.14);
	\draw[-,thick] (0.06,-.3) to (0.2,-.05);
	\draw[-,line width=1pt] (0.07,.14) to (0.19,-.06);
	\draw[-,thick] (-0.19,-.3) to (0.08,0.14);
	\draw[-,line width=1.5pt] (0.08,.45) to (0.08,.1);
        \node at (0.45,-.41) {$\scriptstyle c$};
        \node at (0.07,-.4) {$\scriptstyle b$};
        \node at (-0.28,-.41) {$\scriptstyle a$};
\end{tikzpicture}\:,
\qquad
\begin{tikzpicture}[baseline = -1mm, color=\clr]
	\draw[-,thick] (0.35,.3) to (0.08,-0.14);
	\draw[-,thick] (0.1,.3) to (-0.04,0.06);
	\draw[-,line width=1pt] (0.085,-.14) to (-0.035,0.06);
	\draw[-,thick] (-0.2,.3) to (0.07,-0.14);
	\draw[-,line width=1.5pt] (0.08,-.45) to (0.08,-.1);
        \node at (0.45,.4) {$\scriptstyle c$};
        \node at (0.07,.42) {$\scriptstyle b$};
        \node at (-0.28,.4) {$\scriptstyle a$};
\end{tikzpicture}
=\begin{tikzpicture}[baseline = -1mm, color=\clr]
	\draw[-,thick] (0.36,.3) to (0.09,-0.14);
	\draw[-,thick] (0.06,.3) to (0.2,.05);
	\draw[-,line width=1pt] (0.07,-.14) to (0.19,.06);
	\draw[-,thick] (-0.19,.3) to (0.08,-0.14);
	\draw[-,line width=1.5pt] (0.08,-.45) to (0.08,-.1);
        \node at (0.45,.4) {$\scriptstyle c$};
        \node at (0.07,.42) {$\scriptstyle b$};
        \node at (-0.28,.4) {$\scriptstyle a$};
\end{tikzpicture}\:,
\\
\label{mergesplitC}
\begin{tikzpicture}[baseline = 7.5pt,scale=.8, color=\clr]
	\draw[-,line width=1pt] (0,0) to (.275,.3) to (.275,.7) to (0,1);
	\draw[-,line width=1pt] (.6,0) to (.315,.3) to (.315,.7) to (.6,1);
        \node at (0,1.13) {$\scriptstyle b$};
        \node at (0.63,1.13) {$\scriptstyle d$};
        \node at (0,-.1) {$\scriptstyle a$};
        \node at (0.63,-.1) {$\scriptstyle c$};
\end{tikzpicture}
&=
\sum_{\substack{0 \leq s \leq \min(a,b)\\0 \leq t \leq \min(c,d)\\t-s=d-a}}
\begin{tikzpicture}[baseline = 7.5pt,scale=.8, color=\clr]
	\draw[-,thick] (0.58,0) to (0.58,.2) to (.02,.8) to (.02,1);
	\draw[-,thick] (0.02,0) to (0.02,.2) to (.58,.8) to (.58,1);
	\draw[-,thin] (0,0) to (0,1);
	\draw[-,line width=1pt] (0.61,0) to (0.61,1);
        \node at (0,1.13) {$\scriptstyle b$};
        \node at (0.6,1.13) {$\scriptstyle d$};
        \node at (0,-.1) {$\scriptstyle a$};
        \node at (0.6,-.1) {$\scriptstyle c$};
        \node at (-0.1,.5) {$\scriptstyle s$};
        \node at (0.77,.5) {$\scriptstyle t$};
\end{tikzpicture},
\\
\label{splitmergeC}
\begin{tikzpicture}[baseline = -1mm,scale=.8,color=\clr]
	\draw[-,line width=1.5pt] (0.08,-.8) to (0.08,-.5);
	\draw[-,line width=1.5pt] (0.08,.3) to (0.08,.6);
\draw[-,thick] (0.1,-.51) to [out=45,in=-45] (0.1,.31);
\draw[-,thick] (0.06,-.51) to [out=135,in=-135] (0.06,.31);
        \node at (-.33,-.05) {$\scriptstyle a$};
        \node at (.45,-.05) {$\scriptstyle b$};
\end{tikzpicture}
&= 
\binom{a+b}{a}\:\:
\begin{tikzpicture}[baseline = -1mm,scale=.5,color=\clr]
	\draw[-,line width=1.5pt] (0.08,-.8) to (0.08,.6);
        \node at (.08,-1.2) {$\scriptstyle a+b$};
\end{tikzpicture},
\\
\label{dotmovecrossingC}
\begin{tikzpicture}[baseline = 7.5pt, scale=0.35, color=\clr]
\draw[-,thick] (0,-.2) to  (1,2.2);
\draw[-,thick] (1,-.2) to  (0,2.2);
\draw(0.2,1.6)\bdot;
\node at (0, -.5) {$\scriptstyle 1$};
\node at (1, -.5) {$\scriptstyle 1$};
\end{tikzpicture} 
&=
\begin{tikzpicture}[baseline = 7.5pt, scale=0.35, color=\clr]
\draw[-, thick] (0,-.2) to (1,2.2);
\draw[-,thick] (1,-.2) to(0,2.2);
\draw(.8,0.3)\bdot;
\node at (0, -.5) {$\scriptstyle 1$};
\node at (1, -.5) {$\scriptstyle 1$};
\end{tikzpicture} 
+
\begin{tikzpicture}[baseline = 7.5pt, scale=0.35, color=\clr]
\draw[-, thick] (0,.5) to (0,2.2);
\draw[-, thick] (0,-.2) to (0,.5);
\draw[-, thick]   (1,1.8) to (1,2.2); 
\draw[-, thick] (1,1.8) to (1,-.2); 
\node at (0, -.5) {$\scriptstyle 1$};
\node at (1, -.5) {$\scriptstyle 1$};
\end{tikzpicture},
\qquad
\begin{tikzpicture}[baseline = 7.5pt, scale=0.35, color=\clr]
\draw[-, thick] (0,-.2) to (1,2.2);
\draw[-,thick] (1,-.2) to(0,2.2);
\draw(.2,0.2)\bdot;
\node at (0, -.5) {$\scriptstyle 1$};
\node at (1, -.5) {$\scriptstyle 1$};
\end{tikzpicture}
= 
\begin{tikzpicture}[baseline = 7.5pt, scale=0.35, color=\clr]
\draw[-,thick] (0,-.2) to  (1,2.2);
\draw[-,thick] (1,-.2) to  (0,2.2);
\draw(0.8,1.6)\bdot;
\node at (0, -.5) {$\scriptstyle 1$};
\node at (1, -.5) {$\scriptstyle 1$};
\end{tikzpicture}
+
\begin{tikzpicture}[baseline = 7.5pt, scale=0.35, color=\clr]
\draw[-,thick] (0,-.3) to (0,2.2);
\draw[-,thick] (1,2.2) to (1,-.3); 
\node at (0, -.6) {$\scriptstyle 1$};
\node at (1, -.6) {$\scriptstyle 1$};
\end{tikzpicture}.
\end{align}
\end{definition}

Following \cite[Section 5]{SWweb}, the (redundant generating) morphisms $\wkdota$ in $\AWC$, for $1\le r \le a$, can be defined as:
\begin{equation}
\label{newgendot}
\wkdotr :=\frac{1}{r!}
\begin{tikzpicture}[baseline = 1.5mm, scale=.5, color=\clr]
\draw[-, line width=1.5pt] (0.5,2) to (0.5,2.5);
\draw[-, line width=1.5pt] (0.5,0) to (0.5,-.4);
\draw[-,thin]  (0.5,2) to[out=left,in=up] (-.5,1)
to[out=down,in=left] (0.5,0);
\draw[-,thin]  (0.5,2) to[out=left,in=up] (0,1)
 to[out=down,in=left] (0.5,0);      
\draw[-,thin] (0.5,0)to[out=right,in=down] (1.5,1)
to[out=up,in=right] (0.5,2);
\draw[-,thin] (0.5,0)to[out=right,in=down] (1,1) to[out=up,in=right] (0.5,2);
\node at (0.5,.7){$\scriptstyle \cdots$};
\draw (-0.5,1) \bdot; 
%\node at (-.8,1) {$\scriptstyle \omega_1$}; 
\draw (0,1) \bdot; 
%\node at (0.3,1) {$\scriptstyle \omega_1$};
\node at (0.5,-.6) {$\scriptstyle r$};
\draw (1,1) \bdot;
\draw (1.5,1) \bdot; 
%\node at (1.8,1) {$\scriptstyle \omega_1$};
\node at (-.22,0) {$\scriptstyle 1$};
\node at (1.2,0) {$\scriptstyle 1$};
\node at (.3,0.3) {$\scriptstyle 1$};
\node at (.7,0.3) {$\scriptstyle 1$};
\end{tikzpicture}, 
\qquad 
\wkdota :=
\begin{tikzpicture}[baseline = -1mm,scale=.8,color=\clr]
\draw[-,line width=1.5pt] (0.08,-.3) to (0.08,0);
\draw[-,line width=1.5pt] (0.08,.8) to (0.08,1.1);
\draw[-,thick] (0.1,-.01) to [out=45,in=-45] (0.1,.81);
\draw[-,thick] (0.06,-.01) to [out=135,in=-135] (0.06,.81);
\draw(-.1,0.5) \bdot;
\draw(-.4,0.5)node {$\scriptstyle \omega_r$};
\node at (-.3,.15) {$\scriptstyle r$};
\node at (.6,.15) {$\scriptstyle a-r$};
\end{tikzpicture}  .
\end{equation}

\begin{proposition} [{\cite[Theorem 5.2]{SWweb}}]
 \label{prop:AWisom}
    The monoidal categories $\AWC$ and $\AW$ are isomorphic by identifying generating objects and generating morphisms in the same symbols.
 \end{proposition}

\subsection{Affine Schur category over $\C$}

The presentation for affine Schur category $\ASch$ given in Definition \ref{def-affine-Schur} will be simplified to a reduced presentation below over $\C$.

\begin{definition}   \label{def-affine-SchurC}
The affine Schur category  $\ASchC$ is a strict $\kk$-monoidal category with generating objects $a\in \Z_{\ge 1}$ and $u\in \kk$. We denote the object $a\in \N$ by $\stra$ and denote the object $u \in \kk$ by a red strand labeled by $u$ as $\stru .$
The morphisms are generated by 
\begin{equation} 
\label{generator-affschur C}
  \begin{tikzpicture}[baseline = 10pt, scale=0.4, color=\clr] 
                \draw[-,line width=1pt] (-2,0.2) to (-1.5,1);
	\draw[-,line width=1pt ] (-1,0.2) to (-1.5,1);
	\draw[-,line width=1.5pt] (-1.5,1) to (-1.5,1.8);
                \draw (-1,0) node{$\scriptstyle {b}$};
                \draw (-2,0) node{$\scriptstyle {a}$};
                \draw (-1.5,2.2) node{$\scriptstyle {a+b}$};
                \end{tikzpicture}, 
                \quad 
\begin{tikzpicture}[baseline = 10pt, scale=0.4, color=\clr] 
                \draw[-,line width=1pt] (-2,0.2) to (-1,1.8);
	\draw[-,line width=1pt ] (-1,0.2) to (-2,1.8); 
                \draw (-1,0) node{$\scriptstyle {b}$};
                \draw (-2,0) node{$\scriptstyle {a}$};
                \draw (-1,2.2) node{$\scriptstyle {a}$};
                 \draw (-2,2.2) node{$\scriptstyle {b}$};
                \end{tikzpicture},  
                \quad 
  \begin{tikzpicture}[baseline = 10pt, scale=0.4, color=\clr]
                \draw[-,line width=1pt] (-2,1.8) to (-1.5,1);
	\draw[-,line width=1pt ] (-1,1.8) to (-1.5,1);
	\draw[-,line width=1.5pt] (-1.5,1) to (-1.5,0.2);
                \draw (-1,2.2) node{$\scriptstyle {b}$};
                \draw (-2,2.2) node{$\scriptstyle {a}$};
                \draw (-1.5,0) node{$\scriptstyle {a+b}$};
                \end{tikzpicture}, 
                \quad  \;
 ~
 \begin{tikzpicture}[baseline = 10pt, scale=0.5, color=\clr]
\draw[-,thick] (-1.5,1.9) to (-1.5,0.2);      \draw (-1.5,1) \bdot;
\draw (-1.5,-.1) node{$\scriptstyle {1}$};
\end{tikzpicture}~, 
 \end{equation}
and
\begin{equation}
\label{rightleftcrossinggen C}
\begin{tikzpicture}[baseline = 10pt, scale=.8, color=\clr]
 \draw[-,line width=1.2pt] (-0.3,.3) to (.3,1);
\draw[-,line width=1pt,color=\cred] (0.3,.3) to (-.3,1);
\draw(-.3,0.2) node{$\scriptstyle a$};
\draw (.3, 0.2) node{$\scriptstyle \red{u}$};
\end{tikzpicture} \;\; (\text{traverse-up}), 
                \qquad
\begin{tikzpicture}[baseline = 10pt, scale=.8, color=\clr]
 \draw[-,line width=1pt,color=\cred] (-0.3,.3) to (.3,1);
\draw[-,line width=1.2pt] (0.3,.3) to (-.3,1);
\draw(-.3,0.2) node{$\scriptstyle \red{u}$};
\draw (.3, 0.2) node{$\scriptstyle a$};
\end{tikzpicture} \;\; (\text{traverse-down}),    
\end{equation}
subject to relations \eqref{webassoc C}--\eqref{dotmovecrossingC} for the generators in \eqref{generator-affschur C} and additional relations \eqref{chamvanish1 C}--\eqref{adaptrermovecross C} involving also the traverse-ups and traverse-downs \eqref{rightleftcrossinggen C}:
\begin{align}
\label{chamvanish1 C}
\begin{tikzpicture}[baseline = 10pt, scale=0.35, color=\clr]
\draw[-,line width=1pt,color=\cred](0,-.22)to (0,2.5);
\draw (0,-.5) node{$\scriptstyle \red{u}$};
\draw[-,thick](-.5,-.1) to[out=45, in=down] (.5,1.4);
\draw[-,thick](.5,1.4) to[out=up, in=300] (-.5,2.4);
\draw (-.5,-0.5) node{$\scriptstyle {1}$};
\end{tikzpicture}
&=~
\begin{tikzpicture}[baseline = 10pt, scale=0.35, color=\clr]
\draw[-,line width=1pt,color=\cred](-.5,-.25)to (-.5,2.15);
\draw (-.5,-.6) node{$\scriptstyle \red{u}$};
\draw[-,line width=1pt] (-1.3,2.15) to (-1.3,-0.2);      
\draw (-1.3,1) \bdot;
%\draw (-1,.8) node{$\scriptstyle {g}$};
\draw (-1.3,-.6) node{$\scriptstyle {1}$};
\end{tikzpicture}
- u  
\begin{tikzpicture}[baseline = 10pt, scale=0.35, color=\clr]
\draw[-,line width=1pt,color=\cred](-.5,-.25)to (-.5,2.15);
\draw (-.5,-.6) node{$\scriptstyle \red{u}$};
\draw[-,line width=1pt] (-1.3,2.15) to (-1.3,-0.2);      
\draw (-1.3,-.6) node{$\scriptstyle {1}$};
\end{tikzpicture}
~, \qquad \qquad  
\begin{tikzpicture}[baseline = 10pt, scale=0.35, color=\clr]
\draw[-,line width=1pt,color=\cred](0,-.6)to (0,2.2);
\draw (0,-.9) node{$\scriptstyle \red{u}$};
\draw[-,thick](.5,-.5) to[out=135, in=down] (-.5,1);
\draw[-,thick](-.5,1) to[out=up, in=270] (.5,2.2);
\draw (.5,-0.9) node{$\scriptstyle {1}$};
\end{tikzpicture}
~=~ 
\begin{tikzpicture}[baseline = 10pt, scale=0.35, color=\clr]
\draw[-,line width=1pt,color=\cred](-1.8,-.25)to (-1.8,2.15);
\draw (-1.8,-.6) node{$\scriptstyle \red{u}$};
\draw[-,line width=1pt] (-1,2.15) to (-1,-0.25);          
\draw (-1,1) \bdot;
%\draw (-.5,.8) node{$\scriptstyle {g}$};
\draw (-1,-.6) node{$\scriptstyle {1}$};
\end{tikzpicture}
- u
\begin{tikzpicture}[baseline = 10pt, scale=0.35, color=\clr]
\draw[-,line width=1pt,color=\cred](-1.8,-.25)to (-1.8,2.15);
\draw (-1.8,-.6) node{$\scriptstyle \red{u}$};
\draw[-,line width=1pt] (-1,2.15) to (-1,-0.25);          
\draw (-1,-.6) node{$\scriptstyle {1}$};
\end{tikzpicture},
\\
 \label{adaptrermovecross C}
\begin{tikzpicture}[baseline = 7.5pt, scale=0.35, color=\clr]
\draw[-,line width=1pt] (0,-.35) to (1.5,2);
\draw[-,line width=1pt](0,2) to (1.5,-.35);
\draw[-,line width=1pt,color=\cred](.8,2) to[out=down,in=90] (0.2,1) to [out=down,in=up] (.8,-.4);
\draw (.8,-.8) node{$\scriptstyle \red{u}$};
\node at (0, -.7) {$\scriptstyle 1$};
\node at (1.5, -.7) {$\scriptstyle 1$};
\end{tikzpicture}
&=~ 
\begin{tikzpicture}[baseline = 7.5pt, scale=0.35, color=\clr]
\draw[-,line width=1pt] (0,-.35) to (2.1,2);
\draw[-,line width=1pt](0,2) to   (2,-.35);
\draw[-,line width=1pt,color=\cred](.8,2) to
[out=down,in=up] (1.5,1) to [out=down,in=up] (.8,-.4);
\draw (.8,-.7) node{$\scriptstyle \red{u}$};
\node at (0, -.7) {$\scriptstyle 1$};
\node at (2, -.7) {$\scriptstyle 1$};
\end{tikzpicture}
~+~ 
\begin{tikzpicture}[baseline = 7.5pt, scale=0.35, color=\clr]
\draw[-, line width=1pt] (0,-.2) to (0,1.8);
\draw[-, line width=1pt,color=\cred] (.5,1.8) to (.5,-.2);
\draw (.5,-.6) node{$\scriptstyle \red{u}$};
\draw[-, line width=1pt] (1,1.8) to (1,-.2); 
\node at (0, -.6) {$\scriptstyle 1$};
\node at (1, -.6) {$\scriptstyle 1$};
\end{tikzpicture}, 
\\
\label{adaptermovemerge C}
\begin{tikzpicture}[anchorbase,scale=.7,color=\clr]
	\draw[-,line width=1pt,color=\cred] (0.4,0) to (-0.6,1);
  \draw (-.6,1.2) node{$\scriptstyle \red{u}$};
	\draw[-,thick] (0.08,0) to (0.08,1);
	\draw[-,thick] (0.1,0) to (0.1,.6) to (.5,1);
        \node at (0.6,1.13) {$\scriptstyle a$};
        \node at (0.1,1.16) {$\scriptstyle b$};
        %\node at (-0.65,1.13) {$\scriptstyle a$};
\end{tikzpicture}
& =\!\!
\begin{tikzpicture}[anchorbase,scale=.7,color=\clr]
	\draw[-,line width=1pt,color=\cred] (0.7,0) to (-0.3,1);
  \draw (-.3,1.16) node{$\scriptstyle \red{u}$};
	\draw[-,thick] (0.08,0) to (0.08,1);
	\draw[-,thick] (0.1,0) to (0.1,.2) to (.9,1);
        \node at (0.9,1.13) {$\scriptstyle a$};
        \node at (0.1,1.16) {$\scriptstyle b$};
        %\node at (-0.4,1.13) {$\scriptstyle a$};
\end{tikzpicture},\:
\qquad
\begin{tikzpicture}[anchorbase,scale=.7,color=\clr]
	\draw[-,line width=1pt,color=\cred] (-0.4,0) to (0.6,1);
  \draw (.6,1.16) node{$\scriptstyle \red{u}$};
	\draw[-,thick] (-0.08,0) to (-0.08,1);
	\draw[-,thick] (-0.1,0) to (-0.1,.6) to (-.5,1);
         \node at (-0.1,1.16) {$\scriptstyle b$};
        \node at (-0.6,1.13) {$\scriptstyle a$};
\end{tikzpicture}
\!\!=\!\!
\begin{tikzpicture}[anchorbase,scale=.7,color=\clr]
	\draw[-,line width=1pt,color=\cred] (-0.7,0) to (0.3,1);
 \draw (.3,1.16) node{$\scriptstyle \red{u}$};
	\draw[-,thick] (-0.08,0) to (-0.08,1);
	\draw[-,thick] (-0.1,0) to (-0.1,.2) to (-.9,1);
         \node at (-0.1,1.16) {$\scriptstyle b$};
        \node at (-0.95,1.13) {$\scriptstyle a$};
\end{tikzpicture},
\\
\notag
\begin{tikzpicture}[baseline=-3.3mm,scale=.7,color=\clr]
\draw[-,line width=1pt,color=\cred] (0.4,.2) to (-0.6,-.8);
\draw (-.6,-.91) node{$\scriptstyle \red{u}$};
\draw[-,thick] (0.08,0.2) to (0.08,-.75);
\draw[-,thick] (0.1,0.2) to (0.1,-.4) to (.5,-.8);
\node at (0.6,-.91) {$\scriptstyle c$};
\node at (0.07,-.9) {$\scriptstyle b$};
\end{tikzpicture}
&
=\!\!
\begin{tikzpicture}[baseline=-3.3mm,scale=.7,color=\clr]
\draw[-,line width=1pt,color=\cred] (0.7,0.2) to (-0.3,-.8);
\draw (-.3,-.91) node{$\scriptstyle \red{u}$};
\draw[-,thick] (0.08,0.2) to (0.08,-.75);
\draw[-,thick] (0.1,0.2) to (0.1,0) to (.9,-.8);
\node at (1,-.91) {$\scriptstyle c$};
\node at (0.1,-.9) {$\scriptstyle b$};
\end{tikzpicture}
,\:
\qquad
\begin{tikzpicture}[baseline=-3.3mm,scale=.7,color=\clr]
\draw[-,line width=1pt,color=\cred] (-0.4,0.2) to (0.6,-.8);
\draw (.6,-.91) node{$\scriptstyle \red{u}$};
\draw[-,thick] (-0.08,0.2) to (-0.08,-.75);
\draw[-,thick] (-0.1,0.2) to (-0.1,-.4) to (-.5,-.8);
\node at (-0.1,-.9) {$\scriptstyle b$};
\node at (-0.6,-.91) {$\scriptstyle a$};
\end{tikzpicture}
\!\!=\!\!
\begin{tikzpicture}[baseline=-3.3mm,scale=.7,color=\clr]
\draw[-,line width=1pt,color=\cred] (-0.7,0.2) to (0.3,-.8);
\draw (.3,-.91) node{$\scriptstyle \red{u}$};
\draw[-,thick] (-0.08,0.2) to (-0.08,-.75);
\draw[-,thick] (-0.1,0.2) to (-0.1,0) to (-.9,-.8);
\node at (-0.1,-.9) {$\scriptstyle b$};
\node at (-0.95,-.91) {$\scriptstyle a$};
\end{tikzpicture}.
\end{align}
\end{definition}

\begin{theorem}
\label{thm:ASchurisom}
   The monoidal categories $\ASchC$ and $\ASch$ are isomorphic.
\end{theorem}  

We already know that $\AWC \cong \AW$ by Proposition~\ref{prop:AWisom}, and note that $\ASchC$ has fewer relations than $\ASch$.
Therefore, to show that $\ASchC$ is isomorphic to $\ASch$, it suffices to verify that all relations in $\ASch$ but not in $\AW$, i.e., the relations \eqref{adaptorR}--\eqref{adaptorL} and \eqref{adaptrermovecross}--\eqref{adaptermovemerge}, hold in $\ASchC$. Note that the relation \eqref{adaptermovemerge} is identical to  \eqref{adaptermovemerge C} in $\ASchC$. The verification of relations \eqref{adaptorR}--\eqref{adaptorL} and \eqref{adaptrermovecross} in $\ASchC$ will be carried out in Lemma~\ref{adpterLRa=1} and Lemma~\ref{chraterzero} below.

\begin{lemma}
\label{adpterLRa=1}
The relations \eqref{adaptorR}--\eqref{adaptorL} hold in $\ASchC$, for all $a\ge 1$.
\end{lemma}
(The relations \eqref{adaptorR}--\eqref{adaptorL} can be viewed as generalizations of \eqref{chamvanish1 C}.)

\begin{proof}
The relation \eqref{adaptorR} follows by 
\begin{align*}
r!~ 
\begin{tikzpicture}[baseline = 10pt, scale=0.4, color=\clr]
\draw[-,line width=1pt,color=\cred](0,-1)to (0,1.8);
\draw(0,-1.2) node{$\scriptstyle \red u $};
\draw[-,line width= 1.2pt](-.5,-.9) to[out=45, in=down] (.5,.7);
\draw[-,line width= 1.2pt](.5,.7) to[out=up, in=300] (-.5,1.7);
\draw (-.5,-1.2) node{$\scriptstyle {r}$};
\end{tikzpicture}
& \overset{\eqref{splitmergeC}} =          
\begin{tikzpicture}[baseline = 5pt, scale=.4, color=\clr]
\draw[-,line width=1.2pt] (.5,1) to (.5,2.4);
\draw[-,line width=1pt,color=\cred](0.8,2.4) to [out=down,in=up] (.2,1.5) to[out=down,in=up] (1.6,.7) to (1.6,-1.5);
\draw(1.6,-1.8) node{$\scriptstyle \red u$};
%\draw[-,line width=1pt] (.5,1) to (.5,1.5);
\draw[-, line width=1.2pt] (0.5,-1) to (0.5,-1.6);
\draw[-,thin]  (0.5,1) to[out=left,in=up] (-.5,0)to[out=down,in=left] (0.5,-1);
\draw[-,thin]  (0.5,1) to[out=left,in=up] (0,0) to[out=down,in=left] (0.5,-1);
\draw[-,thin] (0.5,-1)to[out=right,in=down] (1.5,0)to[out=up,in=right] (0.5,1);
\draw[-,thin] (0.5,-1)to[out=right,in=down] (1,0)to[out=up,in=right] (0.5,1);
\node at (0.5,-.4){$\scriptstyle \cdots$};
\node at (-.4,-1) {$\scriptstyle 1$};
\node at (.3,-.7) {$\scriptstyle 1$};
%\node at (.7,-.7) {$\scriptstyle 1$};
\node at (1.2,-1) {$\scriptstyle 1$};
\draw (.5,-1.8) node{$\scriptstyle {r}$};
    \end{tikzpicture}
~\overset{\eqref{adaptermovemerge C}}{\underset{\eqref{webassoc C}}{=}}
    ~
\begin{tikzpicture}[baseline = 5pt, scale=.4, color=\clr]
\draw[-,line width=1.2pt] (.5,2) to (.5,2.8);
\draw[-,line width=1pt,color=\cred](1.6,2.4) to [out=down,in=up] (-1,.5) to[out=down,in=up] (1.6,-1.5) to (1.6,-1.6);
\draw(1.6,-1.8) node{$\scriptstyle \red u$};
\draw[-, line width=1.2pt] (0.5,-1) to (0.5,-1.6);
\draw[-,thin]  (0.5,2) to[out=left,in=up] (-.5,1);
%\draw[-,line width=1pt](-.5,1)to (-.5,.8);
%\draw[-,line width=1pt](-.5,.8) to (-.5,.2);
\draw[-,thin](-.5,1) to (-.5,0);
\draw[-,thin](-.5,0)to[out=down,in=left](0.5,-1);
\draw[-,thin]  (0.5,2) to[out=left,in=up](0,1);
\draw[-,thin]  (0,0)  to[out=down,in=left] (0.5,-1);
\draw[-,thin](0,1) to (0,0);
\draw[-,thin] (0.5,-1)to[out=right,in=down](1.5,0.2);
\draw[-,thin](1.5,.8) to (1.5,0.2);
\draw[-,thin](1.5,.8)  to[out=up,in=right](0.5,2);
\draw[-,thin] (0.5,-1)to[out=right,in=down](1,0.2);
\draw[-,thin](1,.8) to (1,0.2);
\draw[-,thin](1,.8) to[out=up,in=right] (0.5,2);
\node at (0.5,.6){$\scriptstyle \cdots$};
\node at (-.4,-1) {$\scriptstyle 1$};
\node at (.2,0) {$\scriptstyle 1$};
%\node at (.7,0) {$\scriptstyle 1$};
\node at (1,-1.2) {$\scriptstyle 1$};
\draw (.5,-1.8) node{$\scriptstyle {r}$};
\end{tikzpicture} 
    ~\overset{\eqref{dotmoveadaptor}}=~ 
\begin{tikzpicture}[baseline = 5pt, scale=.4, color=\clr]
\draw[-, line width=1.2pt] (0.5,1) to (0.5,1.8);
\draw[-, line width=1.2pt] (0.5,-1) to (0.5,-1.6);
\draw[-,thin]  (0.5,1) to[out=left,in=up] (-.5,0)to[out=down,in=left] (0.5,-1);
\draw[-,thin](0.5,1) to[out=left,in=up] (0,0)to[out=down,in=left] (0.5,-1);
\draw[-,thin] (0.5,-1)to[out=right,in=down] (1.5,0)to[out=up,in=right] (0.5,1);
\draw[-,thin] (0.5,-1)to[out=right,in=down] (1,0)to[out=up,in=right] (0.5,1);
\node at (0.5,-.4){$\scriptstyle \cdots$};
\draw (-0.5,0) \bdot;
\node at (-1,0) {$\scriptstyle g$}; 
\draw (0,0) \bdot; 
\node at (0.35,0) {$\scriptstyle g$};
\draw (1,0) \bdot;
\draw (1.5,0) \bdot;
\node at (1.85,0) {$\scriptstyle g$};
\node at (-.4,-1) {$\scriptstyle 1$};
\node at (.3,-.7) {$\scriptstyle 1$};
%\node at (.7,-.7) {$\scriptstyle 1$};
\node at (1.2,-1) {$\scriptstyle 1$};
\draw (.5,-1.8) node{$\scriptstyle {r}$};
\draw[-,line width=1pt,color=\cred](2.2,-1.6)to (2.2,1.8);
\draw(2.2,-1.8) node{$\scriptstyle \red u$};
\end{tikzpicture}\\
&=
r!\sum_{0\le i\le r}(-1)^i\prod_{0\le j\le i-1}(u+j)\:\:
 \begin{tikzpicture}[baseline = -1mm,scale=.7,color=\clr]
\draw[-,line width=1.2pt] (0.08,-.6) to (0.08,.5);
\node at (.08,-.8) {$\scriptstyle r$};
\draw(0.08,0) \bdot;
\draw(-.5,0)node {$\scriptstyle \omega_{r-i}$};
\draw[-,line width=1pt,color=\cred](.5,-.6)to (.5,.5);
\draw (.5,-.8) node{$\scriptstyle \red{u}$};
\end{tikzpicture} \;  .
\end{align*}
 where the last equality follows from \cite[Lemma~2.11]{SWweb}.
 %\cite[Lemma~ \ref{lem:gru}]{SWweb}.   
The remaining relation \eqref{adaptorL} can be checked similarly.
\end{proof}

\begin{lemma} \label{chraterzero}
The relation \eqref{adaptrermovecross} holds in $\ASchC$, for all $a,b$.
\end{lemma}
(The relation \eqref{adaptrermovecross} can be viewed as a generalization of \eqref{adaptrermovecross C}.)

\begin{proof}
By the symmetry $\div$, we may assume that $b\ge a$ 
    in \eqref{adaptrermovecross}. We prove by induction on $a+b$.  We compute the case for $a\ge2$ below. (The case for $a=1$ can be similarly proved by applying the $\div$ to the computation below.)
     We have 
\begin{align*}
       a~
\begin{tikzpicture}[baseline = 7.5pt, scale=0.35, color=\clr]
\draw[-,line width=1pt,color=\cred](.8,2) to[out=down,in=90] (0.2,1) to [out=down,in=up] (.8,-.4);
\draw(.8,-.6) node {$\scriptstyle \red u$};
\draw[-,line width=1.2pt] (0,-.35) to (1.5,2);
\draw[-,line width=1.2pt](0,2) to  (1.5,-.35);
\node at (0, -.6) {$\scriptstyle a$};
\node at (1.5, -.6) {$\scriptstyle b$};
\end{tikzpicture}
        & 
        \overset{\eqref{splitmerge}}=
\begin{tikzpicture}[baseline = 7.5pt, scale=0.35, color=\clr]
\draw[-,line width=1.2pt] (0.5,-1) to (0.5,0);
\draw[-,line width=1.2pt](0.5,0) to [out=left,in =down] (0,.5);
\draw[-,line width=1pt](0.5,0) to [out=right,in =down]node[below]{$~\scriptstyle 1$} (1,.5);
\draw[-,line width=1.2pt] (-.3,2.2) to  (2.3,-.2);
\draw[-,line width=1pt](1,.5) to [out=up,in =right] (0.5,1);
\draw[-,line width=1pt,color=\cred](0.25,2.1) to[out=down,in=up] (-.3,.5) to [out=down,in=up] (1,-.8);
\draw(1,-1) node {$\scriptstyle \red u$};
\draw[-,line width=1.2pt](.5,1) to [out=180,in =50] (0,.5);
\draw[-,line width=1.2pt](.5,1)to (.5,2.4);
\end{tikzpicture}
   ~\overset{\eqref{sliders}} =
   ~
\begin{tikzpicture}[baseline = 7.5pt, scale=0.35, color=\clr]
\draw[-,line width=1.2pt] (0.5,-1) to (0.5,0);
\draw[-,line width=1.2pt](0.5,0) to [out=left,in =down] (0,2);
\draw[-,line width=1pt](0.5,0) to [out=right,in =down]node[below]{$~\scriptstyle 1$} (1,2);
\draw[-,line width=1.2pt] (-1.2,2.5) to  (1.5,0);
\draw[-,line width=1pt](1,2) to [out=up,in =right] (0.5,2.5);
\draw[-,line width=1pt,color=\cred](-.2,2.3) to[out=down,in=up] (-.6,.7) to [out=down,in=up] (1,-.9);
\draw(1,-1.1) node {$\scriptstyle \red u$};
\draw[-,line width=1.2pt](.5,2.5) to [out=180,in =50] (0,2);
\draw[-,line width=1.2pt](.5,2.5)to (.5,3);
\end{tikzpicture}
  \overset{\eqref{adaptermovemerge}}
  = 
\begin{tikzpicture}[baseline = 7.5pt, scale=0.35, color=\clr]
\draw[-,line width=1.2pt] (0,-1) to (0,-.5);
\draw[-,line width=1.2pt](0,-.5) to [out=left,in =down] (-1,1);
\draw[-,line width=1pt](0,-.5) to [out=right,in =down]node[below]{$~\scriptstyle1$} (.8,2);
\draw[-,line width=1.2pt] (-1,1) to (0,2);
\draw[-,line width=1.2pt] (-1.2,2.5) to  (1.5,0);
\draw[-,line width=1pt](.8,2) to [out=up,in =right] (0.5,2.5);
\draw[-,line width=1pt,color=\cred](-.2,2.3) to[out=down,in=up] (-.8,1.3) to [out=down,in=up] (.8,-.7);
\draw(.8,-1) node {$\scriptstyle \red u$};
\draw[-,line width=1.2pt](.5,2.5) to [out=180,in =50] (0,2);
\draw[-,line width=1.2pt](.5,2.5)to (.5,3);
\end{tikzpicture}
\\
  & = 
\begin{tikzpicture}[baseline = 7.5pt, scale=0.35, color=\clr]
\draw[-,line width=1.2pt] (0,-1) to (0,-.5);
\draw[-,line width=1.2pt](0,-.5) to [out=left,in =down] (-1,1);
\draw[-,line width=1pt](0,-.5) to [out=right,in =down]node[below]{$~\scriptstyle1$} (.8,2);
\draw[-,line width=1.2pt] (-1,1) to (0,2);
\draw[-,line width=1.2pt] (-1.2,2.5) to  (1.5,0);
\draw[-,line width=1pt](.8,2) to [out=up,in =right] (0.5,2.5);
\draw[-,line width=1pt,color=\cred](-.3,2.3) to[out=-45,in=up] (.7,-.7);
\draw(.7,-1) node {$\scriptstyle \red u$};
\draw[-,line width=1.2pt](.5,2.5) to [out=180,in =50] (0,2);
\draw[-,line width=1.2pt](.5,2.5)to (.5,3);
\end{tikzpicture}
+
~\sum_{1\le t\le a-1}t! ~
\begin{tikzpicture}[baseline = 7.5pt, scale=0.35, color=\clr]
\draw[-,line width=1.2pt] (0,-1) to (0,-.5);
\draw[-,line width=1.2pt](0,-.5) to [out=left,in =down] (-1,1);
\draw[-,line width=1pt, color=\cred](-.3,3) to (-.3,0.8);
\draw[-,line width=1pt, color=\cred](-.3,0.8)to [out=down,in=135] (1.3,-.7);
\draw(1.3,-.9) node {$\scriptstyle \red u$};
\draw[-,line width=1pt](0,-.5) to [out=right,in =down]node[below]{$\scriptstyle1$} (.8,0);
\draw[-,line width=1.2pt] (-1,1) to node[left]{$\scriptstyle t$} (-1,3);
\draw[-,line width=1.2pt] (0,2.5) to  (0,1.2);
\draw[-,line width=1.5](-1,2.5) to (-.3,1.5);
\draw[-,line width=1.5] (-.3,1.5) to (0.2,1);
\draw[-,line width=1.5] (-.3,2.2) to (0,2.5);
\draw[-,line width=1.5](-1,1.5) to (-.3,2.2);
\draw[-,line width=1.2pt]  (0.2,1)to(1.7,0);
\draw[-,line width=1pt](.8,0)to (.8,2.5);
\draw[-,line width=1pt](.8,2.5) to [out=up,in =right] (.5,3);
\draw[-,line width=1.2pt](.5,3) to [out=180,in =50] (0,2.5);
\draw[-,line width=1.2pt](.5,3)to (.5,3.5);
\end{tikzpicture} 
\quad \text{ by induction on $a-1$,}\\
%\end{align*}
%\begin{align*}
& =
~\begin{tikzpicture}[baseline = 7.5pt, scale=0.35, color=\clr]
\draw[-,line width=1.2pt] (0,-1) to (0,-.5);
\draw[-,line width=1.2pt](0,-.5) to [out=left,in =down] (-1,1);
\draw[-,line width=1pt](0,-.5) to [out=right,in =down] (.8,2);
\draw[-,line width=1.2pt] (-1,1) to (0,2);
\draw[-,line width=1.2pt] (-1.2,1.5) to  (1.8,-.8);
\draw[-,line width=1pt](.8,2) to [out=up,in =right]node[right]{$~\scriptstyle 1$} (0.5,2.5);
\draw[-,line width=1pt,color=\cred](-.3,2.2) to[out=down,in=90] (1.2,1) to [out=down,in=up] (.8,-.8);
\draw(.8,-1) node {$\scriptstyle \red u$};
\draw[-,line width=1.2pt](.5,2.5) to [out=180,in =50] (0,2);
\draw[-,line width=1.2pt](.5,2.5)to (.5,3);
\end{tikzpicture}
+
~\begin{tikzpicture}[baseline = 7.5pt, scale=0.35, color=\clr]
\draw[-,line width=1.2pt] (0,-1) to (0,-.5);
\draw[-,line width=1.2pt](0,-.5) to [out=left,in =down] (-1,1);
\draw[-,line width=1pt](0,-.5) to [out=right,in =down]node[below]{$\scriptstyle1$} (.4,0.6);
\draw[-,line width=1.2pt] (-1,1) to (0,2);
\draw[-,line width=1.2pt] (-1,2) to (1,0);
\draw[-,line width=1pt, color=\cred](.3, 2.8) to (1,-.7);
\draw(1,-1) node {$\scriptstyle \red u$};
\draw[-,line width=1.2pt]  (1,0)to(1.5,-.5);
\draw[-,line width=1pt](1.3,-.2)to (1.3,2);
\draw[-,line width=1pt](1.3,2) to [out=up,in =right] (.8,2.5);
\draw[-,line width=1.2pt](.8,2.5) to [out=180,in =50] (0,2);
\draw[-,line width=1.2pt](.8,2.5)to (.8,3);
\end{tikzpicture}\\ 
&+~
\sum_{1\le t\le a-1}t! ~
\begin{tikzpicture}[baseline = 7.5pt, scale=0.35, color=\clr]
\draw[-,line width=1.2pt] (0,-1) to (0,-.5);
\draw[-,line width=1.2pt](0,-.5) to [out=left,in =down] (-1,1);
\draw[-,line width=1pt, color=\cred](-.3,3) to (-.3,0.8);
\draw[-,line width=1pt, color=\cred](-.3,0.8)to [out=down,in=135] (1.3,-.6);
\draw(1.3,-.8) node {$\scriptstyle \red u$};
\draw[-,line width=1pt](0,-.5) to [out=right,in =down]node[below]{$\scriptstyle1$} (.8,0);
\draw[-,line width=1.2pt] (-1,1) to node[left]{$\scriptstyle t$} (-1,3);
\draw[-,line width=1.2pt] (0,2.5) to  (0,1.2);
\draw[-,line width=1.5](-1,2.5) to (-.3,1.5);
\draw[-,line width=1.5] (-.3,1.5) to (0.2,1);
\draw[-,line width=1.5] (-.3,2.2) to (0,2.5);
\draw[-,line width=1.5](-1,1.5) to (-.3,2.2);
\draw[-,line width=1.2pt]  (0.2,1)to(1.7,0);
\draw[-,line width=1pt](.8,0)to (.8,2.5);
\draw[-,line width=1pt](.8,2.5) to [out=up,in =right] (.5,3);
\draw[-,line width=1.2pt](.5,3) to [out=180,in =50] (0,2.5);
\draw[-,line width=1.2pt](.5,3)to (.5,3.5);
\end{tikzpicture} \quad \text{by the result for $a=1$.}
\end{align*}
The right-hand side above can be rewritten as follows:
\begin{align*}
    ~\overset{\eqref{splitmerge},\eqref{sliders}}{\underset{\eqref{swallows}}{=}}
    &
    a~
\begin{tikzpicture}[baseline = 7.5pt, scale=0.35,color=\clr]
\draw[-,line width=1pt,color=\cred](.8,2) to[out=down,in=90] (1.5,1) to [out=down,in=up] (.8,-.4);
\draw(.8,-.6) node {$\scriptstyle \red u$};
\draw[-,line width=1.2pt] (0,-.35) to (1.5,2);
\draw[-,line width=1.2pt](0,2) to  (1.5,-.35);
\node at (0, -.6) {$\scriptstyle a$};
\node at (1.5, -.6) {$\scriptstyle b$};
\end{tikzpicture}
        ~+~
\begin{tikzpicture}[baseline = 7.5pt, scale=0.35, color=\clr]
\draw[-, line width=1pt] (0,-.3) to (0,2.2);
\draw[-, line width=1.2pt](0,-.3) to (0,.4);
\draw[-, line width=1.2pt] (1.8,-.2) to (1.8,.5);
\draw[-, line width=1.2pt] (1.8,1.8) to (1.8,2.2);
\draw[-, line width=1.2pt]   (0,.2) to  (1.8,1.8);
\draw[-, line width=1.2pt]   (0,1.6) to(1.8,0);
\draw[-,line width=1pt,color=\cred](.8,2) to[out=down,in=up] (1.5,1) to [out=down,in=up] (.8,-.4);
\draw(.8,-.6) node {$\scriptstyle \red u$};
\draw[-, line width=1.2pt]   (0,1.8) to (0,2.2); 
\draw[-, line width=1pt] (1.8,2) to (1.8,-.2); 
\node at (0, -.5) {$\scriptstyle a$};
\node at (2, -.5) {$\scriptstyle b$};
\node at (2, 2.5) {$\scriptstyle a$};
\node at (0, 2.5) {$\scriptstyle b$};
\node at (2,1.2) {$\scriptstyle 1$};
\end{tikzpicture}
        +
        ~\sum_{1\le t\le a-1}t!
        ~
\begin{tikzpicture}[baseline = 7.5pt, scale=0.35, color=\clr]
\draw[-,line width=1.2pt] (0,-1) to (0,-.5);
\draw[-,line width=1.2pt](0,-.5) to [out=left,in =down] (-1,1);
\draw[-,line width=1pt](0,-.5) to [out=right,in =down]node[below]{$\scriptstyle1$} (.8,0);
\draw[-,line width=1.2pt] (-1,1) to node[left]{$\scriptstyle t$} (-1,3);
\draw[-,line width=1.2pt] (.8,2.5) to node[right] {$\scriptstyle t$}  (1.5,.2);
\draw[-,line width=1pt, color=\cred](-.3,3) to (-.3,0.8);
\draw[-,line width=1pt, color=\cred](-.3,0.8)to [out=down,in=135] (1.3,-.7);
\draw(1.3,-1) node {$\scriptstyle \red u$};
\draw[-,line width=1.5](-1,2.5) to (-.3,1.5);
\draw[-,line width=1.5] (-.3,1.5) to (0.2,1);
\draw[-,line width=1.5] (-.3,2.2) to (0,2.5);
\draw[-,line width=1.5](-1,1.5) to (-.3,2.2);
\draw[-,line width=1.2pt]  (0.2,1)to(1.7,0);
\draw[-,line width=1pt](.8,0)to (.8,2.5);
\draw[-,line width=1.2pt](.8,2.5) to [out=up,in =right] (.5,3);
\draw[-,line width=1.2pt](.5,3) to [out=180,in =50] (0,2.5);
\draw[-,line width=1.2pt](.5,3)to (.5,3.5);
\end{tikzpicture}
\\
=\, & a~
\begin{tikzpicture}[baseline = 7.5pt, scale=0.35,color=\clr]
\draw[-,line width=1pt,color=\cred](.8,2) to[out=down,in=90] (1.5,1) to [out=down,in=up] (.8,-.4);
\draw(.8,-.6) node {$\scriptstyle \red u$};
\draw[-,line width=1.2pt] (0,-.35) to (1.5,2);
\draw[-,line width=1.2pt](0,2) to  (1.5,-.35);
\node at (0, -.6) {$\scriptstyle a$};
\node at (1.5, -.6) {$\scriptstyle b$};
\end{tikzpicture}
~+~
\begin{tikzpicture}[baseline = 7.5pt, scale=0.35, color=\clr]
\draw[-, line width=1pt] (0,-.3) to (0,2.2);
\draw[-, line width=1.2pt](0,-.3) to (0,.4);
\draw[-, line width=1.2pt] (1.8,-.2) to (1.8,.5);
\draw[-, line width=1.2pt] (1.8,1.8) to (1.8,2.2);
\draw[-, line width=1.2pt]   (0,.2) to  (1.8,1.8);
\draw[-, line width=1.2pt]   (0,1.6) to(1.8,0);
\draw[-,line width=1pt,color=\cred](.8,2) to[out=down,in=90] (1.5,1) to [out=down,in=up] (.8,-.4);
\draw(.8,-.6) node {$\scriptstyle \red u$};
\draw[-, line width=1.2pt]   (0,1.8) to (0,2.2); 
\draw[-, line width=1pt] (1.8,2) to (1.8,-.2); 
\node at (0, -.5) {$\scriptstyle a$};
\node at (2, -.5) {$\scriptstyle b$};
\node at (2, 2.5) {$\scriptstyle a$};
\node at (0, 2.5) {$\scriptstyle b$};
\node at (2,1.2) {$\scriptstyle 1$};
\end{tikzpicture}
        +
        ~\sum_{1\le t\le a-1}t!
        ~\begin{tikzpicture}[baseline = 7.5pt, scale=0.35, color=\clr]
\draw[-,line width=1.2pt] (0,-1) to (0,-.5);
\draw[-,line width=1.2pt](0,-.5) to [out=left,in =down] (-1,1);
\draw[-,line width=1pt](0,-.5) to [out=right,in =down] (.5,1.5);
\draw(0.5,-.6) node{$\scriptstyle1$};
\draw[-,line width=1pt,color=\cred](-.4,3)to (0.9,.8);
\draw[-,line width=1pt,color=\cred](0.9,.8) to (.9,-.5);
\draw(1,-.7) node {$\scriptstyle \red u$};
\draw[-,line width=1.2pt] (-1,1) to node[left]{$\scriptstyle t$} (-1,3);
\draw[-,line width=1.2pt] (.5,2.5) to node[right] {$\scriptstyle t$}  (1.6,-.25);
\draw[-,line width=1.5](-1,2.5) to (1,.3);
\draw[-,line width=1.5](-1,1.5) to (-.3,2.5);
\draw[-,line width=1.2pt]  (1,.3)to(1.7,-.3);
\draw[-,line width=1pt](.5,1.5)to (.5,2.5);
\draw[-,line width=1.2pt](.5,2.5) to [out=up,in =right] (.2,3);
\draw[-,line width=1.2pt](.2,3) to [out=180,in =50] (-0.3,2.5);
\draw[-,line width=1.2pt](.2,3)to (.2,3.5);
\end{tikzpicture}~
\\
& \qquad 
+ \sum_{1\le t\le a-1}t!
~\begin{tikzpicture}[baseline = 7.5pt, scale=0.35, color=\clr]
\draw[-,line width=1.2pt] (-.5,-1) to (-.5,-.5);
\draw[-,line width=1.2pt](-0.5,-.5) to [out=left,in =down] (-1,1);
\draw[-,line width=1pt](-.5,-.5) to [out=right,in =down]node[left]{$\scriptstyle1$} (-.2,1.2);
\draw[-,line width=1pt,color=\cred](-.2,3.2)to (0.2,-.6);
\draw(.2,-.8) node {$\scriptstyle \red u$};
\draw[-,line width=1.2pt,color=\cgreen] (.3,.3) to (.6,-.3);
\draw[-,line width=1.2pt] (-1,1) to node[left]{$\scriptstyle t$} (-1,3);
\draw[-,line width=1.2pt] (.5,2.5) to node[right] {$\scriptstyle t$}  (1.4,-.9);
\draw[-,line width=1.2pt](-1,2.5) to (.3,.3);
\draw[-,line width=1pt](.5,-.3)to  (.5,1.5);
\draw[-,line width=1.5](-1,1.5) to (-.3,2.5);
\draw[-,line width=1.2pt]  (.5,-.2)to(1.5,-1);
\draw[-,line width=1pt](.5,1.5)to (.5,2.5);
\draw[-,line width=1.2pt](.5,2.5) to [out=up,in =right] (.2,3);
\draw[-,line width=1.2pt](.2,3) to [out=180,in =50] (-0.3,2.5);
\draw[-,line width=1.2pt](.2,3)to (.2,3.5);
\end{tikzpicture} ~ \text{ by the result for $a=1$}
\\ 
\overset{\eqref{splitmerge},\eqref{sliders}}{\underset{\eqref{webassoc},\eqref{swallows},\eqref{adaptermovemerge}}{=}}
& a~
\begin{tikzpicture}[baseline = 7.5pt, scale=0.35,color=\clr]
\draw[-,line width=1pt,color=\cred](.8,2) to[out=down,in=90] (1.5,1) to [out=down,in=up] (.8,-.4);
\draw(.8,-.6) node {$\scriptstyle \red u$};
\draw[-,line width=1.2pt] (0,-.35) to (1.5,2);
\draw[-,line width=1.2pt](0,2) to  (1.5,-.35);
\node at (0, -.6) {$\scriptstyle a$};
\node at (1.5, -.6) {$\scriptstyle b$};
\end{tikzpicture}
~+~
\begin{tikzpicture}[baseline = 7.5pt, scale=0.35, color=\clr]
\draw[-, line width=1pt] (0,-.3) to (0,2.2);
\draw[-, line width=1.2pt](0,-.3) to (0,.4);
\draw[-, line width=1.2pt] (1.8,-.2) to (1.8,.5);
\draw[-, line width=1.2pt] (1.8,1.8) to (1.8,2.2);
\draw[-, line width=1.2pt]   (0,.2) to  (1.8,1.8);
\draw[-, line width=1.2pt]   (0,1.6) to(1.8,0);
\draw[-,line width=1pt,color=\cred](.8,2) to[out=down,in=90] (1.5,1) to [out=down,in=up] (.8,-.4);
\draw(.8,-.6) node {$\scriptstyle \red u$};
\draw[-, line width=1.2pt]   (0,1.8) to (0,2.2); 
\draw[-, line width=1pt] (1.8,2) to (1.8,-.2); 
\node at (0, -.5) {$\scriptstyle a$};
\node at (2, -.5) {$\scriptstyle b$};
\node at (2, 2.5) {$\scriptstyle a$};
\node at (0, 2.5) {$\scriptstyle b$};
\node at (2,1.2) {$\scriptstyle 1$};
\end{tikzpicture}
        +
        ~\sum_{1\le t\le a-1}(a-t)t!~   
\begin{tikzpicture}[baseline = 7.5pt, scale=0.35, color=\clr]
\draw[-, line width=1pt] (0,-.3) to (0,2.2);
\draw[-, line width=1.2pt](0,-.3) to (0,.4);
\draw[-, line width=1.2pt] (1.8,-.2) to (1.8,.5);
\draw[-, line width=1.2pt] (1.8,1.8) to (1.8,2.2);
\draw[-, line width=1.2pt]   (0,.2) to  (1.8,1.8);
\draw[-, line width=1.2pt]   (0,1.6) to(1.8,0);
\draw[-,line width=1pt,color=\cred](.8,2) to[out=down,in=90] (1.5,1) to [out=down,in=up] (.8,-.4);
\draw(.8,-.6) node {$\scriptstyle \red u$};
\draw[-, line width=1.2pt]   (0,1.8) to (0,2.2); 
\draw[-, line width=1pt] (1.8,2) to (1.8,-.2); 
\node at (0, -.5) {$\scriptstyle a$};
\node at (2, -.5) {$\scriptstyle b$};
\node at (2, 2.5) {$\scriptstyle a$};
\node at (0, 2.5) {$\scriptstyle b$};
\node at (2,1.2) {$\scriptstyle t$};
\end{tikzpicture}
\\
&\qquad\qquad
     + \sum_{1\le t\le a-1}(t+1)^2t!~
\begin{tikzpicture}[baseline = 7.5pt, scale=0.35, color=\clr]
\draw[-, line width=1pt] (0,-.3) to (0,2.2);
\draw[-, line width=1.2pt](0,-.3) to (0,.4);
\draw[-, line width=1.2pt] (1.8,-.2) to (1.8,.5);
\draw[-, line width=1.2pt] (1.8,1.8) to (1.8,2.2);
\draw[-, line width=1.2pt] (0,.2) to  (1.8,1.8);
\draw[-, line width=1.2pt] (0,1.6) to(1.8,0);
\draw[-,line width=1pt,color=\cred](.8,2) to[out=down,in=90] (1.5,1) to [out=down,in=up] (.8,-.4);
\draw(.8,-.6) node {$\scriptstyle \red u$};
\draw[-, line width=1.2pt]   (0,1.8) to (0,2.2); 
\draw[-, line width=1pt] (1.8,2) to (1.8,-.2); 
\node at (0, -.5) {$\scriptstyle a$};
\node at (2, -.5) {$\scriptstyle b$};
\node at (2, 2.5) {$\scriptstyle a$};
\node at (0, 2.5) {$\scriptstyle b$};
\node at (2.5,1.2) {${\scriptstyle t+1}$};
\end{tikzpicture}
        ,~ \\ 
=\, & a
\begin{tikzpicture}[baseline = 7.5pt, scale=0.35,color=\clr]
\draw[-,line width=1pt,color=\cred](.8,2) to[out=down,in=90] (1.5,1) to [out=down,in=up] (.8,-.4);
\draw(.8,-.6) node {$\scriptstyle \red u$};
\draw[-,line width=1.2pt] (0,-.35) to (1.5,2);
\draw[-,line width=1.2pt](0,2) to  (1.5,-.35);
\node at (0, -.6) {$\scriptstyle a$};
\node at (1.5, -.6) {$\scriptstyle b$};
\end{tikzpicture}
~
 +
~a\sum_{1\le t\le \min\{a,b\}} t! 
\begin{tikzpicture}[baseline = 7.5pt, scale=0.35, color=\clr]
\draw[-, line width=1pt] (0,-.3) to (0,2.2);
\draw[-, line width=1.2pt](0,-.3) to (0,.4);
\draw[-, line width=1.2pt] (1.8,-.2) to (1.8,.5);
\draw[-, line width=1.2pt] (1.8,1.8) to (1.8,2.2);
\draw[-, line width=1.2pt] (0,.2) to  (1.8,1.8);
\draw[-, line width=1.2pt] (0,1.6) to(1.8,0);
\draw[-,line width=1pt,color=\cred](.8,2) to[out=down,in=90] (1.5,1) to [out=down,in=45] (.8,-.4);
\draw(.8,-.6) node {$\scriptstyle \red u$};
\draw[-, line width=1.2pt]   (0,1.8) to (0,2.2); 
\draw[-, line width=1pt] (1.8,2) to (1.8,-.2); 
\node at (0, -.5) {$\scriptstyle a$};
\node at (2, -.5) {$\scriptstyle b$};
\node at (2, 2.5) {$\scriptstyle a$};
\node at (0, 2.5) {$\scriptstyle b$};                    
\node at (2,1.2) {$\scriptstyle t$};
\end{tikzpicture}.  
\end{align*}
Summarizing, we have proved that
\begin{align*}
       a~
\begin{tikzpicture}[baseline = 7.5pt, scale=0.35, color=\clr]
\draw[-,line width=1pt,color=\cred](.8,2) to[out=down,in=90] (0.2,1) to [out=down,in=up] (.8,-.4);
\draw(.8,-.6) node {$\scriptstyle \red u$};
\draw[-,line width=1.2pt] (0,-.35) to (1.5,2);
\draw[-,line width=1.2pt](0,2) to  (1.5,-.35);
\node at (0, -.6) {$\scriptstyle a$};
\node at (1.5, -.6) {$\scriptstyle b$};
\end{tikzpicture}
&= a
   ~
\begin{tikzpicture}[baseline = 7.5pt, scale=0.35,color=\clr]
\draw[-,line width=1pt,color=\cred](.8,2) to[out=down,in=90] (1.5,1) to [out=down,in=up] (.8,-.4);
\draw(.8,-.6) node {$\scriptstyle \red u$};
\draw[-,line width=1.2pt] (0,-.35) to (1.5,2);
\draw[-,line width=1.2pt](0,2) to  (1.5,-.35);
\node at (0, -.6) {$\scriptstyle a$};
\node at (1.5, -.6) {$\scriptstyle b$};
\end{tikzpicture}
~
 +
~a\sum_{1\le t\le \min\{a,b\}} t! 
\begin{tikzpicture}[baseline = 7.5pt, scale=0.35, color=\clr]
\draw[-, line width=1pt] (0,-.3) to (0,2.2);
\draw[-, line width=1.2pt](0,-.3) to (0,.4);
\draw[-, line width=1.2pt] (1.8,-.2) to (1.8,.5);
\draw[-, line width=1.2pt] (1.8,1.8) to (1.8,2.2);
\draw[-, line width=1.2pt] (0,.2) to  (1.8,1.8);
\draw[-, line width=1.2pt] (0,1.6) to(1.8,0);
\draw[-,line width=1pt,color=\cred](.8,2) to[out=down,in=90] (1.5,1) to [out=down,in=45] (.8,-.4);
\draw(.8,-.6) node {$\scriptstyle \red u$};
\draw[-, line width=1.2pt]   (0,1.8) to (0,2.2); 
\draw[-, line width=1pt] (1.8,2) to (1.8,-.2); 
\node at (0, -.5) {$\scriptstyle a$};
\node at (2, -.5) {$\scriptstyle b$};
\node at (2, 2.5) {$\scriptstyle a$};
\node at (0, 2.5) {$\scriptstyle b$};                    
\node at (2,1.2) {$\scriptstyle t$};
\end{tikzpicture}.  
\end{align*}
Canceling the nonzero scalar $a$ on both sides proves the relation \eqref{adaptrermovecross}.
\end{proof}
The proof of Theorem \ref{thm:ASchurisom} is completed. 

\begin{rem} 
Over a field $\kk$ of characteristic zero, the cyclotomic Schur category $\Sch_\bfu$ (cf. Definition~\ref{def-of-cyc-Schur}) can be viewed  as the quotient of the $\kk$-linear category $\ASchC$ in Definition \ref{def-affine-SchurC} by the relations in \eqref{cyclotomicpoly} and additional relations  $1_\lambda=0$, for all $\lambda\notin \Lambda_{\text{st}}^{1+\ell}$. This follows by the isomorphism $\ASchC \cong \ASch$ in Theorem ~\ref{thm:ASchurisom}. 
\end{rem}

\vspace{2mm}

\noindent {\bf Funding and Competing Interests.}
%{\bf Acknowledgement.} 

LS is partially supported by NSFC (Grant No. 12071346) and Natural Science foundation of Shanghai (Grant No. 25ZR1401352), and he thanks Institute of Mathematical Science and Department of Mathematics at University of Virginia for their hospitality and support. WW is partially supported by DMS--2401351. We thank Holden Eriksson for pointing out some crucial typos. We thank 2 anonymous referees for helpful comments and corrections. 

The authors have no competing interests to declare that are relevant to the content
of this article. The authors declare that the data supporting the findings of this study
are available within the paper.